\documentclass[twoside]{irmaems}
\usepackage{amssymb} 
\usepackage{amsmath} 
\usepackage{amscd}
\usepackage{latexsym}
\usepackage{color}

\renewcommand\theequation{\arabic{section}.\arabic{equation}}

\theoremstyle{definition} 

 \newtheorem{definition}{Definition}[section]
 \newtheorem{remark}[definition]{Remark}

\theoremstyle{plain}      

 \newtheorem{proposition}[definition]{Proposition}
 \newtheorem{theorem}[definition]{Theorem}
 \newtheorem{corollary}[definition]{Corollary}
 \newtheorem{lemma}[definition]{Lemma}

\newtheorem*{conjecture}{Conjecture}
\newtheorem*{question}{Question}

\newcommand{\Hom}{\operatorname{Hom}}

\newcommand{\abel}{{\rm abel}}
\newcommand{\HZ}{{H_\mathbb{Z}}}

\begin{document}

\title{The Goldman-Turaev Lie bialgebra and the Johnson homomorphisms}

\author{Nariya Kawazumi\thanks{
partially supported by the Grant-in-Aid for
Scientific Research (B) (No.\ 24340010) from the
Japan Society for Promotion of Sciences}
and  Yusuke Kuno\thanks{
partially supported by the Grant-in-Aid for
Research Activity Start-up (No.\ 24840038) from the
Japan Society for Promotion of Sciences}}

\address{
Department of Mathematical Sciences, University of Tokyo\\
3-8-1 Komaba Meguro-ku, Tokyo 153-8914 JAPAN\\
email:\,\tt{kawazumi@ms.u-tokyo.ac.jp}
\\[4pt]
Department of Mathematics,
Tsuda College,\\
2-1-1, Tsuda-Machi, Kodaira-shi,
Tokyo 187-8577 JAPAN \\
email:\,\tt{kunotti@tsuda.ac.jp}
}

\maketitle

\begin{abstract}
We survey a geometric approach to the Johnson homomorphisms
using the Goldman-Turaev Lie bialgebra.
\end{abstract}



\tableofcontents

\section{Introduction}
The purpose of this chapter is to survey a new aspect 
of topological studies of Riemann surfaces via infinite dimensional Lie algebras.
More concretely, we discuss a geometric approach to the Torelli-Johnson-Morita theory
using an infinite dimensional Lie algebra called the Goldman Lie algebra,
which comes from global structure of surfaces. 

The Torelli-Johnson-Morita theory, initiated by Johnson \cite{Joh80} \cite{Joh83}
and elaborated later by Morita \cite{Mor93}, is a place where infinite dimensional
Lie algebras appear in the study of the mapping class groups and the Torelli groups.
In this chapter, surfaces are always assumed to be differentiable.
Let $\Sigma=\Sigma_{g,1}$ be a compact oriented surface of genus $g>0$ 
with one boundary component, and $\mathcal{M}_{g,1}$ the mapping class group of
$\Sigma$ relative to the boundary. The Torelli group $\mathcal{I}_{g,1}$
is a normal subgroup of $\mathcal{M}_{g,1}$ consisting of mapping classes
acting trivially on the homology of $\Sigma$.
There are many motivations from various fields of mathematics for studying this group,
but still we are very far from full understanding of it.
To study $\mathcal{I}_{g,1}$ we consider a central filtration of $\mathcal{I}_{g,1}$
called the Johnson filtration, which is defined by the action of mapping classes on
the lower central series of the fundamental group of the surface.
A central object of the theory is an injective, graded Lie algebra homomorphism
\begin{equation}
\label{s0:tau}
\tau\colon \bigoplus_{k=1}^{\infty} {\rm gr}^k(\mathcal{I}_{g,1})
\to \bigoplus_{k=1}^{\infty} \mathfrak{h}_{g,1}^{\mathbb{Z}}(k).
\end{equation}
The $k$-th component $\tau_k\colon {\rm gr}^k(\mathcal{I}_{g,1}) \to \mathfrak{h}_{g,1}^{\mathbb{Z}}(k)$
is called the $k$-th Johnson homomorphism.
Here the target is an infinite dimensional Lie algebra called
the Lie algebra of symplectic derivations of type ``Lie" in the sense of Kontsevich \cite{Kon93}.
It is purely algebraically defined and was introduced by Kontsevich \cite{Kon93}
and Morita \cite{Mor89} \cite{MorICM} independently.
The image of $\tau$ is called the Johnson image.
Characterization of it is a hard but very important problem in the study of $\mathcal{I}_{g,1}$.
Morally, the problem asks what the Lie algebra of the Torelli group is.

An infinite dimensional Lie algebra related to an oriented surface $S$ also arises in the following way.
Let $\hat{\pi}(S)$ be the set of homotopy classes of oriented loops on $S$.
Motivated by the study of the symplectic structure of the moduli space of flat
bundles over a surface, Goldman \cite{Go86} introduced a Lie bracket called
the Goldman bracket on the free $\mathbb{Z}$-module $\mathbb{Z}\hat{\pi}(S)$
with basis the set $\hat{\pi}(S)$.
The definition of the Lie bracket involves the intersections of two loops and
this Lie algebra is called the Goldman Lie algebra.
Later Turaev \cite{Tu91} found a Lie cobracket called the Turaev cobracket on the
quotient Lie algebra $\mathbb{Z}\hat{\pi}^{\prime}(S)=\mathbb{Z}\hat{\pi}(S)/\mathbb{Z}1$,
where $1$ is the class of a constant loop, and showed that $\mathbb{Z}\hat{\pi}^{\prime}(S)$
has a structure of a Lie bialgebra. The definition of the Lie cobracket involves
the self-intersections of a loop. This Lie bialgebra is called
the Goldman-Turaev Lie bialgebra and will be our central object of consideration.
In this chapter we consider over the rationals $\mathbb{Q}$
and work with $\mathbb{Q}\hat{\pi}(S)=\mathbb{Z}\hat{\pi}(S)\otimes_{\mathbb{Z}}\mathbb{Q}$.

Our primary goal is to show that the Goldman Lie algebra appears
naturally in the Torelli-Johnson-Morita theory.
This relation is first found in \cite{KK1} and further developed in \cite{KK3} \cite{KK4}.
Here we explain the main idea briefly. To start with, the mapping class group $\mathcal{M}_{g,1}$ acts on
the fundamental group of $\Sigma$ by automorphisms, where we take a base point on the boundary of $\Sigma$.
This is a point of view illustrated in the Dehn-Nielsen theorem. A basic observation is that the
Goldman Lie algebra $\mathbb{Q}\hat{\pi}(\Sigma)$
also acts on the fundamental group, but this time the action
is by derivations. As for the relationship between Lie algebras and Lie groups,
derivations and automorphisms are related by the exponential map. Here we come to
a technical but inevitable point; since we work with the exponential map,
we have to consider completions of objects and care about convergence.
In any case we can construct suitable completions of the Goldman Lie algebra and the fundamental group,
and we have the exponential map from a subset of the completion of $\mathbb{Q}\hat{\pi}(\Sigma)$
to the automorphism group of the completion of the fundamental group.
After that we introduce a Lie subalgebra $L^+(\Sigma)$ of the completion of $\mathbb{Q}\hat{\pi}(\Sigma)$,
and show that the automorphisms of the completion of the fundamental group of $\Sigma$ induced by elements
of $\mathcal{I}_{g,1}$ are in the image of the derivations coming from $L^+(\Sigma)$ by the exponential map.
Taking the logarithm, we obtain an injective group homomorphism
\begin{equation}
\label{s0:geo}
\tau\colon \mathcal{I}_{g,1}\to L^+(\Sigma),
\end{equation}
where the group structure of the target is described by the Hausdorff series.
We call (\ref{s0:geo}) the geometric Johnson homomorphism,
since taking the graded quotients of it we can recover (\ref{s0:tau}).
Actually, $\tau$ is essentially the same as Massuyeau's total Johnson map
\cite{Mas12}. However, our construction is free from any choice and is more intrinsic.
A practical advantage of our construction is that it can be applied to other compact surfaces
with more than one boundary component.

The Turaev cobracket induces a map $\delta$ from the target of the geometric Johnson homomorphism $\tau$.
By the fact that any diffeomorphism of a surface preserves the self-intersections
of loops on the surface, we show that $\delta\circ \tau=0$.
This gives a non-trivial geometric constraint on the Johnson image.
Indeed, we show that all the Morita traces \cite{Mor93}, which are
obstructions of the surjectivity of (\ref{s0:tau}) found first, can be derived from $\delta$.
We hope that some of other known constraints on the Johnson image
also can admit interpretations from our geometric context.

This survey is organized as follows.
In \S \ref{sec:Cla}, we give an overview of the construction of the Johnson homomorphisms
and known results about the Johnson image. We also discuss how to extend the Johnson homomorphisms
to the Torelli group or to the whole mapping class group.
The main body of this chapter is from \S \ref{sec:DN} to \S \ref{sec:Joh}.
We always consider the mapping class groups relative to the boundary; all the diffeomorphisms and
the isotopies that we consider are required to fix the boundary pointwise.
Therefore, when the surface $S$ has more than one boundary component, it is natural
to consider that the mapping class group acts on the fundamental groupoid of the surface with base points
chosen from each boundary component, instead of the fundamental group.
In \S \ref{sec:DN}, we provide some languages to deal with such a situation.
In \S \ref{sec:Ope}, we give the definition of the Goldman-Turaev Lie bialgebra,
and explain how it interacts with the homotopy set of based paths on the surface.
In particular, we show that $\mathbb{Q}\hat{\pi}(S)$ acts on the fundamental groupoid of $S$ by derivations.
We also discuss other operations to curves on surfaces.
In \S \ref{sec:DT}, we investigate Dehn twists from our point of view in detail.
We show that the action of a Dehn twist on the completion of the fundamental groupoid
has the canonical logarithm, and specify it as an element of the completion of
the Goldman Lie algebra. This was first observed in \cite{KK1}, and
leads us to introducing a ``generalized Dehn twist",
which is an automorphism of the completion of the fundamental groupoid
associated to a loop on the surface which is not necessarily simple.
Generalized Dehn twists are first introduced in \cite{Ku11} and are further studied in \cite{KK3} \cite{KK4}.
Massuyeau and Turaev \cite{MT11} also study them from a slight different point of view.
In \S \ref{subsec:GDT} and \S \ref{subsec:int} we present basic properties of generalized Dehn twists.
In \S \ref{sec:Revi} and \S \ref{sec:Joh}, we define the geometric Johnson homomorphism.
We first treat the case $S=\Sigma$ in \S \ref{sec:Revi}, then the general case in \S \ref{sec:Joh}.
In the general case, we obtain an injective group homomorphism
\begin{equation}
\label{s0:gengeo}
\tau\colon \mathcal{I}^L(S)\to L^+(S).
\end{equation}
Here $\mathcal{I}^L(S)$ is the ``largest" Torelli group in the sense of Putman \cite{Pu},
and $L^+(S)$ is a Lie subalgebra of the completion of $\mathbb{Q}\hat{\pi}(S)$.
Note that when $S$ has more than one boundary component, there are natural choices
for the Torelli group, see \cite{Pu}. Recently Church \cite{Chu11} constructed the first Johnson homomorphism
for all kinds of Putman's Torelli groups. We do not know any relation between
Church's construction and ours.
In \S \ref{sec:Revi} and \S \ref{sec:Joh}, we also give an algebraic description of the Goldman bracket.
When $S=\Sigma$, this means that the completion of $\mathbb{Q}\hat{\pi}(S)$
is isomorphic to (the degree completion of an enhancement of)
the Lie algebra of symplectic derivations of type ``associative" in the sense of Kontsevich \cite{Kon93}.
Note that the tensorial description of the Goldman bracket is also obtained by
Massuyeau and Turaev \cite{MT11} \cite{MT12} by a different approach.
In the case $S=\Sigma$, we also mention a partial result
about a tensorial description of the Turaev cobracket based on a result of \cite{MT11}.
In \S \ref{sec:Oth}, we discuss other related topics.

Finally, we make a remark that is less relevant to the main part of the text but is still worth mentioning.
As for an infinite dimensional Lie algebra coming from local structures of surfaces, there is an observation
due to Kontsevich \cite{Ko87} and Beilinson, Manin and Schechtman \cite{BMS}.
They discovered that the Lie algebra of germs of meromorphic vector fields
at the origin of $\mathbb{C}$, i.e., a complex analytic version of
the Lie algebra ${\rm Vect}(S^1)$, acts on the moduli space of compact
Riemann surfaces with local coordinates in an infinitesimally transitive way.
It enables us to regard the Lie algebra as the Lie algebra of the stable mapping class group.
This idea has been well-understood for the last few decades.
For example, from this fact, we can derive some topological information on the stable cohomology
of the mapping class group of a surface.
For details, see \cite{ADKP}, \cite{Ka93}, \cite{Ka96} and \cite{Ka98}.
Compared with this idea, our approach, which will be presented in this chapter,
suggests us a quite new interaction between
the mapping class group and an infinite dimensional Lie algebra
coming from the global nature of surfaces.

\section{Classical Torelli-Johnson-Morita theory}
\label{sec:Cla}

We describe the Torelli-Johnson-Morita theory for once bordered surface
and its recent developments. This theory, initiated by Johnson \cite{Joh80} \cite{Joh83}
and elaborated later by Morita \cite{Mor93}, studies a certain filtration of the mapping
class group and a graded Lie algebra associated to it.
The treatment here is brief and limited.
In particular, we confine ourselves to compact surfaces with one boundary component.
For more details and other aspects, we refer to the chapters of Habiro and Massuyeau \cite{HM12},
Morita \cite{Mor07}, Sakasai \cite{Sak12} and Satoh \cite{Sat12}.

\subsection{Lower central series and the higher Torelli groups}
\label{subsec:LT}

Let $\Sigma=\Sigma_{g,1}$ be a compact connected oriented surface of
genus $g>0$ with one boundary component, and $\mathcal{M}_{g,1}$ the {\it mapping class 
group} of $\Sigma$ relative to the boundary, i.e., the group of diffeomorphisms of
$\Sigma$ fixing the boundary $\partial \Sigma$ pointwise, modulo isotopies fixing $\partial \Sigma$ pointwise.
Taking a base point $*$ on $\partial \Sigma$, we denote $\pi=\pi_1(\Sigma,*)$.
The group $\mathcal{M}_{g,1}$ acts naturally on $\pi$. Let $\Gamma_k=\Gamma_k(\pi)$, $k\ge 1$,
be the {\it lower central series} of $\pi$, i.e., a series of normal subgroups of $\pi$ successively
defined by $\Gamma_1=\pi$ and $\Gamma_k=[\Gamma_{k-1},\pi]$ for $k\ge 2$.
The intersection $\bigcap_{k=1}^{\infty}\Gamma_k$ is trivial since $\pi$ is a free group.
Since $\Gamma_k$ is characteristic, $\mathcal{M}_{g,1}$ acts naturally on the $k$-th
{\it nilpotent quotient} $N_k=N_k(\pi)=\pi/\Gamma_{k+1}$.

For $k\ge 1$, the {\it $k$-th Torelli group} is defined as
$$\mathcal{M}_{g,1}(k)=\{ \varphi\in \mathcal{M}_{g,1}| \varphi {\rm \ acts\ trivially\ on\ }
N_k \}.$$
Then we obtain a decreasing filtration $\{ \mathcal{M}_{g,1}(k)\}_{k=1}^{\infty}$ of normal subgroups of $\mathcal{M}_{g,1}$
called the {\it Johnson filtration}. The first term $\mathcal{M}_{g,1}(1)$ is nothing but
the {\it Torelli group} $\mathcal{I}_{g,1}$ since $N_1=\pi/[\pi,\pi]$ is canonically isomorphic
to the first homology group $H_{\mathbb{Z}}=H_1(\Sigma;\mathbb{Z})$.
The second term $\mathcal{M}_{g,1}(2)$ is known as the {\it Johnson kernel} $\mathcal{K}_{g,1}$, which is
by definition the kernel of the first Johnson homomorphism $\tau_1$ (see \S \ref{subsec:J-i}).
Due to a deep result by Johnson \cite{Joh85},
$\mathcal{K}_{g,1}$ is equal to the group generated by Dehn twists along separating
simple closed curves on $\Sigma$.

It is known that the filtration $\{ \mathcal{M}_{g,1}(k)\}_{k=1}^{\infty}$ is central, i.e.,
\begin{equation}
\label{eq:J-cent}
[\mathcal{M}_{g,1}(k),\mathcal{M}_{g,1}(\ell)]\subset \mathcal{M}_{g,1}(k+\ell)
\quad {\rm for\ } k,\ell \ge 1
\end{equation}
(see \cite{Mor91} Corollary 3.3).
Thus commutator product induces a structure of a graded Lie algebra on
the graded module $\bigoplus_{k=1}^{\infty} {\rm gr}^k(\mathcal{I}_{g,1})$,
where ${\rm gr}^k(\mathcal{I}_{g,1})=\mathcal{M}_{g,1}(k)/\mathcal{M}_{g,1}(k+1)$.
On the other hand, the intersection $\bigcap_{k=1}^{\infty} \mathcal{M}_{g,1}(k)$ is trivial
since $\bigcap_{k=1}^{\infty}\Gamma_k=\{ 1\}$. We can regard
the quotient groups $\mathcal{M}_{g,1}/\mathcal{M}_{g,1}(k)$ and
$\mathcal{I}_{g,1}/\mathcal{M}_{g,1}(k)$ as approximations of the whole group $\mathcal{M}_{g,1}$
and the Torelli group $\mathcal{I}_{g,1}$. From this point of view it is
important to understand ${\rm gr}^k(\mathcal{I}_{g,1})$ for a specific $k$
or the whole graded Lie algebra $\bigoplus_{k=1}^{\infty} {\rm gr}^k(\mathcal{I}_{g,1})$.
The Johnson homomorphisms are key tool to study them.

\subsection{The Johnson homomorphisms and their images}
\label{subsec:J-i}
We briefly recall the definition of the Johnson homomorphisms.
Let us fix $k\ge 1$ and consider a $\mathcal{M}_{g,1}$-equivariant exact sequence
$0\to \Gamma_{k+1}/\Gamma_{k+2}\to N_{k+1}\to N_k \to 1$.
Since $\pi$ is free, the quotient
$\Gamma_k/\Gamma_{k+1}$ is canonically isomorphic to $\mathcal{L}_{\mathbb{Z}}(k)$,
the degree $k$-part of the free Lie algebra generated by $N_1=H_{\mathbb{Z}}$
(see e.g., \cite{MKS} \cite{SerLNM}).
Thus the exact sequence becomes a central extension
\begin{equation}
\label{eq:central}
0\to \mathcal{L}_{\mathbb{Z}}(k+1) \to N_{k+1} \to N_k\to 1.
\end{equation}
Take $\varphi \in \mathcal{M}_{g,1}(k)$. Since $\varphi$ acts trivially on $N_k$,
for any $x\in \pi$ the image of $\varphi(x)x^{-1}$ in $N_{k+1}$ is actually
an element of $\mathcal{L}_{\mathbb{Z}}(k+1)$ in view of (\ref{eq:central}).
Then we obtain a mapping $\pi \to \mathcal{L}_{\mathbb{Z}}(k+1)$, $x\mapsto [\varphi(x)x^{-1}]$.
One can show that this mapping is a homomorphism, thus induces a homomorphism
$\tau_k(\varphi)\colon H_{\mathbb{Z}}\to \mathcal{L}_{\mathbb{Z}}(k+1)$.
The mapping $\tau_k\colon \mathcal{M}_{g,1}(k)
\to {\rm Hom}(H_{\mathbb{Z}},\mathcal{L}_{\mathbb{Z}}(k+1))$, $\varphi\mapsto \tau_k(\varphi)$
is in fact a homomorphism, and is called the {\it $k$-th Johnson homomorphism}.
It was introduced by Johnson \cite{Joh80} \cite{Joh83}.
Note that using the intersection form $(\ \cdot \ )\colon H_{\mathbb{Z}}\times H_{\mathbb{Z}}\to \mathbb{Z}$
on the surface, we can identify $H_{\mathbb{Z}}$ and its dual
$H_{\mathbb{Z}}^*={\rm Hom}(H_{\mathbb{Z}},\mathbb{Z})$ by
$H_{\mathbb{Z}}\to H_{\mathbb{Z}}^*$, $X\mapsto (Y\mapsto (Y\cdot X))$,
where $X,Y\in H_{\mathbb{Z}}$. This induces an isomorphism
$${\rm Hom}(H_{\mathbb{Z}},\mathcal{L}_{\mathbb{Z}}(k+1))
=H_{\mathbb{Z}}^* \otimes \mathcal{L}_{\mathbb{Z}}(k+1)
\cong H_{\mathbb{Z}} \otimes \mathcal{L}_{\mathbb{Z}}(k+1),$$
through which we can also write $\tau_k$ as
\begin{equation}
\label{eq:tau}
\tau_k\colon \mathcal{M}_{g,1}(k) \to H_{\mathbb{Z}} \otimes \mathcal{L}_{\mathbb{Z}}(k+1).
\end{equation}
One can easily see that the kernel of $\tau_k$ is $\mathcal{M}_{g,1}(k+1)$, hence
$\tau_k$ induces an injective group homomorphism
\begin{equation}
\label{eq:gr(tau)}
\tau_k\colon {\rm gr}^k(\mathcal{I}_{g,1}) \hookrightarrow H_{\mathbb{Z}} \otimes \mathcal{L}_{\mathbb{Z}}(k+1)
\end{equation}
(using the same letter $\tau_k$). In particular the graded quotient ${\rm gr}^k(\mathcal{I}_{g,1})$
is isomorphic to ${\rm Im}(\tau_k)$.

\begin{remark}
\label{rem:sp-eq}
Let $Sp(H_{\mathbb{Z}})$ be the group of $\mathbb{Z}$-linear automorphisms of $H_{\mathbb{Z}}$
preserving the intersection form. Fixing a symplectic basis of $H_{\mathbb{Z}}$, we have
an isomorphism $Sp(H_{\mathbb{Z}})\cong Sp(2g;\mathbb{Z})$.
The group $Sp(H_\mathbb{Z})$ acts on both the domain and the target of (\ref{eq:gr(tau)}).
First of all for each $k\ge 1$ the group $\mathcal{M}_{g,1}$ acts on $\mathcal{M}_{g,1}(k)$
by conjugation, hence on ${\rm gr}^k(\mathcal{I}_{g,1})$. From (\ref{eq:J-cent}) we see that the subgroup
$\mathcal{M}_{g,1}(1)=\mathcal{I}_{g,1}$ acts trivially on ${\rm gr}^k(\mathcal{I}_{g,1})$.
Since we have an exact sequence $1\to \mathcal{I}_{g,1}\to \mathcal{M}_{g,1}\to Sp(H_\mathbb{Z}) \to 1$,
the action of $Sp(H_\mathbb{Z})$ on the domain is induced.
The action of $Sp(H_\mathbb{Z})$ on the target is naturally induced by the action
of $Sp(H_\mathbb{Z})$ on $H_{\mathbb{Z}}$. Then one can see that the map
(\ref{eq:gr(tau)}) is $Sp(H_\mathbb{Z})$-equivariant.
This point of view is particularly important when we study $\tau_k \otimes_{\mathbb{Z}} \mathbb{Q}$,
since we can apply representation theory of $Sp(2g;\mathbb{Q})$.
\end{remark}

Johnson \cite{Joh80} proved $\tau_1(\mathcal{I}_{g,1})=\Lambda_{\mathbb{Z}}^3 H_{\mathbb{Z}}
\subsetneq H_{\mathbb{Z}}\otimes \mathcal{L}_{\mathbb{Z}}(2)$.
Morita \cite{Mor93} found that the target of $\tau_k$ can be smaller and
the collection $\{ \tau_k\}_{k=1}^{\infty}$ constitutes a graded Lie algebra homomorphism.
He introduced a submodule $\mathfrak{h}_{g,1}^{\mathbb{Z}}(k)\subset
H_{\mathbb{Z}} \otimes \mathcal{L}_{\mathbb{Z}}(k+1)$ defined by
$$\mathfrak{h}_{g,1}^{\mathbb{Z}}(k)={\rm Ker}([\ ,\ ]\colon H_{\mathbb{Z}}
\otimes \mathcal{L}_{\mathbb{Z}}(k+1) \to \mathcal{L}_{\mathbb{Z}}(k+2)).$$
When $k=1$, we have $\mathfrak{h}^{\mathbb{Z}}_{g,1}(1)=\Lambda_{\mathbb{Z}}^3 H_{\mathbb{Z}}$.
Let $\mathcal{L}_{\mathbb{Z}}=\bigoplus_{k=1}^{\infty}\mathcal{L}_{\mathbb{Z}}(k)$
be the free Lie algebra generated by $H_{\mathbb{Z}}$.
Any element of $\mathfrak{h}_{g,1}^{\mathbb{Z}}(k)$ can be considered as
a symplectic derivation of $\mathcal{L}_{\mathbb{Z}}$ as follows.
For $u\in \mathfrak{h}_{g,1}^{\mathbb{Z}}(k)$, we define a $\mathbb{Z}$-linear
map $D_u\colon H_{\mathbb{Z}}=\mathcal{L}_{\mathbb{Z}}(1)\to \mathcal{L}_{\mathbb{Z}}(k+1)$ by
$D_u(X)=C_{12}(X\otimes u)$, where $C_{12}\colon H^{\otimes k+3}\to H^{\otimes k+1}$,
$X_1\otimes X_2\otimes X_3 \otimes \cdots \otimes X_{k+3}\mapsto
(X_1\cdot X_2)X_3\otimes \cdots \otimes X_{k+3}$ is
the contraction of the first and the second factor by the intersection form. Then we can extend $D_u$
uniquely to a derivation $D_u\colon \mathcal{L}_{\mathbb{Z}}\to \mathcal{L}_{\mathbb{Z}}$
(using the same letter), that is, a $\mathbb{Z}$-linear map satisfying
the Leibniz rule $D_u([v,w])=[D_u(v),w]+[v,D_u(w)]$ for any $v,w\in \mathcal{L}_{\mathbb{Z}}$. 
The derivation $D_u$ is {\it of degree $k$} in the sense that
$D_u(\mathcal{L}_{\mathbb{Z}}(\ell))\subset \mathcal{L}_{\mathbb{Z}}(k+\ell)$
for any $\ell\ge 1$, and is {\it symplectic} in the sense that
$D_u(\omega)=0$, where $\omega \in \mathcal{L}_{\mathbb{Z}}(2)=\Lambda^2 H_{\mathbb{Z}}
\subset H_{\mathbb{Z}}^{\otimes 2}$ is the tensor called the {\it symplectic form},
corresponding to $-1_H\in {\rm Hom}(H_{\mathbb{Z}},H_{\mathbb{Z}})=H_{\mathbb{Z}}^*\otimes H_{\mathbb{Z}}
=H_{\mathbb{Z}}\otimes H_{\mathbb{Z}}$. Note that if $\{ A_i,B_i\}_{i=1}^g \subset H_{\mathbb{Z}}$
is a symplectic basis, then $\omega=\sum_{i=1}^g A_i\otimes B_i-
B_i\otimes A_i$, cf. \S \ref{subsec:SD}.
The correspondence $u\mapsto D_u$ is injective.
On the other hand any symplectic derivation of
$\mathcal{L}_{\mathbb{Z}}$ of degree $k$ can be written as the form $D_u$ for some
$u\in \mathfrak{h}_{g,1}^{\mathbb{Z}}(k)$. Thus we can identify
$\mathfrak{h}_{g,1}^{\mathbb{Z}}(k)$ with the $\mathbb{Z}$-module of
symplectic derivations of $\mathcal{L}_{\mathbb{Z}}$ of degree $k$. Then the graded module
$\bigoplus_{k=1}^{\infty}\mathfrak{h}_{g,1}^{\mathbb{Z}}(k)$ is the $\mathbb{Z}$-module
of symplectic derivations of $\mathcal{L}_{\mathbb{Z}}$ and naturally has
a structure of a graded Lie algebra. We will discuss
more details of the Lie algebra of symplectic derivations in \S \ref{subsec:SD}.

\begin{theorem}[Morita \cite{Mor93}]
\label{thm:mor}
\begin{enumerate}
\item
The image of {\rm (\ref{eq:tau})} is contained in $\mathfrak{h}_{g,1}^{\mathbb{Z}}(k)$.
\item
The maps $\{ \tau_k \}_{k=1}^{\infty}$ induce an injective homomorphism
of graded Lie algebras
$$\tau \colon \bigoplus_{k=1}^{\infty} {\rm gr}^k(\mathcal{I}_{g,1}) \to
\bigoplus_{k=1}^{\infty}\mathfrak{h}_{g,1}^{\mathbb{Z}}(k).$$
\end{enumerate}
\end{theorem}

By the result of Morita, we can write $\tau_k$ as
\begin{equation}
\label{eq:tau-r}
\tau_k\colon \mathcal{M}_{g,1}(k)\to \mathfrak{h}_{g,1}^{\mathbb{Z}}(k).
\end{equation}
In this chapter, we understand the $k$-th Johnson homomorphism on the $k$-th
Torelli group to be (\ref{eq:tau-r}).

As posed in Morita \cite{Mor07}, characterization of 
$\bigoplus_{k=1}^{\infty} {\rm gr}^k(\mathcal{I}_{g,1})$ as a Lie subalgebra
of $\bigoplus_{k=1}^{\infty}\mathfrak{h}_{g,1}^{\mathbb{Z}}(k)$
is one of big and basic problems in the Torelli-Johnson-Morita theory.
Actually, $\tau_k$ is not surjective in general, which was observed first 
by Morita \cite{Mor93}. One often considers
the problem over $\mathbb{Q}$ to make use of representation theory of $Sp(2g;\mathbb{Q})$,
(see Remark \ref{rem:sp-eq}), but still it is very hard.
In this chapter we call the subalgebra 
$\bigoplus_{k=1}^{\infty} {\rm gr}^k(\mathcal{I}_{g,1})$ tensored by 
the rationals $\mathbb{Q}$ the {\it Johnson image}. 
In his monumental paper \cite{Hain97}, Hain gave an explicit presentation 
of the Malcev completion of the Torelli group $\mathcal{I}_{g,1}$ 
when $g \geq 6$. In particular, from the presentation together with his other result, Proposition 7.1 in \cite{Hain93}, 
the Johnson image is generated by 
the first degree component ${\rm gr}^1(\mathcal{I}_{g,1})\otimes\mathbb{Q} 
= \Lambda^3H_\mathbb{Z}\otimes\mathbb{Q}$. 
In other words, the comprehension of the Johnson image is completely 
determined by Hain. Hence what we want is a complete system of 
the defining equations of the Johnson image in the Lie algebra 
$\bigoplus_{k=1}^{\infty}\mathfrak{h}_{g,1}^{\mathbb{Z}}(k)\otimes\mathbb{Q}$. 
Such an equation is called a {\it Johnson cokernel} or 
an {\it obstruction of the surjectivity of the Johnson homomorphism}.\par

First of all, Morita \cite{Mor93} 
first found an obstruction for the surjectivity of $\tau_k$. 
Let $S^k H_{\mathbb{Z}}$ be the $k$-th symmetric power of $H_{\mathbb{Z}}$,
$C_{12}\colon H_{\mathbb{Z}}^{\otimes k+2}\to H_{\mathbb{Z}}^{\otimes k}$
the contraction of the first and the second factor,
and $s\colon H_{\mathbb{Z}}^{\otimes k} \to S^k H_{\mathbb{Z}}$ the natural projection.
Let ${\rm Tr}_k\colon \mathfrak{h}_{g,1}^{\mathbb{Z}}(k) \to S^k H_{\mathbb{Z}}$ be
a $\mathbb{Z}$-linear map defined by
$$
{\rm Tr}_k\colon \mathfrak{h}_{g,1}^{\mathbb{Z}}(k)\subset
H_{\mathbb{Z}} \otimes \mathcal{L}_{\mathbb{Z}}(k+1) \subset H_{\mathbb{Z}}^{\otimes k+2}
\overset{C_{12}}{\to} H_{\mathbb{Z}}^{\otimes k} \overset{s}{\to} S^k H_{\mathbb{Z}}.
$$
The map ${\rm Tr}_k$ is called the {\it $k$-th Morita trace}.

\begin{theorem}[Morita \cite{Mor93}]
\begin{enumerate}
\item If $k \geq 2$, we have ${\rm Tr}_k\circ \tau_k=0\colon 
{\rm gr}^k(\mathcal{I}_{g,1}) \to S^kH_{\mathbb{Z}}$.
\item If $k$ is odd, then ${\rm Tr}_k$ is non-trivial. In fact,
${\rm Tr}_k\otimes_{\mathbb{Z}} \mathbb{Q}$ is surjective.
\item If $k$ is even, then ${\rm Tr}_k = 0$ on the whole 
$\mathfrak{h}_{g,1}^{\mathbb{Z}}(k)$. 
\end{enumerate}
\end{theorem}

In the original definition \cite{Mor93} the map ${\rm Tr}$ is defined as a map
$\mathfrak{h}_{g,1}^{\mathbb{Z}}(k-1) \to S^{k-1} H_{\mathbb{Z}}$, while 
we follow the grading in \cite{Mor99}, p.376.
In \S \ref{subsec:Tc-Mo}, we will give a topological interpretation of the Morita traces
by the Turaev cobracket \cite{KK4}.

In a natural way the absolute Galois group $\operatorname{Gal}(\overline{\mathbb{Q}}/\mathbb{Q})$
of the rational number field $\mathbb{Q}$ acts on the arithmetic 
fundamental group of a pointed 
algebraic curve defined over the rationals $\mathbb{Q}$, which is 
a group extension of the Galois group $\operatorname{Gal}(\overline{\mathbb{Q}}/\mathbb{Q})$ by the (geometric) fundamental 
group of the curve. This induces an image of the Galois group 
in the Lie algebra $\bigoplus_{k=1}^{\infty}\mathfrak{h}_{g,1}^{\mathbb{Z}}(k)\otimes\mathbb{Q}$, 
which is called the {\it Galois image}. The origin of this constructtion is in 
Grothendieck, Ihara and Deligne.  For its precise description, 
see \cite{Naka} and references therein. 
The relation between the Johnson image and the Galois image 
has been studied by T.\ Oda, H.\ Nakamura, M.\ Matsumoto and others.
For example, H. Nakamura \cite{Naka} introduced some explicit Johnson 
cokernels coming from the Galois image. Such Johnson cokernels are 
called {\it Galois obstructions}. \par

In his study of the $IA$-automorphism group of a free group, 
Satoh \cite{Sat06} \cite{Sat12a} discovered a refinement 
of the Morita traces. Let $F_n$ be a free group of rank $n \geq 2$, 
and $H_\mathbb{Z}$ its abelianization as in \S\ref{subsec:J-i}. 
We denote by $H^*_{\mathbb{Z}}$ its dual $\Hom_\mathbb{Z}(H_\mathbb{Z}, 
\mathbb{Z})$, and by $\mathcal{L}_\mathbb{Z}(k)$ the degree $k$ part of 
the free Lie algebra generated by $H_\mathbb{Z}$. 
The cyclic group of degree $k$ acts on the tensor space 
${H_{\mathbb{Z}}}^{\otimes k}$ by cyclic permutation of the components. 
Following Satoh, we denote by $\mathcal{C}_n(k)$ the coinvariants of 
the action, i.e., 
$$
\mathcal{C}_n(k) := {H_{\mathbb{Z}}}^{\otimes k}/\langle
X_1\otimes X_2\otimes\cdots\otimes X_k - 
X_2\otimes X_3\otimes \cdots \otimes X_k\otimes X_1; \, 
X_i \in H_\mathbb{Z}\rangle.
$$
The Satoh trace $\widehat{\rm Tr}_k\colon H^*_{\mathbb{Z}}\otimes
\mathcal{L}_\mathbb{Z}(k+1) \to \mathcal{C}_n(k)$ is defined to be the 
composite of the inclusion $H^*_{\mathbb{Z}}\otimes
\mathcal{L}_\mathbb{Z}(k+1) \hookrightarrow H^*_{\mathbb{Z}}\otimes
{H_\mathbb{Z}}^{\otimes(k+1)}$, the contraction map $C'_{12}\colon 
H^*_{\mathbb{Z}}\otimes{H_\mathbb{Z}}^{\otimes(k+1)} \to 
{H_{\mathbb{Z}}}^{\otimes k}$, $f\otimes X_2\otimes X_3\otimes\cdots\otimes 
X_{k+2} \mapsto f(X_2)X_3\otimes\cdots\otimes X_{k+2}$, $(f \in H^*_{\mathbb{Z}}, X_i \in H_\mathbb{Z})$, 
and the quotient map
${H_{\mathbb{Z}}}^{\otimes k} \to \mathcal{C}_n(k)$.
Satoh \cite{Sat06} \cite{Sat12a} proved that the images of the lower central 
series of the $IA$-automorphism group under the Johnson homomorphisms
{\it stably} coincide with the kernels of the Satoh traces $\widehat{\rm Tr}_k$ 
up to torsion. For details, see his own chapter \cite{Sat12}. \par
The fundamental group $\pi_1(\Sigma_{g,1}, *)$ is free of rank $2g$, 
so that we can consider the Satoh traces $\widehat{\rm Tr}_k$ on the 
Lie algebra $\bigoplus_{k=1}^{\infty}\mathfrak{h}_{g,1}^{\mathbb{Z}}(k)$. Then the contraction map 
$C'_{12}$ is exactly the same as the map $C_{12}$ under the Poincar\'e 
duality. From Satoh's result \cite{Sat06} together with Hain's result \cite{Hain97}, 
we have $\widehat{\rm Tr}_k\circ\tau_k = 0$ on $\mathcal{M}_{g,1}(k)$ 
for any $k \geq 2$. 
Hence $\widehat{\rm Tr}_k$ is a refinement of the Morita trace ${\rm Tr}_k$. 
Enomoto and Satoh \cite{ES} carried out some explicit computation of 
$\widehat{\rm Tr}_k$'s on $\bigoplus_{k=1}^{\infty}\mathfrak{h}_{g,1}^{\mathbb{Z}}(k)\otimes\mathbb{Q}$, 
to prove that they have many non-trivial components of the Johnson cokernels
other than the Morita traces. Thus the restriction of $\widehat{\rm Tr}_k$ 
to $\bigoplus_{k=1}^{\infty}\mathfrak{h}_{g,1}^{\mathbb{Z}}(k)$ is called the {\it Enomoto-Satoh trace}. 
\par

\subsection{Extensions of the Johnson homomorphisms}
\label{subsec:Ext}

From Theorem \ref{thm:mor} (2) by Morita, the totality of the Johnson 
homomorphisms (tensored by the rationals $\mathbb{Q}$)
$$
\tau \colon \bigoplus_{k=1}^{\infty} {\rm gr}^k(\mathcal{I}_{g,1})
\otimes\mathbb{Q} \to
\bigoplus_{k=1}^{\infty}\mathfrak{h}_{g,1}^{\mathbb{Z}}(k)
\otimes\mathbb{Q}
$$
is an injective homomorphism of graded Lie algebras. 
Hence the Johnson image $\tau\left(\bigoplus_{k=1}^{\infty}
{\rm gr}^k(\mathcal{I}_{g,1})\otimes\mathbb{Q}\right)$
can be regarded as the ``Lie algebra" of the Torelli group $\mathcal{I}_{g,1}$. 
But the map $\tau$ is {\it not} defined on the Torelli group itself, 
but on the graded quotients. So it is desirable to 
find a lift of $\tau$, or equivalently, an extension of $\tau$ to 
the Torelli group or to the whole mapping class group $\mathcal{M}_{g,1}$.
As will be stated below, there are various ways to construct extensions
of the Johnson homomorphisms. The diversity of constructions comes from that
of realizations of the Malcev completion of the free group
$\pi=\pi_1(\Sigma_{g,1},*)$.

The first result on this problem was given by Morita \cite{MorExt} \cite{MorFinland}
through an explicit construction of the automorphism group of the group $N_k$, a truncated
Malcev completion. Here it should be remarked that the abelianization 
${\mathcal{M}_{g,1}}^{\text{abel}}$ is trivial ($g\geq 3$) or finite ($g=2$).
Hence there exists no non-trivial homomorphism from $\mathcal{M}_{g,1}$ 
to any rational vector space if $g \geq2$. In \cite{MorExt} Morita gave an 
extension as a {\it crossed} homomorphism $\tilde k\colon \mathcal{M}_{g,1}
\to {\rm gr}^1(\mathcal{I}_{g,1})\otimes\mathbb{Q} = \Lambda^3H_\mathbb{Z}
\otimes\mathbb{Q}$ of the first Johnson homomorphism $\tau_1$. 
More precisely, he proved that there is a unique cohomology class $2\tilde k 
\in H^1(\mathcal{M}_{g,1}; \Lambda^3H_\mathbb{Z})$ whose restriction 
to $\mathcal{I}_{g,1}$ is twice the first Johnson homomorphism $2\tau_1$. 
Here $\Lambda^3H_\mathbb{Z}$ is a non-trivial $\mathcal{M}_{g,1}$-module 
in an obvious way. 
Let $\rho_0\colon \mathcal{M}_{g,1} \to Sp(H_\mathbb{Z})$ be the natural action of 
$\mathcal{M}_{g,1}$ on the first homology group $H_\mathbb{Z}$. 
The crossed homomorphism $\tilde k$ defines a group homomorphism 
$$
\rho_1\colon \mathcal{M}_{g,1} \to (\frac12\Lambda^3H_\mathbb{Z})\rtimes 
Sp(H_\mathbb{Z}),
$$
which induces a homomorphism of the cohomology groups
$$
\tilde k^*\colon H^*(\frac12\Lambda^3H_\mathbb{Z}; \mathbb{Q})^{Sp(H_\mathbb{Z})}
\to H^*((\frac12\Lambda^3H_\mathbb{Z})\rtimes 
Sp(H_\mathbb{Z}); \mathbb{Q}) 
\overset{{\rho_1}^*}\to H^*(\mathcal{M}_{g,1}; \mathbb{Q}).
$$
\begin{theorem}[Kawazumi-Morita \cite{KM}]
\label{thm:KM}
The image ${\rm Image}(\tilde k^*)$ equals the subalgebra of 
$H^*(\mathcal{M}_{g,1}; \mathbb{Q})$ generated by 
the Morita-Mumford classes $e_i = (-1)^{i+1}\kappa_i$, $i \geq 1$.
\end{theorem}

We remark that the theorem holds also for the unstable range. 
So it is not covered by the Madsen-Weiss theorem \cite{MW07}. 
The original proof of Theorem \ref{thm:KM} is obtained by 
interpreting the extended first Johnson homomorphism $\tilde k$ 
as the $(0,3)$-twisted Morita-Mumford class $m_{0,3}$ \cite{Ka98}.
\par
As for the second Johnson homomorphism $\tau_2$, 
Morita \cite{MorFinland} constructed a group homomorphism
$$
\rho_2\colon \mathcal{M}_{g,1} \to 
((\frac1{24}\mathfrak{h}^\mathbb{Z}_{g,1}(2))\tilde\times
(\frac1{2}\mathfrak{h}^\mathbb{Z}_{g,1}(1)))\rtimes Sp(M_\mathbb{Z})
$$
extending the homomorphisms $\rho_1$ and $\tau_2$, 
where $\tilde\times$ means some central extension of 
$\frac1{2}\mathfrak{h}^\mathbb{Z}_{g,1}(1)$ by 
$\frac1{24}\mathfrak{h}^\mathbb{Z}_{g,1}(2)$. 
From the Madsen-Weiss theorem, all the rational cohomology 
classes coming from $\rho_2$ in the stable range 
are generated by the Morita-Mumford classes.\par
From Hain's theorem \cite{Hain97} stated above
follows the existence of an extension of the $k$-th Johnson 
homomorphism to the whole $\mathcal{M}_{g,1}$ for any $k \geq 1$. 
On the other hand, Kawazumi \cite{Ka05} gave an explicit recipe
for constructing extensions of the totality of the Johnson homomorphisms
from a generalized Magnus expansion of a free group. 
For any $k \geq 1$, Kitano \cite{Kitano} described the $k$-th Johnson 
homomorphism in terms of the standard Magnus expansion of the 
free group $\pi = \pi_1(\Sigma_{g,1}, *)$ associated to a symplectic generating 
system. Moreover Perron \cite{Perron} constructed an extension of 
the $k$-th Johnson homomorphism $\tau_k$ for any $k \geq 1$ 
in terms of the standard Magnus expansion. 
 In general, consider a free group $F_n$ of rank $n \geq 2$ 
with a generating system $\{x_1, x_2, \dots, x_n\}$. 
Let $H$ be the abelianization of $F_n$ tensored by the rationals $\mathbb{Q}$,
$H := {F_n}^{\text{abel}}\otimes_{\mathbb{Z}}\mathbb{Q}$, and 
$\widehat{T} = \widehat{T}(H)$ the completed tensor algebra generated by 
$H$, $\widehat{T} = \widehat{T}(H) := \prod^\infty_{p=0}H^{\otimes p}$, 
which has a decreasing filtration $\{\widehat{T}_m\}^\infty_{m=1}$ 
of two-sided ideals defined by $\widehat{T}_m := 
\prod^\infty_{p\ge m}H^{\otimes p}$. The set $1+ \widehat{T}_1$ is a subgroup 
of the multiplicative group of the algebra $\widehat{T}$. The standard Magnus 
expansion of $F_n$ associated to $\{x_1, x_2, \dots, x_n\}$ is the group 
homomorphism $\operatorname{std}\colon F_n \to 1+\widehat{T}_1$ defined by 
$\operatorname{std}(x_i) := 1 + [x_i]$, $1 \leq i \leq n$. Here $[\gamma] 
:= (\gamma\bmod [F_n, F_n])\otimes 1 \in H = {F_n}^{\text{abel}}
\otimes_{\mathbb{Z}}\mathbb{Q}$ is the homology class of $\gamma \in F_n$. 
On the other hand, Bourbaki \cite{Bou} developed a basic theory of group 
homomorphisms $F_n \to 1 + \widehat{T}_1$. See also \cite{SerLNM}. 
So we define the notion of a (generalized) Magnus expansion of the free group 
$F_n$ by the minimum conditions for describing the Johnson homomorphisms. 
\begin{definition}[\cite{Ka05}]
\label{def:Magnus}
A map $\theta\colon F_n \to \widehat{T}$ is a ($\mathbb{Q}$-valued) 
Magnus expansion of the free group $F_n$ if it is a group homomorphism
of $F_n$ into $1+\widehat{T}_1$ and satisfies the condition $\theta(\gamma) 
\equiv 1 + [\gamma] \pmod{\widehat{T}_2}$ for any $\gamma \in F_n$. 
\end{definition}
The standard Magnus expansion $\operatorname{std}$ is a Magnus expansion 
in this definition. 
Let $\mathbb{Q}F_n$ be the rational group ring of the group $F_n$, and 
$\widehat{\mathbb{Q}F_n}$ its completion 
$$
\widehat{\mathbb{Q}F_n} := \varprojlim_{m\to \infty} \mathbb{Q}F_n/(IF_n)^m.
$$
Here $IF_n$ is the augmentation ideal, or equivalently, the kernel of the 
augmentation map $\operatorname{aug}\colon \mathbb{Q}F_n \to \mathbb{Q}$, 
$\sum_{\gamma\in F_n} a_\gamma \gamma \mapsto \sum_{\gamma\in F_n}  a_\gamma$. The algebra $\widehat{\mathbb{Q}F_n}$ has a natural decreasing 
filtration $\{ F_m \widehat{\mathbb{Q}F_n} \}^\infty_{m=1}$ defined by $F_m \widehat{\mathbb{Q}F_n}
:= \operatorname{Ker}(\widehat{\mathbb{Q}F_n} \to \mathbb{Q}F_n/(IF_n)^m)$. \par
Fix an arbitrary Magnus expansion $\theta$ of the free group $F_n$. 
Then its $\mathbb{Q}$-linear extension $\theta\colon \mathbb{Q}F_n \to \widehat{T}$, 
$\sum_{\gamma}a_\gamma\gamma \mapsto \sum_{\gamma}a_\gamma
\theta(\gamma)$, induces an algebra isomorphism
\begin{equation}
\theta\colon \widehat{\mathbb{Q}F_n} \overset\cong\to \widehat{T}
\label{eq:t-isom}
\end{equation}
such that $\theta(F_m \widehat{\mathbb{Q}F_n}) = \widehat{T}_m$ for any $m \geq 1$. 
See, for example, \cite{Ka05} Theorem 1.3. \par
For any automorphism $\varphi \in \operatorname{Aut}(F_n)$ of the 
group $F_n$, we define an automorphism $T^\theta(\varphi)$ of the algebra 
$\widehat{T}$ by $T^\theta(\varphi) := \theta\circ\varphi\circ\theta^{-1}\colon 
\widehat{T}\overset\cong\to \widehat{\mathbb{Q}F_n} \overset{\varphi}\to \widehat{\mathbb{Q}F_n} \overset\cong\to \widehat{T}$, which satisfies 
$T^\theta(\varphi)(\widehat{T}_m) = \widehat{T}_m$ for any $m \geq 1$. 
Denote by $\operatorname{Aut}(\widehat{T})$ the group of all automorphisms
$U$ of the algebra $\widehat{T}$ satisfying the condition $U(\widehat{T}_m) 
= \widehat{T}_m$ for any $m \geq 1$. Since the completion map $\mathbb{Q}
F_n \to \widehat{\mathbb{Q}F_n}$ is injective, the group homomorphism
$$
T^\theta\colon \operatorname{Aut}(F_n) \to \operatorname{Aut}(\widehat{T}), 
\quad \varphi \mapsto T^\theta(\varphi),
$$
is injective. All the Johnson homomorphisms come from the 
homomorphism $T^\theta$. So we call $T^\theta$ the {\it total Johnson 
map} of the automorphism group $\operatorname{Aut}(F_n)$ \cite{Ka05}. 
There are at least two ways to extract an extension of the $k$-th Johnson 
homomorphism $\tau_k$ from the map $T^\theta$. One way \cite{Ka05}
was prepared for the group cohomology of $\operatorname{Aut}(F_n)$, and 
the other \cite{Mas12} suitable for the Mal'cev completion of the group $F_n$. 
\par
First we explain the original Johnson map introduced in \cite{Ka05}.
Let $IA(\widehat{T})$ be the kernel of the natural action of $\operatorname{Aut}
(\widehat{T})$ on the space $\widehat{T}_1/\widehat{T}_2 = H$. 
Then the restriction to the subspace $H \subset \widehat{T}$ induces a linear 
isomorphism 
$$
IA(\widehat{T}) \cong \Hom(H, \widehat{T}_2) = \prod^\infty_{k=1}
\Hom(H, H^{\otimes (k+1)}),
$$
by which we identify these linear spaces. For any $\varphi \in 
\operatorname{Aut}(F_n)$ the induced map $\vert\varphi\vert$ 
on $H = {F_n}^{\text{abel}}\otimes\mathbb{Q}$ acts on the algebra 
$\widehat{T}$ in an obvious way, so that we may regard the composite 
$T^\theta(\varphi)\circ\vert\varphi\vert^{-1}$ as an element of 
$IA(\widehat{T}) = \prod^\infty_{k=1}\Hom(H, H^{\otimes (k+1)})$. 
We define the {\it $k$-th Johnson map} $\tau^\theta_k\colon \operatorname{Aut}
(F_n) \to \Hom(H, H^{\otimes(k+1)})$, $k\geq 1$, by 
$$
T^\theta(\varphi)\circ\vert\varphi\vert^{-1} = (\tau^\theta_k(\varphi))^\infty_{k=1}
\in IA(\widehat{T}) = \prod^\infty_{k=1}\Hom(H, H^{\otimes (k+1)}).
$$
The maps $\tau^\theta_k$'s are no longer group homomorphisms. 
Instead they satisfy an infinite sequence of coboundary equations.
For example, we have
\begin{align}
&-d\tau^\theta_1(\varphi) = 0 \in C^2(\operatorname{Aut}(F_n); 
\Hom(H, H^{\otimes 2})),\nonumber\\
&-d\tau^\theta_2(\varphi) = (\tau^\theta_1\otimes 1_H + 1_H\otimes \tau^\theta_1)\cup \tau^\theta_1 \in C^2(\operatorname{Aut}(F_n); 
\Hom(H, H^{\otimes 3})). \label{eq:cob}
\end{align}
Here $C^*(\operatorname{Aut}(F_n); M)$ is the normalized cochain complex 
of the group $\operatorname{Aut}(F_n)$ with values in an 
$\operatorname{Aut}(F_n)$-module $M$, $d$ the coboundary operator, 
and $\cup$ the Alexander-Whitney cup product. From the equation 
(\ref{eq:cob}) we obtain a straightforward proof of Theorem \ref{thm:KM}. 
Let $IA_n$ be the kernel of the natural action of $\operatorname{Aut}(F_n)$ 
on the abelianization ${F_n}^{\text{abel}}$, which is called 
the {\it IA-automorphism group}, and an analogue of the Torelli group.
Then we have an injective group homomorphism $T^\theta\colon IA_n 
\to IA(\widehat{T})$.
In the case $n=2g$ and $F_n = \pi = \pi_1(\Sigma_{g,1}, *)$, the restriction 
$\tau^\theta_k\vert_{\mathcal{M}_{g,1}(k)}$ equals the (original) $k$-th 
Johnson homomorphism $\tau_k$. In other words, the graded quotient of 
the restriction $T^\theta\vert_{\mathcal{I}_{g,1}}$ equals the totality of the 
(original) Johnson homomorphisms
\begin{equation}
{\rm gr}(T^\theta\vert_{\mathcal{I}_{g,1}}) = \tau\colon 
\bigoplus_{k=1}^{\infty} {\rm gr}^k(\mathcal{I}_{g,1}) \to
\bigoplus_{k=1}^{\infty}\mathfrak{h}_{g,1}^{\mathbb{Z}}(k)
\subset \bigoplus_{k=1}^{\infty}\Hom(H, H^{\otimes (k+1)}).
\label{eq:qJ}
\end{equation}
For details, see \cite{Ka05}.\par
Next we discuss Massuyeau's total Johnson map \cite{Mas12}
$$
\tau^\theta\colon IA_n \to \Hom(H, \mathcal{L}^+(\widehat{T})). 
$$
We need some generalities on a complete Hopf algebra 
to explain the definition of the target. 
The completed group ring $\widehat{\mathbb{Q}F_n}$ and 
the completed tensor algebra $\widehat{T} = \widehat{T}(H)$ 
are complete Hopf algebras, whose coproducts $\Delta$ are 
given by $\Delta(\gamma) = \gamma\widehat{\otimes}\gamma 
\in \widehat{\mathbb{Q}F_n}\widehat{\otimes}\widehat{\mathbb{Q}F_n}$ 
for $\gamma \in F_n$, and by $\Delta(X) = X\widehat{\otimes} 1 + 
1\widehat{\otimes} X \in \widehat{T}\widehat{\otimes}\widehat{T}$ 
for $X \in H$, respectively. We denote by $\operatorname{Gr}(R)$ 
the set of all group-like elements in a complete Hopf algebra $R$, 
or equivalently $\operatorname{Gr}(R) := \{r \in R\setminus\{0\}; \, 
\Delta(r) = r\widehat{\otimes}r \in R\widehat{\otimes} R\}$, 
which is a subgroup of the multiplicative group of the algebra $R$. 
The group $\operatorname{Gr}(\widehat{\mathbb{Q}F_n})$ is, 
by definition, the Mal'cev completion of the group $F_n$. 
Similarly we denote by $\mathcal{L}(R) := \{u \in R; \, 
\Delta(u) = u\widehat{\otimes}1 + 1\widehat{\otimes}u\}$ the set of 
all primitive elements, which is a Lie subalgebra of the associative 
algebra $R$. The Lie algebra $\mathcal{L}(\widehat{T})$ equals 
the degree completion of $\bigoplus^\infty_{k=1}
\mathcal{L}_\mathbb{Z}(k)\otimes\mathbb{Q}$. 
As is known \cite{Qui69}, the exponential $\exp\colon \mathcal{L}(R) \to 
\operatorname{Gr}(R)$, $\exp(u) := \sum^\infty_{k=1}(1/k!) u^k$, 
and the logarithm $\log\colon \operatorname{Gr}(R) \to \mathcal{L}(R)$, 
$\log(r) := \sum^\infty_{k=1}((-1)^{k-1}/k)(r-1)^k$, are the inverses 
of each other. \par
Let $IA_\Delta(\widehat{T})$ be the stabilizer of the coproduct $\Delta$ 
in the group $IA(\widehat{T})$, and $\mathcal{L}^+(\widehat{T})$ the 
degree completion of the Lie algebra $\bigoplus^\infty_{k=2}
\mathcal{L}_\mathbb{Z}(k)\otimes\mathbb{Q}$. Then the Lie algebra 
consisting of continuous derivations of $\widehat{T}$ which stabilize 
the coproduct $\Delta$ and vanish on the quotient $H = \widehat{T}_1/
\widehat{T}_2$ is naturally identified with $\Hom(H, 
\mathcal{L}^+(\widehat{T}))$, which is the target of the map $\tau^\theta$. 
The exponential $\exp\colon \Hom(H, \mathcal{L}^+(\widehat{T})) \to 
IA_\Delta(\widehat{T})$, $\exp(D) := \sum^\infty_{k=1}(1/k!) D^k$, 
and the logarithm $\log\colon IA_\Delta(\widehat{T}) \to \Hom(H, \mathcal{L}^+(\widehat{T}))$, 
$\log(U) := \sum^\infty_{k=1}((-1)^{k-1}/k)(U-1)^k$, are the inverses 
of each other. 
Massuyeau \cite{Mas12} introduced the notion of a group-like expansion of 
the group $F_n$.
\begin{definition}[Massuyeau \cite{Mas12}]
\label{def:gplike}
A Magnus expansion $\theta\colon F_n \to \widehat{T}$ of the free group $F_n$ 
is group-like if $\theta(F_n) \subset \operatorname{Gr}(\widehat{T})$, 
or equivalently $\Delta(\theta(\gamma)) = \theta(\gamma)\widehat{\otimes}\theta(\gamma)$ for any $\gamma \in F_n$.
\end{definition}
Fix a group-like expansion $\theta$. Then the isomorphism $\theta\colon 
\widehat{\mathbb{Q}F_n} \overset\cong\to \widehat{T}$ (\ref{eq:t-isom}) 
preserves the coproduct, so that the Malcev completion of $F_n$ is isomorphic to the group
of the group-like elements of $\widehat{T}$ through $\theta$, and we have 
$T^\theta(IA_n) \subset IA_\Delta(\widehat{T})$. Massuyeau introduced the composite 
\begin{equation}
\label{eq:MJM}
\tau^\theta := \log\circ T^\theta\colon IA_n \to \Hom(H, \mathcal{L}^+(\widehat{T})), 
\quad \varphi \mapsto \sum^\infty_{k=1}\frac{(-1)^{k+1}}{k}(T^\theta(\varphi) 
- 1)^k\vert_{H},
\end{equation}
which we call {\it Massuyeau's total Johnson map}. From (\ref{eq:qJ}) 
the graded quotient of $\tau^\theta$ equals the totality of the 
(original) Johnson homomorphisms (for $\operatorname{Aut}(F_n)$). 
In the case $n=2g$ and $F_n = \pi = \pi_1(\Sigma_{g,1}, *)$, it is desirable 
that $\tau^\theta(\mathcal{I}_{g,1}) \subset  \mathfrak{l}^+_g 
:= \prod^\infty_{k=1}\mathfrak{h}_{g,1}^\mathbb{Z}(k)\otimes\mathbb{Q}
\subsetneqq \Hom(H, \mathcal{L}^+(\widehat{T}))$. 
A symplectic expansion introduced by Massuyeau \cite{Mas12} 
makes it possible as will be stated in \S\ref{subsec:Symp}. 
Our purpose is to re-construct the map $\tau^\theta$ 
in a geometric context with no use of Magnus expansions.\par
We conclude this subsection by reviewing some other approaches  
to extending the Johnson homomorphisms or their enlargement 
to the whole mapping class group or some wider objects.
The fatgraph decompositions of the surface $\Sigma_{g,1}$ define 
the Ptolemy groupoid of $\Sigma_{g,1}$, which includes the 
mapping class group $\mathcal{M}_{g,1}$. 
Morita and Penner \cite{MP} introduced an explicit $1$-cocycle on the 
Ptolemy groupoid representing the extended first Johnson homomorphism
$\tilde k$. Bene, Kawazumi and Penner \cite{BKP} discovered a canonical way 
to associate a group-like expansion to any bordered trivalent fatgraph with 
one tail. Unfortunately it is not symplectic. But the $1$-cocycle in \cite{MP}
is the first term of the difference of two group-like expansions 
associated to two fatgraphs adjacent by one Whitehead move. 
Contracting the coefficients $\Lambda^3 H \to H$ by the intersection form 
on the homology group $H = H_1(\Sigma_{g,1}; \mathbb{Q})$, we have 
the Earle class $k \in H^1(\mathcal{M}_{g,1}; \HZ)$. Kuno, Penner and Turaev 
\cite{KPT} introduced an explicit $1$-cocycle 
on the Ptolemy groupoid representing 
the Earle class, which is simpler than the contraction of the Morita-Penner 
cocycle. On the other hand, in \cite{Mor93}, Morita introduced
a refinement of the $k$-th Johnson homomorphism $\mathcal{M}_{g,1}(k) 
\to H_3(N_k)$ for any $k \geq 1$. Here $H_3(N_k)$ is the third homology 
group of the nilpotent group $N_k$ in \S\ref{subsec:J-i}. 
Massuyeau \cite{Mas12a} discovered a canonical way
to attach a $3$-chain of the group $\pi$ modulo the boundaries to each
marked trivalent fatgraph, which is an extension of
Morita's refinement of the Johnson homomorphisms to the Ptolemy 
groupoid. It is unknown that the cocycle representing the Johnson homomorphisms induced from Massuyeau's and that in \cite{BKP} 
coincide with each other or not. As was proved by 
Massuyeau \cite{Mas12} Theorem 4.4, Morita's refinement is equivalent to
the sum $\bigoplus^{2k-1}_{j=k} \tau^\theta_j$. In \cite{Mas12}, he 
gave an extension of all of Morita's refinements to the monoid of homology 
cylinders, which includes the mapping class group $\mathcal{M}_{g,1}$. 
See the chapter by Habiro and Massuyeau \cite{HM12}. 
M. Day \cite{Day07} \cite{Day09} realized truncations of the Malcev completion
of the group $\pi$ in the framework of general Lie theory of nilpotent groups \cite{Nomizu}
to present two geometric ways to extend Morita's 
refinements to the whole mapping class group $\mathcal{M}_{g,1}$. 
It is also unknown whether one of them coincides with any of what we have stated
above or not.

\section{Dehn-Nielsen embedding}
\label{sec:DN}

In study of the mapping class group of a surface, it is often useful to consider
its action on curves on the surface.
In the classical case, the surface is $\Sigma=\Sigma_{g,1}$ as in \S \ref{sec:Cla}
and the mapping class group $\mathcal{M}_{g,1}$ acts on $\pi=\pi_1(\Sigma,*)$.
This action induces an injective group homomorphism
\begin{equation}
\label{eq:DN}
{\sf DN}\colon \mathcal{M}_{g,1}\to {\rm Aut}(\pi),
\end{equation}
whose image is characterized as the automorphisms of $\pi$ preserving the boundary loop of $\Sigma$.
This is the {\it Dehn-Nielsen} theorem.

In this section we work with general oriented surfaces
and consider an analogue of (\ref{eq:DN}) for their mapping class groups.
Instead of the fundamental group as in the case $\Sigma=\Sigma_{g,1}$,
we consider the {\it fundamental groupoid} of the surface with suitably chosen
base points and the action of the mapping class group on it.

\subsection{Groupoids and their completions}
\label{subsec:GP}

We begin by some general discussions about groupoids.

Let us recall a classical construction for a group $G$ (see \cite{Qui69}).
The {\it group ring} $\mathbb{Q}G$ is a $\mathbb{Q}$-vector space with basis the set $G$.
Extending $\mathbb{Q}$-bilinearly the product of $G$, it is a $\mathbb{Q}$-algebra.
Also it is a Hopf algebra with respect to the coproduct
$\Delta \colon \mathbb{Q}G\to \mathbb{Q}G\otimes \mathbb{Q}G$, $G\ni g\mapsto g\otimes g$
and the antipode $\iota\colon \mathbb{Q}G\to \mathbb{Q}G$, $G\ni g\mapsto g^{-1}$.
The {\it augmentation ideal} $IG$ is the kernel of the $\mathbb{Q}$-algebra homomorphism
$\mathbb{Q}G\to \mathbb{Q},\ G\ni g\mapsto 1$.
The powers $(IG)^n$, $n\ge0$, are two sided ideals of $\mathbb{Q}G$.
We denote by $\widehat{\mathbb{Q}G}$ the projective limit
$\varprojlim_n \mathbb{Q}G/(IG)^n$. The product, the coproduct, and the antipode of $\widehat{\mathbb{Q}G}$
are induced by those of $\mathbb{Q}G$. It is called the {\it completed group ring}
of $G$, and is naturally a complete Hopf algebra with respect to the filtration
$F_n\widehat{\mathbb{Q}G}={\rm Ker}(\widehat{\mathbb{Q}G}\to \mathbb{Q}G/(IG)^n)$, $n\ge 0$.
The set of group-like elements $\widehat{G}=\{ g\in \widehat{\mathbb{Q}G};
\Delta(g)=g\widehat{\otimes}g, g\neq 0 \}$ is a group with respect to the product of
$\widehat{\mathbb{Q}G}$, and is called the {\it Malcev completion} of $G$.
We have a canonical group homomorphism $G\to \widehat{G}$. This map is not injective in general.
Note that in \S \ref{subsec:Ext} we have already seen the above construction for a free group.

Let us consider an analogous construction for groupoids. Let $\mathcal{G}$ be a groupoid
such that the set of objects is ${\rm Ob}(\mathcal{G})$ and the set of morphisms
from $p_0\in {\rm Ob}(\mathcal{G})$ to $p_1\in {\rm Ob}(\mathcal{G})$ is $\mathcal{G}(p_0,p_1)$.
First we consider the ``group ring" for $\mathcal{G}$.
Let $\mathbb{Q}\mathcal{G}$ be the following small category. The set of objects
of $\mathbb{Q}\mathcal{G}$ is the same as that of $\mathcal{G}$, i.e., ${\rm Ob}(\mathcal{G})$.
The set of morphisms from $p_0\in {\rm Ob}(\mathcal{G})$ to $p_1\in {\rm Ob}(\mathcal{G})$ is
$\mathbb{Q} \mathcal{G}(p_0,p_1)$, the $\mathbb{Q}$-vector space with basis the set $\mathcal{G}(p_0,p_1)$.
By an obvious manner the product of morphisms in $\mathbb{Q}\mathcal{G}$ is
induced from that in $\mathcal{G}$. For any $p_0,p_1,p_2\in {\rm Ob}(\mathcal{G})$
the product $\mathbb{Q}\mathcal{G}(p_0,p_1)\times \mathbb{Q}\mathcal{G}(p_1,p_2)
\to \mathbb{Q}\mathcal{G}(p_0,p_2)$ is $\mathbb{Q}$-bilinear.
We define the {\it coproduct}, the {\it antipode}, and the {\it augmentation} of
$\mathbb{Q}\mathcal{G}$ as the collections
\begin{align*}
& \{ \Delta_{p_0,p_1}\colon \mathbb{Q}\mathcal{G}(p_0,p_1) \to \mathbb{Q}\mathcal{G}(p_0,p_1) \otimes
\mathbb{Q}\mathcal{G}(p_0,p_1) \}_{p_0,p_1\in {\rm Ob}(\mathcal{G})}, \\
& \{ \iota_{p_0,p_1}\colon \mathbb{Q}\mathcal{G}(p_0,p_1) \to \mathbb{Q}\mathcal{G}(p_1,p_0) \}_{p_0,p_1\in {\rm Ob}(\mathcal{G})},
\quad {\rm and}\\
& \{ {\rm aug}_{p_0,p_1}\colon \mathbb{Q}\mathcal{G}(p_0,p_1)\to \mathbb{Q} \}_{p_0,p_1\in {\rm Ob}(\mathcal{G})},
\end{align*}
of $\mathbb{Q}$-linear maps respectively,
where $\Delta_{p_0,p_1}$ is defined by
$\Delta_{p_0,p_1}(\ell)=\ell \otimes \ell$ for $\ell\in \Pi S(p_0,p_1)$,
$\iota_{p_0,p_1}$ is induced by
taking the inverse of morphisms in $\mathcal{G}$, and ${\rm aug}_{p_0,p_1}$ is defined by
${\rm aug}_{p_0,p_1}(\ell)=1$ for $\ell \in \mathcal{G}(p_0,p_1)$.
For $\ell \in \mathcal{G}(p_0,p_1)$, we denote
$\overline{\ell}:=\iota_{p_0,p_1}(\ell)$. If there is no fear of confusion,
we simply write $\Delta$, $\iota$, and ${\rm aug}$ instead of
$\Delta_{p_0,p_1}$, $\iota_{p_0,p_1}$, and ${\rm aug}_{p_0,p_1}$, respectively.
Clearly $\iota$ is a contravariant functor from $\mathbb{Q}\mathcal{G}$ to itself.
Let $\mathbb{Q}\mathcal{G}\otimes \mathbb{Q}\mathcal{G}$ be the following small category.
The set of objects of $\mathbb{Q}\mathcal{G}\otimes \mathbb{Q}\mathcal{G}$ is
${\rm Ob}(\mathcal{G})$, the set of morphisms from $p_0\in {\rm Ob}(\mathcal{G})$ to
$p_1\in {\rm Ob}(\mathcal{G})$ is $\mathbb{Q}\mathcal{G}(p_0,p_1) \otimes
\mathbb{Q}\mathcal{G}(p_0,p_1)$, and the product of morphisms in
$\mathbb{Q}\mathcal{G}\otimes \mathbb{Q}\mathcal{G}$ is the tensor product of
morphisms in $\mathbb{Q}\mathcal{G}$. We call $\mathbb{Q}\mathcal{G}\otimes \mathbb{Q}\mathcal{G}$
the {\it tensor product}. Then we can regard the coproduct $\Delta$ as a
covariant functor from $\mathbb{Q}\mathcal{G}$ to $\mathbb{Q}\mathcal{G}\otimes \mathbb{Q}\mathcal{G}$.

We next consider a concept corresponding to the augmentation ideal $IG$ and its powers.
Notice that for any $p\in {\rm Ob}(\mathcal{G})$ the set $\mathcal{G}_p=\mathcal{G}(p,p)$ is a group.
Let $p_0,p_1\in {\rm Ob}(\mathcal{G})$ and $n\ge 0$. If there is no morphism from
$p_0$ to $p_1$, i.e., $\mathcal{G}(p_0,p_1)=\emptyset$,
we set $F_n\mathbb{Q}\mathcal{G}(p_0,p_1)=0$. Otherwise, taking a morphism
$\ell\in \mathcal{G}(p_0,p_1)$ we set $F_n\mathbb{Q}\mathcal{G}(p_0,p_1)=(I\mathcal{G}_{p_0})^n\ell$.
Here $I\mathcal{G}_{p_0}$ is the augmentation ideal of the group $\mathcal{G}_{p_0}$.
We understand that $F_n\mathbb{Q}\mathcal{G}(p_0,p_1)=\mathbb{Q}\mathcal{G}(p_0,p_1)$ for $n<0$.

\begin{proposition}
\label{prop:filter}
\begin{enumerate}
\item The subspace $F_n\mathbb{Q}\mathcal{G}(p_0,p_1)$ is independent of the choice of $\ell$,
and $\{ F_n\mathbb{Q}\mathcal{G}(p_0,p_1) \}_{n\ge 0}$ is a decreasing filtration of
$\mathbb{Q}\mathcal{G}(p_0,p_1)$. The augmentation induces an isomorphism
$\mathbb{Q}\mathcal{G}(p_0,p_1)/F_1 \mathbb{Q}\mathcal{G}(p_0,p_1) \cong \mathbb{Q}$.
\item For any $p_0,p_1,p_2\in {\rm Ob}(\mathcal{G})$ and $n_1,n_2\ge 0$, we have
$$F_{n_1}\mathbb{Q}\mathcal{G}(p_0,p_1) \cdot F_{n_2}\mathbb{Q}\mathcal{G}(p_1,p_2)
\subset F_{n_1+n_2}\mathbb{Q}\mathcal{G}(p_0,p_2).$$
\item For any $p_0,p_1\in {\rm Ob}(\mathcal{G})$ and $n\ge 0$, we have
\begin{align*}
\Delta F_n\mathbb{Q}\mathcal{G}(p_0,p_1) &\subset
\sum_{n_1+n_2=n} F_{n_1}\mathbb{Q}\mathcal{G}(p_0,p_1)\otimes
F_{n_2}\mathbb{Q}\mathcal{G}(p_0,p_1), \\
\iota F_n\mathbb{Q}\mathcal{G}(p_0,p_1) &\subset
F_n\mathbb{Q}\mathcal{G}(p_1,p_0).
\end{align*}
\end{enumerate}
\end{proposition}

Clearly $F_n\mathbb{Q}\mathcal{G}(p,p)=(I\mathcal{G}_p)^n$ for any $p\in {\rm Ob}(\mathcal{G})$
and $n\ge 0$, and $F_1\mathbb{Q}\mathcal{G}(p_0,p_1)={\rm Ker}({\rm aug})$ for
any $p_0,p_1\in {\rm Ob}(\mathcal{G})$.

Now we construct a completion of $\mathbb{Q}\mathcal{G}$. Let $\widehat{\mathbb{Q}\mathcal{G}}$
be the following small category. The set of objects of $\widehat{\mathbb{Q}\mathcal{G}}$ is
${\rm Ob}(\mathcal{G})$. For $p_0,p_1\in {\rm Ob}(\mathcal{G})$ we set
$$\widehat{\mathbb{Q}\mathcal{G}}(p_0,p_1):=\varprojlim_n \mathbb{Q}\mathcal{G}(p_0,p_1)/
F_n\mathbb{Q}\mathcal{G}(p_0,p_1),$$
and define the set of morphisms from $p_0$ to $p_1$ to be $\widehat{\mathbb{Q}\mathcal{G}}(p_0,p_1)$.
By Proposition \ref{prop:filter} (2), the product of morphisms in $\widehat{\mathbb{Q}\mathcal{G}}$
is induced from that in $\mathbb{Q}\mathcal{G}$. Also, by Proposition \ref{prop:filter} (1)(3)
the coproduct, the antipode, and
the augmentation of $\widehat{\mathbb{Q}\mathcal{G}}$ are induced naturally.
We shall use the same letters $\Delta$, $\iota$, $\varepsilon$ for them.
For example the coproduct of $\widehat{\mathbb{Q}\mathcal{G}}$ is the collection of
maps $\Delta=\Delta_{p_0,p_1}$, $p_0,p_1\in {\rm Ob}(\mathcal{G})$, where $\Delta$ is
a map from $\widehat{\mathbb{Q}\mathcal{G}}(p_0,p_1)$ to the completed tensor product
\begin{align*}
& \widehat{\mathbb{Q}\mathcal{G}}(p_0,p_1) \widehat{\otimes}
\widehat{\mathbb{Q}\mathcal{G}}(p_0,p_1) \\
= &\varprojlim_n
(\mathbb{Q}\mathcal{G}(p_0,p_1) \otimes \mathbb{Q}\mathcal{G}(p_0,p_1))/
\sum_{n_1+n_2=n} F_{n_1}\mathbb{Q}\mathcal{G}(p_0,p_1) \otimes
F_{n_2}\mathbb{Q}\mathcal{G}(p_0,p_1).
\end{align*}
Again, introducing a small category $\widehat{\mathbb{Q}\mathcal{G}} \widehat{\otimes}
\widehat{\mathbb{Q}\mathcal{G}}$ by an obvious manner, we can regard $\Delta$
as a covariant functor from $\widehat{\mathbb{Q}\mathcal{G}}$ to
$\widehat{\mathbb{Q}\mathcal{G}} \widehat{\otimes} \widehat{\mathbb{Q}\mathcal{G}}$.

We call $\widehat{\mathbb{Q}\mathcal{G}}$ the {\it completion} of $\mathbb{Q}\mathcal{G}$.
We define a filtration of $\widehat{\mathbb{Q}\mathcal{G}}(p_0,p_1)$ by
$$F_n\widehat{\mathbb{Q}\mathcal{G}}(p_0,p_1):={\rm Ker}
(\widehat{\mathbb{Q}\mathcal{G}}(p_0,p_1) \to \mathbb{Q}\mathcal{G}(p_0,p_1)/
F_n\mathbb{Q}\mathcal{G}(p_0,p_1)),\quad {\rm for\ } n\ge 0.$$
There is a canonical isomorphism $\widehat{\mathbb{Q}\mathcal{G}}(p_0,p_1)\cong
\varprojlim_n \widehat{\mathbb{Q}\mathcal{G}}(p_0,p_1)/F_n\widehat{\mathbb{Q}\mathcal{G}}(p_0,p_1)$.
This filtration enjoys a property similar to Proposition \ref{prop:filter},
and endows $\widehat{\mathbb{Q}\mathcal{G}}(p_0,p_1)$ with a topology.
We shall say $u\in \widehat{\mathbb{Q}\mathcal{G}}(p_0,p_1)$ is {\it group-like}
if $\Delta(u)=u\widehat{\otimes}u$ and $u\neq 0$. Note that the set of
group like elements of $\widehat{\mathbb{Q}\mathcal{G}}$ is closed
under the product of morphisms and the antipode, and any group like element
of $\widehat{\mathbb{Q}\mathcal{G}}$ is an isomorphism.
Thus the set of group like elements of $\widehat{\mathbb{Q}\mathcal{G}}$ constitutes
a subcategory ${\rm Gr}(\widehat{\mathbb{Q}\mathcal{G}})$ of $\widehat{\mathbb{Q}\mathcal{G}}$
and is in fact a groupoid. There is a natural homomorphism of groupoids from
$\mathcal{G}$ to ${\rm Gr}(\widehat{\mathbb{Q}\mathcal{G}})$.
We call ${\rm Gr}(\widehat{\mathbb{Q}\mathcal{G}})$ the {\it Malcev completion of the groupoid $\mathcal{G}$}.

We end this subsection by recording the following fact which will be used later.
Let $n\ge 1$ and $p_0,p_1,\ldots,p_n\in {\rm Ob}(\mathcal{G})$, and assume
$\mathcal{G}(p_{i-1},p_i)\neq \emptyset$ for $1\le i\le n$. Then the multiplication
$\bigotimes_{i=1}^n F_1\mathbb{Q}\mathcal{G}(p_{i-1},p_i) \to F_n\mathbb{Q}\mathcal{G}(p_0,p_n)$
is surjective, and the sum of the multiplication and the inclusion
$\bigotimes_{i=1}^n F_1\widehat{\mathbb{Q}\mathcal{G}}(p_{i-1},p_i) \oplus F_{n+1}\widehat{\mathbb{Q}\mathcal{G}}(p_0,p_n)
\to F_n\widehat{\mathbb{Q}\mathcal{G}}(p_0,p_n)$ is surjective.

\subsection{Derivations and their exponentials}
\label{subsec:DA}

Let $\mathcal{G}$ be a groupoid. Recall that a derivation of an associative $\mathbb{Q}$-algebra $A$ is a
$\mathbb{Q}$-endomorphism $D\colon A\to A$ satisfying the Leibniz rule
$D(ab)=(Da)b+a(Db)$ for any $a,b\in A$. We generalize this notion to $\mathbb{Q}\mathcal{G}$
and $\widehat{\mathbb{Q}\mathcal{G}}$.
We define a {\it derivation} of $\mathbb{Q}\mathcal{G}$ to be a collection
$D=\{ D_{p_0,p_1}\}_{p_0,p_1\in {\rm Ob}(\mathcal{G})}$ of $\mathbb{Q}$-endomorphisms
$D_{p_0,p_1}\colon \mathbb{Q}\mathcal{G}(p_0,p_1)\to \mathbb{Q}\mathcal{G}(p_0,p_1)$
satisfying the Leibniz rule in the sense that
$$D_{p_0,p_2}(uv)=(D_{p_0,p_1}u)v+u(D_{p_1,p_2}v)$$
for any $p_0,p_1,p_2\in {\rm Ob}(\mathcal{G})$, $u\in \mathbb{Q}\mathcal{G}(p_0,p_1)$
and $v\in \mathbb{Q}\mathcal{G}(p_1,p_2)$. To simplify the notation we often
write $D$ instead of $D_{p_0,p_1}$. The derivations of $\mathbb{Q}\mathcal{G}$
form a Lie algebra ${\rm Der}(\mathbb{Q}\mathcal{G})$ with the Lie bracket
$[D_1,D_2]=D_1D_2-D_2D_1$, $D_1,D_2\in {\rm Der}(\mathbb{Q}\mathcal{G})$.
Similarly, we define a derivation of $\widehat{\mathbb{Q}\mathcal{G}}$
to be a collection of continuous $\mathbb{Q}$-endomorphisms of $\widehat{\mathbb{Q}\mathcal{G}}(p_0,p_1)$, $p_0,p_1\in {\rm Ob}(\mathcal{G})$,
satisfying the Leibniz rule in the same sense as before. We denote by ${\rm Der}(\widehat{\mathbb{Q}\mathcal{G}})$
the set of derivations of $\widehat{\mathbb{Q}\mathcal{G}}$. This is a Lie algebra by an obvious manner.
For later use we introduce a filtration of ${\rm Der}(\widehat{\mathbb{Q}\mathcal{G}})$.
For $n\in \mathbb{Z}$, we define $F_n{\rm Der}(\widehat{\mathbb{Q}\mathcal{G}})$
to be the set of $D\in {\rm Der}(\widehat{\mathbb{Q}\mathcal{G}})$
such that
$$D(F_{\ell}\widehat{\mathbb{Q}\mathcal{G}}(p_0,p_1)) \subset
F_{\ell+n}\widehat{\mathbb{Q}\mathcal{G}}(p_0,p_1)$$
for any $p_0,p_1\in {\rm Ob}(\mathcal{G})$ and $l\ge 0$.
We say that a derivation $D\in {\rm Der}(\widehat{\mathbb{Q}\mathcal{G}})$ {\it stabilizes}
the coproduct if $\Delta D=(D\widehat{\otimes} 1+1\widehat{\otimes} D)\Delta\colon
\widehat{\mathbb{Q}\mathcal{G}}(p_0,p_1)\to \widehat{\mathbb{Q}\mathcal{G}}(p_0,p_1)$ for
any $p_0,p_1\in {\rm Ob}(\mathcal{G})$. The derivations of $\widehat{\mathbb{Q}\mathcal{G}}$
stabilizing the coproduct form a Lie subalgebra ${\rm Der}_{\Delta}(\widehat{\mathbb{Q}\mathcal{G}})$
of ${\rm Der}(\widehat{\mathbb{Q}\mathcal{G}})$.

We show that a derivation of $\mathbb{Q}\mathcal{G}$ naturally induces
a derivation of $\widehat{\mathbb{Q}\mathcal{G}}$.
Let $D\in {\rm Der}(\mathbb{Q}\mathcal{G})$. We claim that for any $p_0,p_1\in {\rm Ob}(\mathcal{G})$ and $n\ge 0$ we have
$$D(F_n\mathbb{Q}\mathcal{G}(p_0,p_1))\subset F_{n-1}\mathbb{Q}\mathcal{G}(p_0,p_1).$$
To prove this, we may assume that $\mathcal{G}(p_0,p_1)\neq \emptyset$.
By the remark at the end of \S \ref{subsec:GP} there exist
$u_1,\ldots,u_{n-1}\in F_1\mathbb{Q}\mathcal{G}(p_0,p_0)$ and $u_n \in F_1\mathbb{Q}\mathcal{G}(p_0,p_1)$
such that $u=u_1\cdots u_{n-1}u_n$. Then
$$D(u_1\cdots u_n)=\sum_{i=1}^n u_1\cdots u_{i-1}(Du_i)u_{i+1}\cdots u_n \in F_{n-1}\mathbb{Q}\mathcal{G}(p_0,p_1),$$
as desired. This shows that $D=D_{p_0,p_1}$ induces a continuous $\mathbb{Q}$-endomorphism
of $\widehat{\mathbb{Q}\mathcal{G}}(p_0,p_1)$, and there is a natural Lie algebra homomorphism
${\rm Der}(\mathbb{Q}\mathcal{G})\to {\rm Der}(\widehat{\mathbb{Q}\mathcal{G}})$.

We next discuss the exponential of derivations. Recall that if
$A$ is an associative $\mathbb{Q}$-algebra and $D$ is a derivation of $A$,
then the formal power series $\exp(D)=\sum_{n=0}^{\infty}(1/n!)D^n$ is
a $\mathbb{Q}$-algebra automorphism of $A$, provided it converges.
To prove this note that for any $a,b\in A$ and $n\ge 0$, we have
$$D^n(ab)=\sum_{\substack{n_1,n_2\ge 0, \\ n_1+n_2=n}} \frac{n!}{n_1 !n_2 !}D^{n_1}(a)D^{n_2}(b)$$
by the Leibniz rule. Now let us consider the exponential of derivations of $\widehat{\mathbb{Q}\mathcal{G}}$.

\begin{lemma}[\cite{KK3} Lemma 1.3.2]
\label{lem:conv}
Suppose $D\in {\rm Der}(\widehat{\mathbb{Q}\mathcal{G}})$ satisfies the following conditions.
\begin{enumerate}
\item For any $p_0,p_1\in {\rm Ob}(\mathcal{G})$, $n\ge 0$, we have
$D(F_n\widehat{\mathbb{Q}\mathcal{G}}(p_0,p_1))\subset
F_n\widehat{\mathbb{Q}\mathcal{G}}(p_0,p_1)$.
\item For any $p_0,p_1\in {\rm Ob}(\mathcal{G})$, we have
$D(\mathbb{Q}\mathcal{G}(p_0,p_1))\subset F_1\widehat{\mathbb{Q}\mathcal{G}}(p_0,p_1)$.
\item For any $p_0,p_1\in {\rm Ob}(\mathcal{G})$, there exists $\nu>0$ such that
$D^{\nu}(F_1\widehat{\mathbb{Q}\mathcal{G}}(p_0,p_1))
\subset F_2\widehat{\mathbb{Q}\mathcal{G}}(p_0,p_1)$.
\end{enumerate}
Then for any $p_0,p_1\in {\rm Ob}(\mathcal{G})$ the series
$\exp(D)=\sum_{n=0}^{\infty}(1/n!)D^n$ converges and is a $\mathbb{Q}$-linear
homeomorphism of $\widehat{\mathbb{Q}\mathcal{G}}(p_0,p_1)$.
Moreover, if $D^{\prime}\in {\rm Der}(\widehat{\mathbb{Q}\mathcal{G}})$ satisfies
the above conditions and $\exp(D)=\exp(D^{\prime})$, then we have $D=D^{\prime}$.
\end{lemma}

Assume that $D\in {\rm Der}(\widehat{\mathbb{Q}\mathcal{G}})$ satisfies the assumption
of Lemma \ref{lem:conv}. It is a formal consequence of the Leibniz rule that
$\exp(D)(uv)=(\exp(D)u)(\exp(D)v)$ for any $p_0,p_1,p_2\in {\rm Ob}(\mathcal{G})$,
$u\in \widehat{\mathbb{Q}\mathcal{G}}(p_0,p_1)$, and
$v\in \widehat{\mathbb{Q}\mathcal{G}}(p_1,p_2)$. Thus $\exp(D)$ is
an automorphism of the small category $\widehat{\mathbb{Q}\mathcal{G}}$
acting on the set of objects as the identity. Moreover, by the condition
(2) of the assumption of Lemma \ref{lem:conv} the automorphism $\exp(D)$
preserves the augmentation in the sense that
${\rm aug}\circ \exp(D)={\rm aug}\colon \widehat{\mathbb{Q}\mathcal{G}}(p_0,p_1)\to \mathbb{Q}$
for any $p_0,p_1\in {\rm Ob}(\mathcal{G})$.

\subsection{Fundamental groupoid}
\label{subsec:MC}

Let us consider the construction in \S \ref{subsec:GP} for surfaces.
Let $S$ be an oriented surface. For $p_0,p_1\in S$, let
$$\Pi S(p_0,p_1)=[([0,1],0,1),(S,p_0,p_1)]$$
be the homotopy set of paths from $p_0$ to $p_1$. Throughout this chapter,
we often ignore the distinction between a path and its homotopy class.

Let $E$ be a non-empty closed subset of $S$, which is the disjoint
union of finitely many simple closed curves and finitely many points.
We denote by $\Pi S|_E$ the {\it fundamental groupoid} of
$S$ based at $E$. Namely, the set of objects of $\Pi S|_E$ is $E$,
and the set of morphisms from $p_0\in E$ to $p_1\in E$ is $\Pi S(p_0,p_1)$.
The product of morphisms is induced by conjunction of paths.
For $p_0,p_1\in E$ we denote $\mathbb{Q}\Pi S(p_0,p_1)$ and $\widehat{\mathbb{Q}\Pi S}(p_0,p_1)$
instead of $\Pi S|_E(p_0,p_1)$ and $\widehat{\mathbb{Q}\Pi S|_E}(p_0,p_1)$, respectively.

\subsection{Dehn-Nielsen homomorphism}
\label{subsec:DN}
Let $S$ and $E$ be as in \S \ref{subsec:MC}. We define the
{\it mapping class group of the pair $(S,E)$}, denoted by $\mathcal{M}(S,E)$,
as the group of diffeomorphisms of $S$ fixing $E\cup \partial S$ pointwise,
modulo isotopies fixing $E\cup \partial S$ pointwise.
If $E\subset \partial S$, we denote $\mathcal{M}(S,\partial S)$ or $\mathcal{M}(S)$
instead of $\mathcal{M}(S,E)$.
Unless otherwise stated we ignore the distinction between a diffeomorphism
and its mapping class in $\mathcal{M}(S,E)$.

The mapping class group $\mathcal{M}(S,E)$ acts naturally on the groupoid $\Pi S|_E$.
Let ${\rm Aut}(\Pi S|_E)$ be the group of automorphisms
of the groupoid $\Pi S|_E$ acting on the set of objects as the identity.
If $\varphi$ is a diffeomorphism fixing $E\cup \partial S$ pointwise,
then for any $p_0,p_1\in E$ and any path $\ell$ from $p_0$ to $p_1$ the path
$\varphi(\ell)$ is from $p_0$ to $p_1$. Moreover the homotopy class
$\varphi(\ell)\in \Pi S(p_0,p_1)$ depends only on the isotopy class of $\varphi$
and the homotopy class of $\ell$. In this way (the mapping class of) $\varphi$ induces
an automorphism of $\Pi S|_E$, giving a group homomorphism
\begin{equation}
\label{eq:DNhom}
{\sf DN}\colon \mathcal{M}(S,E) \to {\rm Aut}(\Pi S|_E).
\end{equation}
We call it the {\it Dehn-Nielsen} homomorphism.

We are interested in the case that ${\sf DN}$ is injective.
We say $S$ is {\it of finite type}, if $S$ is a compact oriented
surface, or a surface obtained from a compact oriented surface by
removing finitely many points in the interior.

\begin{theorem}
\label{thm:DN}
Suppose $S$ is of finite type and any component of $S$ has the non-empty boundary,
$E\subset \partial S$, and any connected component of $\partial S$
has an element of $E$. Then the homomorphism
${\sf DN}\colon \mathcal{M}(S,\partial S)\to {\rm Aut}(\Pi S|_E)$
is injective.
\end{theorem}

To prove Theorem \ref{thm:DN}, we argue as follows.
Let $\varphi \in \mathcal{M}(S,\partial S)$ and suppose ${\sf DN}(\varphi)=1$.
Take a system of proper arcs in $S$ such that the surface obtained
from $S$ by cutting along these arcs is the union of disks and
punctured disks. Since ${\sf DN}(\varphi)=1$, we may assume that $\varphi$
is identity on these arcs. Finally we deform $\varphi$ out side of these
arcs to the identity to conclude that $\varphi=1$.
For more detail, see the proof of Theorem 3.1.1 in \cite{KK3}.

Let ${\rm Aut}(\mathbb{Q}\Pi S|_E)$ be the group of automorphisms
of the small category $\mathbb{Q}\Pi S|_E$ acting on the set of objects
as the identity and on the set of morphisms $\mathbb{Q}$-linearly.
Further, let $\widehat{\mathbb{Q}\Pi S|_E}$ be the completion
of $\mathbb{Q}\Pi S|_E$ introduced in \S \ref{subsec:GP}.
We introduce the group ${\rm Aut}(\widehat{\mathbb{Q}\Pi S|_E})$
by the same manner as for $\mathbb{Q}\Pi S|_E$ except for considering only the automorphisms
acting on the set of morphisms continuously. Then
we have natural group homomorphisms ${\rm Aut}(\Pi S|_E)\to {\rm Aut}(\mathbb{Q}\Pi S|_E)$
and ${\rm Aut}(\mathbb{Q}\Pi S|_E) \to {\rm Aut}(\widehat{\mathbb{Q}\Pi S|_E})$.
By post-composing them to ${\sf DN}$, we get a group homomorphism
$$\widehat{{\sf DN}}\colon \mathcal{M}(S,E) \to {\rm Aut}(\widehat{\mathbb{Q}\Pi S|_E})$$
which we call the {\it completed Dehn-Nielsen homomorphism}.

For technical reasons and topological considerations,
we introduce a subgroup of ${\rm Aut}(\widehat{\mathbb{Q}\Pi S|_E})$
in which the homomorphism $\widehat{{\sf DN}}$ takes value.

\begin{definition}
\label{def:ASE}
Define the group $A(S,E)$ as the subgroup of ${\rm Aut}(\widehat{\mathbb{Q}\Pi S|_E})$
consisting of automorphisms $U$ satisfying the following conditions.

\begin{enumerate}
\item If $\gamma \in \Pi S(p_0,p_1)$ is represented by a
path included in $E$, then $U(\gamma)=\gamma$.

\item We have ${\rm aug}\circ U={\rm aug}\colon
\widehat{\mathbb{Q}\Pi S}(p_0,p_1) \to \mathbb{Q}$ for any
$p_0,p_1\in E$.

\item We have $\Delta U=(U\widehat{\otimes} U)\Delta\colon
\widehat{\mathbb{Q}\Pi S}(p_0,p_1)\to \widehat{\mathbb{Q}\Pi S}(p_0,p_1)
\widehat{\otimes} \widehat{\mathbb{Q}\Pi S}(p_0,p_1)$ for any
$p_0,p_1\in E$.
\end{enumerate}
\end{definition}

By (3), any $U\in A(S,E)$ preserves the group-like elements of $\widehat{\mathbb{Q}\Pi S|_E}$.
For any $\varphi \in \mathcal{M}(S,E)$, the element $\widehat{{\sf DN}}(\varphi)$ satisfies the
three conditions above. Note that $\widehat{{\sf DN}}(\varphi)$ satisfies (1)
since $\varphi$ fixes $E\cup \partial S$ pointwise. Thus we can write
\begin{equation}
\label{eq:DN-c}
\widehat{{\sf DN}}\colon \mathcal{M}(S,E) \to A(S,E).
\end{equation}

If $S$ and $E$ satisfy the assumption of Theorem \ref{thm:DN},
the fundamental group of each component of $S$ is a finitely generated free group.
Then for any $p\in E$ the natural map $\pi_1(S,p)\to \widehat{\mathbb{Q}\pi_1(S,p)}$
is injective, since $\bigcap_{n=1}^{\infty}I\pi_1(S,p)^n=0$ (see \cite{MKS}).
It follows that for any $p_0,p_1\in E$, the natural map
$\Pi S(p_0,p_1)\to \widehat{\mathbb{Q}\Pi S}(p_0,p_1)$ is also injective.

\begin{corollary}
\label{cor:DN-inj}
If $S$ and $E$ satisfy the assumption of {\rm Theorem \ref{thm:DN}},
the completed Dehn-Nielsen homomorphism {\rm (\ref{eq:DN-c})} is injective.
\end{corollary}

\subsection{Cut and paste arguments}
\label{subsec:c-p}

Notice that the construction of $\mathbb{Q}\mathcal{G}$
and $\widehat{\mathbb{Q}\mathcal{G}}$ for a groupoid $\mathcal{G}$ is functorial.
If $S$ is a subsurface of an oriented
surface $S^{\prime}$, and $E\subset S$ and $E^{\prime}\subset S^{\prime}$
are closed subsets as in \S \ref{subsec:MC} such that $E\subset E^{\prime}$,
then the inclusion map $S\hookrightarrow S^{\prime}$ induces
a groupoid homomorphism from $\Pi S|_E$ to $\Pi S^{\prime}|_{E^{\prime}}$,
called the {\it inclusion homomorphism}.
In this subsection we study certain kind of cut and paste arguments associated
to the inclusion homomorphism.

First we show the easier half of the van Kampen theorem for $\Pi S|_E$.
Let $S$ and $E$ be as in \S \ref{subsec:MC}, and let $S_1$ and $S_2$
be closed subsurfaces of $S$ such that $S_1\cup S_2=S$ and $S_1\cap S_2$ is
a disjoint union of finitely many simple closed curves on $S$.
We further assume that for $i=1,2$, the set $E_i:=S_i \cap E$ is a
disjoint union of finitely many simple closed curves and finitely many points,
and any connected component of $S_1\cap S_2$ has an element of $E$.
We denote $\mathcal{C}_i:=\Pi S_i|_{E_i}$, $i=1,2$.

We claim that $\Pi S|_E$ is ``generated by $\mathcal{C}_1$ and $\mathcal{C}_2$".
To formulate this claim we prepare some notations. For $p_0,p_1\in E$, we denote
by $\overline{\mathcal{E}}(p_0,p_1)$ the set of finite sequences of points in $E$,
$\lambda=(q_0,q_1,\ldots,q_n)\in E^{n+1}$, $n\ge 0$, such that
\begin{enumerate}
\item We have $q_0=p_0$ and $q_n=p_1$,
\item For $1\le j\le n$, either $\{ q_{j-1},q_j \}\subset S_1$ or $\{ q_{j-1},q_j \}\subset S_2$.
\end{enumerate}
Further let be $\mathcal{E}(p_0,p_1)$ the set of pairs $(\lambda,\mu)$,
$\lambda=(q_0,q_1,\ldots,q_n)\in \overline{\mathcal{E}}(p_0,p_1)$,
$\mu=(\mu_1,\ldots,\mu_n)\in \{1,2 \}^n$ such that $\{ q_{j-1},q_j \}\subset S_{\mu_j}$
for any $1\le j\le n$. For $(\lambda,\mu)\in \mathcal{E}(p_0,p_1)$, we set
$\mathbb{Q}\mathcal{C}(\lambda,\mu):=\bigotimes_{j=1}^n \mathbb{Q}\mathcal{C}_{\mu_j}(q_{j-1},q_j)$.
Then the multiplication map $\mathbb{Q}\mathcal{C}(\lambda,\mu) \to \mathbb{Q}\Pi S(p_0,p_1)$
is defined.

\begin{proposition}[the easier half of the van Kampen theorem, \cite{KK3}]
\label{prop:vc}
Keep the notations as above.
For any $p_0,p_1\in E$ the multiplication map
$$\bigotimes_{(\lambda,\mu)\in \mathcal{E}(p_0,p_1)} \mathbb{Q}\mathcal{C}(\lambda,\mu)
\to \mathbb{Q}\Pi S(p_0,p_1)$$
is surjective.
\end{proposition}

We next consider the forgetful homomorphisms.
Let $S$ and $E$ be as in \S \ref{subsec:MC}, and we assume that $S$ is of finite type
and has the non-empty boundary.
If $C$ is a simple closed curve on ${\rm Int}(S)\setminus E$,
we can consider the forgetful homomorphism $\mathcal{M}(S,E\cup C)\to \mathcal{M}(S,E)$,
and the kernel of this is generated by push maps along simple closed curves on
${\rm Int}(S)\setminus (E\cup C)$ parallel to $C$.
This can be proved by a standard argument found in e.g., \cite{FM11} \S 3.6.
We shall give a corresponding result for $A(S,E)$.

Let $C_i \subset {\rm Int}(S)\setminus E$, $1\le i\le n$, be disjoint
simple closed curves {\it not} null-homotopic in $S$. Set $E_1:=\bigcup_{i=1}^n C_i$.
The inclusion homomorphism $\Pi S|_E\to \Pi S|_{E\cup E_1}$
naturally induces the {\it forgetful homomorphism}
$\phi\colon A(S,E\cup E_1)\to A(S,E)$. We study the kernel of $\phi$.
For definiteness, we fix an orientation of each curve $C_i$, $1\le i\le n$.
For $1\le i\le n$ and $p\in C_i$, we denote by $\eta_{i,p}$ the loop $C_i$ based at $p$.
We can regard that $\eta_{i,p}\in \pi_1(S,p)$. Then for a rational number $a\in \mathbb{Q}$,
we can define
$$\eta_{i,p}^a:=\exp(a\log \eta_{i,p})\in \widehat{\mathbb{Q}\pi_1(S,p)}.$$

\begin{proposition}[\cite{KK3}]
\label{prop:kerphi}
Keep the notations as above.
Let $U\in A(S,E\cup E_1)$ and suppose $\phi(U)=1\in A(S,E)$.
Then there exist rational numbers $a_i=a_i^U\in \mathbb{Q}$, $1\le i\le n$,
such that for any $p_0,p_1\in E\cup E_1$ and $v\in \widehat{\mathbb{Q}\Pi S}(p_0,p_1)$,
we have
$$
Uv = 
\begin{cases}
v, &\mbox{if $p_0, p_1 \in E$,}\\
\eta^{a_{i_0}}_{i_0,p_0}v, &\mbox{if $p_0 \in C_{i_0}$,
$p_1\in E$,}\\
v(\eta^{a_{i_1}}_{i_1,p_1})^{-1}, &\mbox{if $p_0\in E$,
$p_1 \in C_{i_1}$,}\\
\eta^{a_{i_0}}_{i_0,p_0}v(\eta^{a_{i_1}}_{i_1,p_1})^{-1}, &\mbox{if $p_0
\in C_{i_0}$,
$p_1 \in C_{i_1}$.}
\end{cases}
$$
\end{proposition}



Morally, this proposition says that the kernel of $\phi$ is generated
by ``rational push maps" along $C_i$.

\section{Operations to curves on surfaces}
\label{sec:Ope}

Let $S$ be an oriented surface. In this section
we consider several operations to (the homotopy classes of) curves on $S$. Here a curve
on $S$ means a loop or a path on $S$.
These operations are first defined for curves
in general position, then shown to be homotopy invariant.

The quite natural but important property of these operations
is that they are equivariant with the action of the mapping class group.
In later sections we will see applications of this fact.
Rather technical but worth mentioning is
that these operations are compatible with filtrations on
the $\mathbb{Q}$-vector spaces based on curves
coming from the augmentation ideal of $\mathbb{Q}\pi_1(S)$ (see \S \ref{sec:DN}).
This point will be explained in \S \ref{subsec:Comp}.

We say that a curve on $S$ is
{\it generic} if it is an immersion and its self intersections
consist of transverse double points. Likewise, we say that finitely many curves on $S$
are {\it in general position} if each of the curves is generic and
their intersections consist of transverse double points. We often identify
a generic curve, which is a map to $S$, with its image, which is a subset of $S$.

For simplicity, we will consider over the rationals $\mathbb{Q}$.
However, all the constructions in this section works well over the integers $\mathbb{Z}$
as well as over any commutative ring with unit.

\subsection{Goldman-Turaev Lie bialgebra}
\label{subsec:GTL}
Let $\hat{\pi}(S)=[S^1,S]$ be the homotopy set of oriented free loops on $S$.
For $p\in S$ we denote by $|\ |\colon \pi_1(S)=\pi_1(S,p)\to \hat{\pi}(S)$
the map obtained by forgetting the base point of a based loop. If $S$ is connected, $|\ |$ is surjective.
Let $\mathbb{Q}\hat{\pi}(S)$ be the $\mathbb{Q}$-vector space with basis the set $\hat{\pi}(S)$.
The map $|\ |$ extends $\mathbb{Q}$-linearly to $|\ |\colon \mathbb{Q}\pi_1(S)\to \mathbb{Q}\hat{\pi}(S)$.

Let us recall the definition of the Goldman bracket. We use the intersection
of two generic oriented loops on $S$.
Let $\alpha$ and $\beta$ be oriented loops on $S$ in general position.
For each $p\in \alpha \cap \beta$, let $\varepsilon(p;\alpha,\beta)\in \{ \pm 1\}$
be the local intersection number of $\alpha$ and $\beta$ at $p$. Also let $\alpha_p$ be
the loop $\alpha$ based at $p$ and define $\beta_p$ similarly. Then the conjunction
$\alpha_p\beta_p\in \pi_1(S,p)$, and $|\alpha_p\beta_p|\in \hat{\pi}(S)$ are defined.
The {\it Goldman bracket} \cite{Go86} of $\alpha$ and $\beta$ is
\begin{equation}
\label{eq:G-bra}
[\alpha,\beta]:=\sum_{p\in \alpha\cap \beta}\varepsilon(p;\alpha,\beta)
|\alpha_p\beta_p|\in \mathbb{Q}\hat{\pi}(S).
\end{equation}
\begin{theorem}[Goldman \cite{Go86}]
\label{thm:G-bra}
The Goldman bracket {\rm (\ref{eq:G-bra})} induces a Lie bracket
$[\ ,\ ]\colon \mathbb{Q}\hat{\pi}(S)\otimes \mathbb{Q}\hat{\pi}(S)\to \mathbb{Q}\hat{\pi}(S)$.
\end{theorem}

We call $\mathbb{Q}\hat{\pi}(S)$ the {\it Goldman Lie algebra} of $S$.
Goldman introduced this Lie algebra along the study of the Poisson bracket
of two trace functions on the moduli space of flat $G$-bundles ${\rm Hom}(\pi_1(S),G)/G$,
where $G$ is a Lie group satisfying
very general conditions. The proof of Theorem \ref{thm:G-bra} goes as follows.

\begin{enumerate}
\item To prove that the Goldman bracket is well-defined, it suffices to check that
$[\alpha,\beta]$ is unchanged under the three local moves in Figure 1.
For, every pair of free loops on $S$ is homotopic to a generic pair of free loops,
and if two generic pairs of free loops on $S$ are homotopic to each other,
then they are related by a sequence of the three moves.
For another proof using twisted homology, see \cite{KK1} Proposition 3.4.3.
\item To prove that the Goldman bracket is a Lie bracket, one needs to check
that it is skew-symmetric and satisfies the Jacobi identity. The skew-symmetry
is clear from (\ref{eq:G-bra}) since $|\alpha_p\beta_p|=|\beta_p\alpha_p|$ and
$\varepsilon(p;\alpha,\beta)=-\varepsilon(p;\beta,\alpha)$ for $p\in \alpha\cap \beta$.
To prove the Jacobi identity, take three free loops $\alpha$, $\beta$, $\gamma$ in general position.
Then one can directly check $[\alpha,[\beta,\gamma]]+[\beta,[\gamma,\alpha]]+[\gamma,[\alpha,\beta]]=0$
using (\ref{eq:G-bra}).
\end{enumerate}
In this section we will see statements similar to Theorem \ref{thm:G-bra},
e.g., Theorems \ref{thm:T-cb}, \ref{thm:sigma}, \ref{thm:htpy-int}, and \ref{thm:bim}.
They can be proved by the same method as above. Note that if $S=\coprod_{\lambda}S_{\lambda}$
is the decomposition of $S$ into connected components, then $\mathbb{Q}\hat{\pi}(S)=
\bigoplus_{\lambda} \mathbb{Q}\hat{\pi}(S_{\lambda})$ as Lie algebras.
\begin{figure}
\label{fig:3move}
\begin{center}
\input{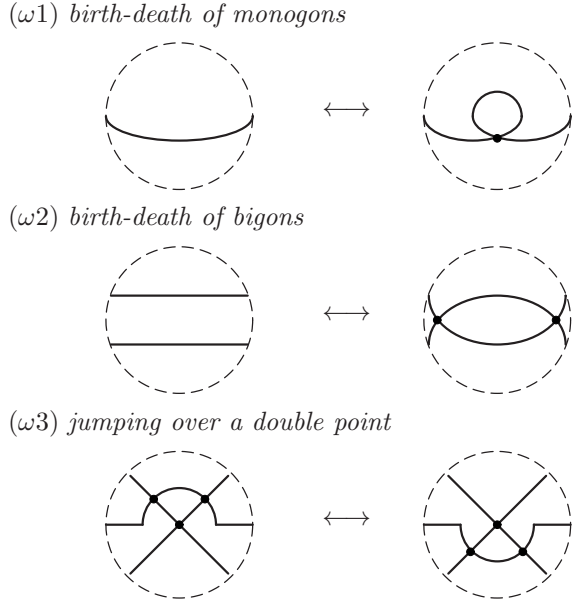}
\end{center}
\caption{the three moves}
\end{figure}

Next let us recall the definition of the Turaev cobracket. We use the self-intersection of a generic oriented loop on $S$.
For simplicity and a direct sum decomposition given in the last sentence of the proceeding paragraph we assume that $S$ is connected.
We denote by $1\in \hat{\pi}(S)$ the class of a constant loop. The $\mathbb{Q}$-linear
subspace $\mathbb{Q}1$ is an ideal of $\mathbb{Q}\hat{\pi}(S)$. We denote by
$\mathbb{Q}\hat{\pi}^{\prime}(S)$ the quotient Lie algebra $\mathbb{Q}\hat{\pi}(S)/\mathbb{Q}1$,
and let $\varpi\colon \mathbb{Q}\hat{\pi}(S)\to \mathbb{Q}\hat{\pi}^{\prime}(S)$ be
the projection. We write $|\ |^{\prime}:=\varpi\circ |\ |\colon
\mathbb{Q}\pi_1(S)\to \mathbb{Q}\hat{\pi}^{\prime}(S)$.

Let $\alpha\colon S^1\to S$ be a generic oriented loop.
Set $D=D_{\alpha}:=\{ (t_1,t_2) \in S^1\times S^1; t_1\neq t_2, \alpha(t_1)=\alpha(t_2) \}$.
For $(t_1,t_2)\in D$, let $\alpha_{t_1t_2}$ (resp. $\alpha_{t_2t_1}$) be
the restriction of $\alpha$ to the interval $[t_1,t_2]$ (resp. $[t_2,t_1]$) $\subset S^1$
(they are indeed loops since $\alpha(t_1)=\alpha(t_2)$).
Also, let $\varepsilon(\dot{\alpha}(t_1),\dot{\alpha}(t_2))\in \{ \pm 1 \}$ be the local intersection
number of the velocity vectors $\dot{\alpha}(t_i) \in T_{\alpha(t_i)}S$, $i=1,2$.
The {\it Turaev cobracket} \cite{Tu91} of $\alpha$ is
\begin{equation}
\label{eq:T-cb}
\delta(\alpha):=\sum_{(t_1,t_2)\in D}
\varepsilon(\dot{\alpha}(t_1),\dot{\alpha}(t_2)) |\alpha_{t_1t_2}|^{\prime} \otimes
|\alpha_{t_2t_1}|^{\prime}\in \mathbb{Q}\hat{\pi}^{\prime}(S)
\otimes \mathbb{Q}\hat{\pi}^{\prime}(S).
\end{equation}

\begin{theorem}[Turaev \cite{Tu91}, the involutivity is due to Chas \cite{Cha04}]
\label{thm:T-cb}
The Turaev cobracket {\rm (\ref{eq:T-cb})} induces a Lie cobracket
$\delta\colon \mathbb{Q}\hat{\pi}^{\prime}(S) \to
\mathbb{Q}\hat{\pi}^{\prime}(S) \otimes \mathbb{Q}\hat{\pi}^{\prime}(S)$.
Moreover, the $\mathbb{Q}$-vector space $\mathbb{Q}\hat{\pi}^{\prime}(S)$
is an involutive Lie bialgebra with respect to the Goldman bracket and the
Turaev cobracket.
\end{theorem}

To be more precise we have the following.

\begin{enumerate}
\item The space $\mathbb{Q}\hat{\pi}^{\prime}(S)$ is a Lie algebra with respect to the Goldman bracket.
\item The space $\mathbb{Q}\hat{\pi}^{\prime}(S)$ is a Lie coalgebra with respect to the Turaev cobracket.
\item We have $\delta [u,v]=\sigma(u)(\delta v)-\sigma(v)(\delta u)$ for
any $u, v\in \mathbb{Q}\hat{\pi}^{\prime}(S)$. Here $\sigma(u)(v\otimes w)=[u,v]\otimes w+v\otimes [u,w]$
for $u,v,w\in \mathbb{Q}\hat{\pi}^{\prime}(S)$.
\item We have $[\ ,\ ]\circ \delta=0\colon \mathbb{Q}\hat{\pi}^{\prime}(S) \to \mathbb{Q}\hat{\pi}^{\prime}(S)$.
\end{enumerate}

The condition (3) is called the {\it compatibility}, and (4) is called the {\it involutivity}.
We call $\mathbb{Q}\hat{\pi}^{\prime}(S)$ the {\it Goldman-Turaev Lie bialgebra}.
He introduced this Lie bialgebra along the study of a skein quantization of Poisson algebras
of loops on surfaces, and he showed that some skein bialgebra of links in $S\times [0,1]$
quantizes the Goldman-Turaev Lie bialgebra (\cite{Tu91} Theorem 10.1).

\subsection{The action of free loops on based paths}
\label{subsec:action}

We introduce an operation denoted by $\sigma$, using the intersection of an oriented loop
and a based path in general position.
Let $S$ and $E$ be as in \S \ref{subsec:MC}. Put $S^*=S\setminus (E\setminus \partial S)$.
Note that $S^*=S$ if $E\subset \partial S$.
We show that the Goldman Lie algebra $\mathbb{Q}\hat{\pi}(S^*)$ acts
on $\mathbb{Q}\Pi S|_E$ by derivations.

Take two points $*_0,*_1\in E$ which are not necessarily distinct.
Let $\alpha$ be an oriented loop on $S^*$ and $\beta\colon [0,1]\to S$
a path from $*_0$ to $*_1$, and assume that they are in general position.
For $p\in \alpha \cap \beta$, let $\varepsilon(p;\alpha,\beta)$ be the
local intersection number as before. Also let $\beta_{*_0p}$ be the path from $*_0$ to $p$
traversing $\beta$, and define $\beta_{p*_1}$ similarly. Then the conjunction
$\beta_{*_0p}\alpha_p\beta_{p*_1}\in \Pi S(*_0,*_1)$ is defined.
Set
\begin{equation}
\label{eq:sigma}
\sigma(\alpha\otimes \beta):=\sum_{p\in \alpha \cap \beta}
\varepsilon(p;\alpha,\beta)\beta_{*_0p}\alpha_p\beta_{p*_1}\in
\mathbb{Q}\Pi S(*_0,*_1).
\end{equation}

\begin{theorem}[\cite{KK1}]
\label{thm:sigma}
The formula {\rm (\ref{eq:sigma})} induces a $\mathbb{Q}$-linear map
$\sigma\colon \mathbb{Q}\hat{\pi}(S^*)\otimes \mathbb{Q}\Pi S(*_0,*_1)
\to \mathbb{Q}\Pi S(*_0,*_1)$. Moreover, with respect to $\sigma$
and the Goldman bracket, the vector space $\mathbb{Q}\Pi S(*_0,*_1)$ is a left
$\mathbb{Q}\hat{\pi}(S^*)$-module.
\end{theorem}

Recall that $1\in \hat{\pi}(S^*)$ denotes the class of a constant loop. We have $\sigma(1\otimes v)=0$
for any $v\in \mathbb{Q}\Pi S(*_0,*_1)$. Thus $\sigma$ naturally induces a map
$\mathbb{Q}\hat{\pi}^{\prime}(S^*)\otimes \mathbb{Q}\Pi S(*_0,*_1) \to \mathbb{Q}\Pi S(*_0,*_1)$, which we denote
by the same letter $\sigma$.
For $u\in \mathbb{Q}\hat{\pi}(S^*)$ and $m\in \mathbb{Q}\Pi S(*_0,*_1)$ we often
write $\sigma(u)m$ or $um$ for short instead of $\sigma(u\otimes m)$.
That $\mathbb{Q}\Pi S(*_0,*_1)$ is a left $\mathbb{Q}\hat{\pi}(S^*)$-module means
$$[u,v]m=u(vm)-v(um)$$ for $u,v\in \mathbb{Q}\hat{\pi}(S^*)$ and
$m\in \mathbb{Q}\Pi S(*_0,*_1)$.

If we consider not only a single pair $(*_0,*_1)$ but also all the ordered pairs of elements of $E$,
we obtain a derivation of $\mathbb{Q}\Pi S|_E$. First notice that the operation $\sigma$
satisfies the Leibniz rule in the following sense. For any $*_0,*_1,*_2\in E$ and
$\alpha\in \mathbb{Q}\hat{\pi}(S^*)$, $\beta_1\in \mathbb{Q}\Pi S(*_0,*_1)$, $\beta_2\in \mathbb{Q}\Pi S(*_1,*_2)$, we have
\begin{equation}
\label{eq:Leib}
\sigma(\alpha)(\beta_1\beta_2)=
(\sigma(\alpha)\beta_1)\beta_2+\beta_1 (\sigma(\alpha)\beta_2).
\end{equation}
This shows that for any $\alpha \in \mathbb{Q}\hat{\pi}(S^*)$ the collection
$\sigma(\alpha)=\sigma(\alpha)_{*_0,*_1}$, $*_0,*_1\in E$, determines
a derivation of $\mathbb{Q}\Pi S|_E$ in the sense of \S \ref{subsec:DA}.
Thus we get a $\mathbb{Q}$-linear map
\begin{equation}
\label{eq:s-Lie}
\sigma\colon \mathbb{Q}\hat{\pi}(S^*)\to {\rm Der}(\mathbb{Q}\Pi S|_E),
\end{equation}
and by the second sentence of Theorem \ref{thm:sigma}, this is a Lie algebra homomorphism.
As a special case, if $E=\{ * \}$ is a singleton with $*\in \partial S$,
the group ring $\mathbb{Q}\pi_1(S,*)$ is a $\mathbb{Q}\hat{\pi}(S)$-module
and we have a Lie algebra homomorphism
$\sigma\colon \mathbb{Q}\hat{\pi}(S) \to {\rm Der}(\mathbb{Q}\pi_1(S,*))$.
Note that for any $u\in \mathbb{Q}\hat{\pi}(S)$ and $v\in \mathbb{Q}\pi_1(S,*)$ we have
\begin{equation}
\label{eq:Gb-sig}
[u,|v|]=|\sigma(u\otimes v)|.
\end{equation}


\subsection{Intersection of based paths}
\label{subsec:htpy}

Take points $*_1,*_2,*_3,*_4$ on the boundary of $S$.
We define a $\mathbb{Q}$-linear map
$$\kappa \colon \mathbb{Q}\Pi S(*_1,*_2) \otimes \mathbb{Q}\Pi S(*_3,*_4)
\to \mathbb{Q}\Pi S(*_1,*_4) \otimes \mathbb{Q}\Pi S(*_3,*_2),$$
using the intersection of two based paths in general position.
Then we show that this is closely related to an operation called the homotopy intersection form
by Massuyeau and Turaev \cite{MT11}.

First we discuss the most generic case. Namely, we assume $\{ *_1,*_2\} \cap \{ *_3, *_4 \}=\emptyset$.
Let $x\colon [0,1]\to S$ be a path from $*_1$ to $*_2$ and $y\colon [0,1]\to S$ a path
from $*_3$ to $*_4$, and assume that they are in general position. Set
\begin{align}
\kappa(x,y) :&=-\sum_{p\in x\cap y} \varepsilon(p;x,y)
(x_{*_1p}y_{p*_4})\otimes (y_{*_3p}x_{p*_2}) \label{eq:ka-or} \\
& \in \mathbb{Q}\Pi S(*_1,*_4)\otimes \mathbb{Q}\Pi S(*_3,*_2). \nonumber
\end{align}
Here $x_{*_1p}$ is the path from $*_1$ to $p$ traversing $x$, etc.
One can show that (\ref{eq:ka-or}) gives rise to a well-defined $\mathbb{Q}$-linear map
$$\kappa\colon \mathbb{Q}\Pi S(*_1,*_2) \otimes \mathbb{Q}\Pi S(*_3,*_4)
\to \mathbb{Q}\Pi S(*_1,*_4) \otimes \mathbb{Q}\Pi S(*_3,*_2).$$
The operation $\kappa$ is introduced in \cite{KK4}. It satisfies the following product formula.
\begin{lemma}
\label{lem:kappa}
\begin{enumerate}
\item
Let $*_1,*_2,*_2^{\prime},*_3,*_4$ be points on the boundary of $S$
such that $\{ *_1,*_2,*_2^{\prime} \}\cap \{ *_3,*_4\}=\emptyset$. Then
for any $u\in \mathbb{Q}\Pi S(*_1,*_2)$, $v\in \mathbb{Q}\Pi S(*_2,*_2^{\prime})$
and $w\in \mathbb{Q}\Pi S(*_3,*_4)$, we have
\begin{equation}
\label{eq:k-pro1}
\kappa(uv,w)=\kappa(u,w)(1\otimes v)+(u\otimes 1)\kappa(v,w).
\end{equation}
\item
Let $*_1,*_2,*_3,*_4,*_4^{\prime}$ be points on the boundary of $S$
such that $\{ *_1,*_2\}\cap \{ *_3,*_4,*_4^{\prime} \}=\emptyset$.
Then for any $u\in \mathbb{Q}\Pi S(*_1,*_2)$, $v\in \mathbb{Q}\Pi S(*_3,*_4)$
and $w\in \mathbb{Q}\Pi S(*_4,*_4^{\prime})$, we have
\begin{equation}
\label{eq:k-pro2}
\kappa(u,vw)=\kappa(u,v)(w\otimes 1)+(1\otimes v)\kappa(u,w).
\end{equation}
\end{enumerate}
Here, $\kappa(u,w)(1\otimes v)$ is the image of $\kappa(u,w)\otimes v$ by the map
$\mathbb{Q}\Pi S(*_1,*_4)\otimes \mathbb{Q}\Pi S(*_3,*_2)\otimes \mathbb{Q}\Pi S(*_2,*_2^{\prime})
\to \mathbb{Q}\Pi S(*_1,*_4) \otimes \mathbb{Q}\Pi S(*_3,*_2^{\prime})$,
$a\otimes b\otimes c\mapsto a\otimes bc$, etc.
\end{lemma}

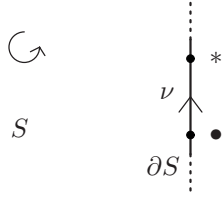
\begin{figure}
\label{fig:*-bul}
\begin{center}
\unitlength 0.1in
\begin{picture}( 10.5000, 10.0000)( 10.5000,-14.0000)
%
{\color[named]{Black}{%
\special{pn 13}%
\special{pa 2000 1200}%
\special{pa 2000 600}%
\special{fp}%
}}%
%
{\color[named]{Black}{%
\special{pn 8}%
\special{ar 1130 640 80 80  6.2831853 6.2831853}%
\special{ar 1130 640 80 80  0.0000000 4.7123890}%
}}%
%
{\color[named]{Black}{%
\special{pn 8}%
\special{pa 1210 640}%
\special{pa 1162 664}%
\special{fp}%
\special{pa 1210 640}%
\special{pa 1238 684}%
\special{fp}%
}}%
\put(10.6000,-11.1000){\makebox(0,0)[lb]{$S$}}%
\put(17.7000,-13.1000){\makebox(0,0)[lb]{$\partial S$}}%
%
{\color[named]{Black}{%
\special{pn 4}%
\special{sh 1}%
\special{ar 2000 1100 20 20 0  6.28318530717959E+0000}%
}}%
%
{\color[named]{Black}{%
\special{pn 4}%
\special{sh 1}%
\special{ar 2000 700 20 20 0  6.28318530717959E+0000}%
}}%
%
{\color[named]{Black}{%
\special{pn 13}%
\special{pa 2000 1200}%
\special{pa 2000 1400}%
\special{dt 0.045}%
}}%
%
{\color[named]{Black}{%
\special{pn 13}%
\special{pa 2000 600}%
\special{pa 2000 400}%
\special{dt 0.045}%
}}%
\put(21.0000,-7.3000){\makebox(0,0)[lb]{$*$}}%
\put(21.0000,-11.3000){\makebox(0,0)[lb]{$\bullet$}}%
%
{\color[named]{Black}{%
\special{pn 8}%
\special{pa 2000 900}%
\special{pa 1940 980}%
\special{fp}%
\special{pa 2000 900}%
\special{pa 2060 980}%
\special{fp}%
}}%
\put(18.4000,-9.1000){\makebox(0,0)[lb]{$\nu$}}%
\end{picture}%
\end{center}
\caption{the base points $*$ and $\bullet$}
\end{figure}

Next we consider the degenerate case. Let $*_1,*_2,*_3,*_4 \in \partial S$ be points on the
boundary of $S$ and assume $\{ *_1,*_2\} \cap \{ *_3,*_4\} \neq \emptyset$. To define
$\kappa$ for this case, we move the points $*_1,*_2$ slightly along the negatively
oriented boundary of $S$ to achieve
$\{ *_1,*_2\} \cap \{ *_3,*_4\}=\emptyset$, then apply the formula (\ref{eq:ka-or}).
For more precise explanation we use an example, which is the most extreme.
Namely, let us consider the case $*_1=*_2=*_3=*_4$.
Take a base point $*\in \partial S$ and pick an orientation preserving embedding
$\nu\colon [0,1]\to \partial S$ such that $\nu(1)=*$. Set $\nu(0)=\bullet$.
See Figure 2. Then we have three isomorphisms
$\pi_1(S,*) \cong \pi_1(S,\bullet)$, $x\mapsto \nu x \overline{\nu}$,
$\pi_1(S,*) \cong \Pi S(\bullet,*)$, $x\mapsto \nu x$, and
$\pi_1(S,*) \cong \Pi S(*,\bullet)$, $x\mapsto x \overline{\nu}$.
Now we define
\begin{equation}
\label{eq:ka-deg}
\kappa \colon \mathbb{Q}\pi_1(S,*) \otimes \mathbb{Q}\pi_1(S,*)
\to \mathbb{Q}\pi_1(S,*) \otimes \mathbb{Q}\pi_1(S,*)
\end{equation}
so that the diagram
$$
\begin{CD}
\mathbb{Q}\pi_1(S,*) \otimes \mathbb{Q}\pi_1(S,*) @>{\kappa}>>
\mathbb{Q}\pi_1(S,*) \otimes \mathbb{Q}\pi_1(S,*)\\
@V{\cong}VV @V{\cong}VV\\
\mathbb{Q}\pi_1(S,\bullet) \otimes \mathbb{Q}\pi_1(S,*) @>{\kappa}>>
\mathbb{Q}\Pi S(\bullet,*) \otimes \mathbb{Q}\Pi S(*,\bullet)
\end{CD}
$$
commutes. Here the vertical maps are via the above isomorphisms, and
the bottom horizontal arrow is the map already defined. To write down $\kappa$ in (\ref{eq:ka-deg}) explicitly,
let $\alpha$ be a loop based at $\bullet$, and $\beta$ a loop based at $*$ and
assume that they are in general position. By the isomorphism $\pi_1(S,*)\cong \pi_1(S,\bullet)$
given by $\nu$, we regard that $\alpha$ represents an element of $\pi_1(S,*)$. Then
\begin{equation}
\label{eq:ka-exp}
\kappa(\alpha,\beta):=-\sum_{p\in \alpha\cap \beta} \varepsilon(p;\alpha,\beta)
(\overline{\nu}\alpha_{\bullet p}\beta_{p*})\otimes
(\beta_{*p}\alpha_{p\bullet}\nu).
\end{equation}
Note this $\kappa$ satisfies (\ref{eq:k-pro1}) (\ref{eq:k-pro2})
for any $u,v,w\in \mathbb{Q}\pi_1(S,*)$.
By a similar way for any four points $*_1,*_2,*_3,*_4 \in \partial S$, which are not
necessarily distinct, we can define the operation $\kappa$.
Since we use only the most extreme case (\ref{eq:ka-deg}), we omit the detail of the construction.

Post-composing $-1\otimes {\rm aug}\colon \mathbb{Q}\pi_1(S,*) \otimes \mathbb{Q}\pi_1(S,*)
\to \mathbb{Q}\pi_1(S,*) \otimes \mathbb{Q} \cong \mathbb{Q}\pi_1(S,*)$ to (\ref{eq:ka-deg}),
we obtain a $\mathbb{Q}$-linear map
$$\eta\colon \mathbb{Q}\pi_1(S,*)\otimes \mathbb{Q}\pi_1(S,*)\to \mathbb{Q}\pi_1(S,*).$$
By (\ref{eq:ka-exp}), an explicit formula for $\eta$ is given by
\begin{equation}
\label{eq:hp-int}
\eta(\alpha,\beta):=\sum_{p\in \alpha \cap \beta} \varepsilon(p;\alpha,\beta)
\overline{\nu}\alpha_{\bullet p}\beta_{p*} \in \mathbb{Q}\pi_1(S,*),
\end{equation}
where notations are the same as in the preceding paragraph. The map $\eta$
is introduced by Massuyeau and Turaev \cite{MT11}, and is called the {\it homotopy intersection form}.
It is actually a modification of the operation $\lambda\colon \mathbb{Q}\pi_1(S,*) \times
\mathbb{Q}\pi_1(S,*) \to \mathbb{Q}\pi_1(S,*)$
introduced by Papakyriakopoulos \cite{Pa75} and Turaev \cite{Tu78} independently.
The relationship between $\lambda$ and $\eta$ is given by
$\lambda(\alpha,\beta)=\eta(\alpha,\beta)\beta^{-1}$ for $\alpha, \beta \in \pi_1(S,*)$.
By (\ref{eq:k-pro1}) (\ref{eq:k-pro2}), we have the following, which is essentially
due to Papakyriakopoulos \cite{Pa75} and Turaev \cite{Tu78}.

\begin{proposition}[\cite{MT11}]
\label{thm:htpy-int}
The homotopy intersection form satisfies the following identities:
\begin{align}
\eta(\alpha_1\alpha_2,\beta)& = \eta(\alpha_1,\beta){\rm aug}(\alpha_2)+
\alpha_1\eta(\alpha_2,\beta), \nonumber \\
\eta(\alpha,\beta_1\beta_2)&= \eta(\alpha,\beta_1)\beta_2+ {\rm aug}(\beta_1)\eta(\alpha,\beta_2),
\label{eq:eta-p}
\end{align}
where $\alpha,\alpha_1,\alpha_2,\beta,\beta_1,\beta_2\in \mathbb{Q}\pi_1(S,*)$.
Here ${\rm aug}\colon \mathbb{Q}\pi_1(S,*)\to \mathbb{Q}$ is the augmentation map.
\end{proposition}

In \cite{MT11}, a bilinear pairing on the group ring satisfying (\ref{eq:eta-p}) is
called a {\it Fox pairing}. In their theory, given a Fox pairing one
can consider its derived form. Actually the derived form $\eta$ turns out to be $\sigma$.
Let $u,v\in \mathbb{Q}\pi_1(S,*)$. The element $v$ is uniquely written
as $v=\sum_{x\in \pi}c_x x$ where $c_x\in \mathbb{Q}$. We denote
$u^v=\sum_{x\in \pi}c_x x^{-1}ux$.
We define a $\mathbb{Q}$-linear map
$\sigma^{\eta}\colon \mathbb{Q}\pi_1(S,*) \otimes \mathbb{Q}\pi_1(S,*) \to \mathbb{Q}\pi_1(S,*)$
by setting $\sigma^{\eta}(x\otimes y)=y(x^{\eta(x,y)})$ for $x,y\in \pi_1(S,*)$ and
extending $\mathbb{Q}$-linearly to $\mathbb{Q}\pi_1(S,*) \otimes \mathbb{Q}\pi_1(S,*)$.
In \cite{MT11}, $\sigma^{\eta}$ is called the {\it derived form of $\eta$}.
From (\ref{eq:eta-p}), we have
\begin{align}
& \sigma^{\eta}(u,vw)=\sigma^{\eta}(u,v)w+v\sigma^{\eta}(u,w), \nonumber \\
& \sigma^{\eta}(uv,w)=\sigma^{\eta}(vu,w)
\label{eq:sigma-p}
\end{align}
for $u,v,w \in \mathbb{Q}\pi_1(S,*)$.

\begin{lemma}[Massuyeau-Turaev \cite{MT11}]
\label{lem:derived}
The composition of $|\ |\otimes 1\colon \mathbb{Q}\pi_1(S,*) \otimes \mathbb{Q}\pi_1(S,*)
\to \mathbb{Q}\hat{\pi}(S)\otimes \mathbb{Q}\pi_1(S,*)$ and
$\sigma\colon \mathbb{Q}\hat{\pi}(S)\otimes \mathbb{Q}\pi_1(S,*) \to
\mathbb{Q}\pi_1(S,*)$ coincides with the map $\sigma^{\eta}$.
\end{lemma}

We end this subsection by a remark that one can recover $\kappa$ from $\eta$.

\begin{proposition}
\label{prop:k-eta}
Let $\kappa$ be the map in {\rm (\ref{eq:ka-deg})}. We have
$$\kappa=-(1\otimes m)(1\otimes 1\otimes m)P_{2431}
(1\otimes((1\otimes \iota)\Delta\eta)\otimes 1)(\Delta \otimes \Delta).$$
Here, $1$ is the identity map,
$\Delta$, $\iota$, and $m$ are the coproduct, the antipode,
and the product of the group ring $\mathbb{Q}\pi_1(S,*)$,
and $P_{2431}\colon \mathbb{Q}\pi_1(S,*)^{\otimes 4} \to \mathbb{Q}\pi_1(S,*)^{\otimes 4}$
is the $\mathbb{Q}$-linear map given by $P_{2431}(x_1\otimes x_2\otimes x_3\otimes x_4)
=x_2\otimes x_4\otimes x_3\otimes x_1$.
\end{proposition}

\subsection{Self intersections}
\label{subsec:self}

Take two points $*_0,*_1$ on the boundary of $S$. We define a $\mathbb{Q}$-linear map
$$\mu\colon \mathbb{Q}\Pi S(*_0,*_1)\to \mathbb{Q}\Pi S(*_0,*_1) \otimes \mathbb{Q}\hat{\pi}^{\prime}(S),$$
using the self intersections of a generic path from $*_0$ to $*_1$. Then we mention
a certain product formula for $\mu$ and a relationship with the Turaev cobracket.

First we consider the general case $*_0\neq *_1$.
Let $\gamma \colon [0,1]\to S$ be a generic path from $*_0$ to $*_1$.
We denote by $\Gamma\subset S$ the set of double points of $\gamma$.
For $p\in \Gamma$, we denote $\gamma^{-1}(p)=\{ t_1^p,t_2^p \}$,
so that $t_1^p<t_2^p$. Let $\varepsilon(\dot{\gamma}(t_1^p), \dot{\gamma}(t_2^p))$
be the local intersection number as in \S \ref{subsec:GTL}. We also define $\gamma_{0t_1^p}$ to
be the restriction of $\gamma$ to the interval $[0,t_1^p]$, and
define $\gamma_{t_2^p1}$ and $\gamma_{t_1^pt_2^p}$ similarly. Set
\begin{equation}
\label{eq:mu}
\mu(\gamma):=-\sum_{p\in \Gamma}\varepsilon(\dot{\gamma}(t_1^p), \dot{\gamma}(t_2^p))
(\gamma_{0t_1^p}\gamma_{t_2^p1})\otimes |\gamma_{t_1^pt_2^p}|^{\prime}
\in \mathbb{Q}\Pi S(*_0,*_1)\otimes \mathbb{Q}\hat{\pi}^{\prime}(S).
\end{equation}
One can show that this gives rise to a well-defined $\mathbb{Q}$-linear map
$\mu\colon \mathbb{Q}\Pi S(*_0,*_1)\to \mathbb{Q}\Pi S(*_0,*_1)\otimes \mathbb{Q}\hat{\pi}^{\prime}(S)$.
Next we consider the case $*_0=*_1$. Let $*\in \partial S$, $\nu\colon [0,1]\to \partial S$,
and $\bullet=\nu(0)$ be as in \S \ref{subsec:htpy}.
Then we have an isomorphism $\nu\colon \mathbb{Q}\pi_1(S,*)=\mathbb{Q}\Pi S(*,*)\cong
\mathbb{Q}\Pi S(\bullet,*)$, $u\mapsto \nu u$. We define
$\mu\colon \mathbb{Q}\pi_1(S,*)\to \mathbb{Q}\pi_1(S,*)\otimes \mathbb{Q}\hat{\pi}^{\prime}(S)$ so that the diagram
$$
\begin{CD}
\mathbb{Q}\pi_1(S,*) @>{\mu}>> \mathbb{Q}\pi_1(S,*) \otimes\mathbb{Q}\hat{\pi}^{\prime}(S)\\
@V{\nu}VV @V{\nu \otimes 1}VV\\
\mathbb{Q}\Pi S(\bullet,*) @>{\mu}>> \mathbb{Q}\Pi S(\bullet,*)\otimes\mathbb{Q}\hat{\pi}^{\prime}(S)
\end{CD}
$$
commutes.

\begin{theorem}[\cite{KK4}]
\label{thm:bim}
The $\mathbb{Q}$-vector space $\mathbb{Q}\Pi S(*_0,*_1)$
is an involutive right $\mathbb{Q}\hat{\pi}^{\prime}(S)$-bimodule with respect to $\sigma$ and $\mu$.
\end{theorem}

To be more precise we have the following.

\begin{enumerate}
\item The space $\mathbb{Q}\Pi S(*_0,*_1)$ is a left $\mathbb{Q}\hat{\pi}^{\prime}(S)$-module
with respect to $\sigma$ (see Theorem \ref{thm:sigma}).
\item The space $\mathbb{Q}\Pi S(*_0,*_1)$ is a right {\it $\mathbb{Q}\hat{\pi}^{\prime}(S)$-comodule}
with respect to $\mu$. That is, the diagram
$$
\begin{CD}
\mathbb{Q}\Pi S(*_0,*_1) @>{\mu}>> \mathbb{Q}\Pi S(*_0,*_1) \otimes\mathbb{Q}\hat{\pi}^{\prime}(S)\\
@V{\mu}VV @V{1\otimes\delta}VV\\
\mathbb{Q}\Pi S(*_0,*_1) \otimes\mathbb{Q}\hat{\pi}^{\prime}(S) @>{(1\otimes(1-T))(\mu\otimes
1)}>> \mathbb{Q}\Pi S(*_0,*_1)\otimes\mathbb{Q}\hat{\pi}^{\prime}(S)
\otimes\mathbb{Q}\hat{\pi}^{\prime}(S)
\end{CD}
$$
commutes. Here $\delta$ is the Turaev cobracket and
$T\colon \mathbb{Q}\hat{\pi}^{\prime}(S) \otimes
\mathbb{Q}\hat{\pi}^{\prime}(S) \to \mathbb{Q}\hat{\pi}^{\prime}(S) \otimes
\mathbb{Q}\hat{\pi}^{\prime}(S)$, $u\otimes v\mapsto v\otimes u$ is the switch map.
\item The operations $\sigma$ and $\mu$ satisfy the {\it compatibility} in the sense that
$$\sigma(u)\mu(m)-\mu(\sigma(u)m)-(\overline{\sigma}\otimes 1)(1\otimes \delta)(m\otimes u)=0$$
for $u\in \mathbb{Q}\hat{\pi}^{\prime}(S)$, $m\in \mathbb{Q}\Pi S(*_0,*_1)$.
Here $\overline{\sigma}\colon \mathbb{Q}\Pi S(*_0,*_1)\otimes \mathbb{Q}\hat{\pi}^{\prime}(S)
\to \mathbb{Q}\Pi S(*_0,*_1)$ is given by $\overline{\sigma}(m\otimes u)=-\sigma(u\otimes m)$,
and $\sigma(u)\mu(m)=(\sigma\otimes 1)(u\otimes \mu(m))+(1\otimes {\rm ad}(u))\mu(m)$.
\item The operation s $\sigma$ and $\mu$ satisfy the {\it involutivity} condition
$$\overline{\sigma}\mu=0\colon \mathbb{Q}\Pi S(*_0,*_1)\to \mathbb{Q}\Pi S(*_0,*_1).$$
\end{enumerate}

The operation $\mu$ is introduced in \cite{KK4}, and inspired by
Turaev's self intersection $\mu=\mu^T\colon \pi_1(S,*)\to \mathbb{Z}\pi_1(S,*)$
in \cite{Tu78} \S 1.4. Indeed, for any $\gamma \in \pi_1(S,*)$ we have
$\mu^T(\gamma)\gamma=-(1\otimes \varepsilon)\mu(\gamma)$, where
$\varepsilon \colon \mathbb{Q}\hat{\pi}^{\prime}(S)\to \mathbb{Q}$ is the $\mathbb{Q}$-linear map
given by $\varepsilon(\varpi(\alpha))=1$ for $\alpha \in \hat{\pi}(S)\setminus \{ 1\}$.

We end this subsection by stating two results about $\mu$.
The first one is a certain product formula, and the second one is a relation with the Turaev cobracket.

\begin{lemma}
\label{lem:mu}
For any $*_1,*_2,*_3\in E$ and $u\in \mathbb{Q}\Pi S(*_1,*_2)$, $v\in \mathbb{Q}\Pi S(*_2,*_3)$,
we have
$$\mu(uv)=\mu(u)(v\otimes 1)+(u\otimes 1)\mu(v)+
(1\otimes |\ |^{\prime})\kappa(u,v).$$
Here $\mu(u)(v\otimes 1)$ is the image of $\mu(u)\otimes v$
by the map $\mathbb{Q}\pi_1(S,*)\otimes \mathbb{Q}\hat{\pi}^{\prime}(S)
\otimes \mathbb{Q}\pi_1(S,*)\to \mathbb{Q}\pi_1(S,*) \otimes \mathbb{Q}\hat{\pi}^{\prime}(S)$,
$a\otimes b\otimes c \mapsto ac \otimes b$, etc.
\end{lemma}

As a corollary, for any $n\ge 2$ and $*_0,\ldots, *_n \in \partial S$,
$u_i \in \mathbb{Q}\Pi S(*_{i-1},*_i)$, $1\le i\le n$, we have

\begin{align}
\mu(u_1\cdots u_n)=& \sum_{i=1}^n ((u_1\cdots u_{i-1})\otimes 1)\mu(u_i)((u_{i+1}\cdots u_n)\otimes 1)
\nonumber \\
& +\sum_{i<j}((u_1\cdots u_{i-1})\otimes 1)K_{i,j}((u_{j+1}\cdots u_n)\otimes 1),
\label{eq:mu-pro}
\end{align}
where $K_{i,j}=(1\otimes |\ |^{\prime})(\kappa(u_i,u_j)(1\otimes (u_{i+1}\cdots u_{j-1})))$.

\begin{proposition}
\label{prop:mu-d}
The following diagram is commutative:
$$
\begin{CD}
\mathbb{Q}\pi_1(S,*) @>{\mu}>> \mathbb{Q}\pi_1(S,*) \otimes\mathbb{Q}\hat{\pi}^{\prime}(S)\\
@V{|\ |^{\prime}}VV @V{(1-T)(|\ |^{\prime}\otimes 1)}VV\\
\mathbb{Q}\hat{\pi}^{\prime}(S) @>{\delta}>> \mathbb{Q}\hat{\pi}^{\prime}(S) \otimes\mathbb{Q}\hat{\pi}^{\prime}(S)
\end{CD}
$$
\end{proposition}

\subsection{Completions of the operations}
\label{subsec:Comp}

We shall see that the operations we have considered extends naturally to completions.

First of all let us introduce a filtration of the vector space $\mathbb{Q}\hat{\pi}(S)$ and
its completion. Recall from \S \ref{subsec:GTL} the map
$|\ |\colon \mathbb{Q}\pi_1(S)\to \mathbb{Q}\hat{\pi}(S)$.
Note that the constant loop $1$ is always in the kernel of the homomorphism (\ref{eq:s-Lie}).
For $n\ge 0$, set
$$\mathbb{Q}\hat{\pi}(S)(n):=| \mathbb{Q}1+(I\pi_1(S))^n |$$
and $\mathbb{Q}\hat{\pi}^{\prime}(S)(n):=\varpi(\mathbb{Q}\hat{\pi}(S)(n))$.
We define the $\mathbb{Q}$-vector space $\widehat{\mathbb{Q}\hat{\pi}}(S)$ by
$$\widehat{\mathbb{Q}\hat{\pi}}(S):=\varprojlim_n \mathbb{Q}\hat{\pi}(S)/\mathbb{Q}\hat{\pi}(S)(n)
\cong \varprojlim_n \mathbb{Q}\hat{\pi}^{\prime}(S)/\mathbb{Q}\hat{\pi}^{\prime}(S)(n),$$
and introduce its filtration by
$$\widehat{\mathbb{Q}\hat{\pi}}(S)(n):={\rm Ker}(\widehat{\mathbb{Q}\hat{\pi}}(S)
\to \mathbb{Q}\hat{\pi}(S)/\mathbb{Q}\hat{\pi}(S)(n)), \quad n\ge 0.$$
The map $|\ |$ naturally induces a $\mathbb{Q}$-linear map
$|\ |\colon \widehat{\mathbb{Q}\pi_1(S)}\to \widehat{\mathbb{Q}\hat{\pi}}(S)$.
We understand that if $n<0$, $\mathbb{Q}\hat{\pi}(S)(n)=\mathbb{Q}\hat{\pi}(S)$.

\begin{proposition}[\cite{KK3} \cite{KK4}]
\label{prop:filt}
Let $m,n$ be integers $\ge 0$.
\begin{enumerate}
\item Let $S$ and $E$ be as in {\rm \S \ref{subsec:MC}}. For any $*_0,*_1\in E$ we have
$$\sigma \left( \mathbb{Q}\hat{\pi}(S^*)(m)
\otimes F_n\mathbb{Q}\Pi S(*_0,*_1) \right) \subset F_{m+n-2}\mathbb{Q}\Pi S(*_0,*_1).$$
\item For any $*_0,*_1\in \partial S$ we have
$$\mu(F_n\mathbb{Q}\Pi S(*_0,*_1))\subset
\sum_{p+q=n-2}F_p\mathbb{Q}\Pi S(*_0,*_1) \otimes \mathbb{Q}\hat{\pi}^{\prime}(S)(q).$$
\end{enumerate}
\end{proposition}

We remark that (2) follows from (\ref{eq:mu-pro}).
As an immediate consequence, we see that $\sigma$ and $\mu$ extends to completions:
\begin{align}
& \sigma\colon \widehat{\mathbb{Q}\hat{\pi}}(S^*) \widehat{\otimes}
\widehat{\mathbb{Q}\Pi S}(*_0,*_1) \to \widehat{\mathbb{Q}\Pi S}(*_0,*_1), \nonumber \\
& \mu\colon \widehat{\mathbb{Q}\Pi S}(*_0,*_1) \to
\widehat{\mathbb{Q}\Pi S}(*_0,*_1) \widehat{\otimes} \widehat{\mathbb{Q}\hat{\pi}}(S).
\label{eq:co}
\end{align}
Here $\widehat{\otimes}$ means the complete tensor product.
From (\ref{eq:Gb-sig}) and Proposition \ref{prop:mu-d}, we have the following corollary
to Proposition \ref{prop:filt}.

\begin{corollary}
\label{cor:G--T-c}
\begin{enumerate}
\item
For $u\in \mathbb{Q}\hat{\pi}(S)(m)$ and $v\in \mathbb{Q}\hat{\pi}(S)(n)$,
we have $[u,v]\in \mathbb{Q}\hat{\pi}(S)(m+n-2)$. In particular, the Goldman
bracket naturally induces a complete Lie bracket
$[\ ,\ ]\colon \widehat{\mathbb{Q}\hat{\pi}}(S)\widehat{\otimes}
\widehat{\mathbb{Q}\hat{\pi}}(S) \to \widehat{\mathbb{Q}\hat{\pi}}(S)$.
\item
If $u\in \mathbb{Q}\hat{\pi}^{\prime}(S)(n)$, then
$\delta (u)\in \sum_{p+q=n-2} \mathbb{Q}\hat{\pi}^{\prime}(S)(p) \otimes
\mathbb{Q}\hat{\pi}^{\prime}(S)(q)$. In particular, the Turaev cobracket
naturally induces a complete Lie cobracket
$\delta\colon \widehat{\mathbb{Q}\hat{\pi}}(S) \to \widehat{\mathbb{Q}\hat{\pi}}(S) \widehat{\otimes}
\widehat{\mathbb{Q}\hat{\pi}}(S)$.
\end{enumerate}
\end{corollary}

It is easy to see that the completed Lie bracket and cobracket on $\widehat{\mathbb{Q}\hat{\pi}}(S)$
inherit the compatibility and the involutivity from those on $\mathbb{Q}\hat{\pi}^{\prime}(S)$.
We call $\widehat{\mathbb{Q}\hat{\pi}}(S)$ the {\it completed Goldman-Turaev Lie bialgebra}.
Also if $*_0,*_1\in \partial S$, the vector space $\widehat{\mathbb{Q}\Pi S}(*_0,*_1)$ is
a {\it complete $\widehat{\mathbb{Q}\hat{\pi}}(S)$-bimodule}
with respect to the completed operations (\ref{eq:co}).

Let $S$ be a compact connected oriented surface with non-empty boundary,
and $E \subset \partial S$ a finite subset consisting of one point from each
component of the boundary $\partial S$. Then we have $S^* = S$, so that
the homomorphism (\ref{eq:s-Lie}) induces a Lie algebra homomorphism
$\sigma\colon \widehat{\mathbb{Q}\hat\pi}(S) \to {\rm Der}(
\widehat{\mathbb{Q}\Pi S|_E})$. We denote by
${\rm Der}_{\partial}(\widehat{\mathbb{Q}\Pi S|_E})$
the Lie subalgebra consisting of continuous derivations on
$\widehat{\mathbb{Q}\Pi S|_E}$ annihilating all based loops
inside the boundary $\partial S$. Clearly it includes the image $\sigma
(\widehat{\mathbb{Q}\hat\pi}(S))$. The following is an inifinitesimal version
of the Dehn-Nielsen theorem in \S\ref{sec:DN}.
\begin{theorem}
\label{thm:s-isom}
Let $S$ and $E$ be as above. Then the Lie algebra homomorphism
$$
\sigma\colon \widehat{\mathbb{Q}\hat\pi}(S) \to
{\rm Der}_{\partial}(\widehat{\mathbb{Q}\Pi S|_E})
$$
is an isomorphism.
\end{theorem}
The injectivity is proved in \cite{KK3}. The proof of the surjectivity,
which follows from a tensorial description of
$\widehat{\mathbb{Q}\hat\pi}(S)$, will appear in our forthcoming paper \cite{KK5}. in \S\ref{subsec:tensor} we will give an outline of the proof. 
\par
Finally we consider $\kappa$ and $\eta$. From Lemma \ref{lem:kappa}, for any
integers $m,n\ge 0$, and points $*_i\in \partial S$, $1\le i\le 4$, we have
\begin{align*}
& \kappa(F_m\mathbb{Q}\Pi S(*_1,*_2)\otimes F_n\mathbb{Q}\Pi S(*_3,*_4) ) \\
\subset & \bigoplus_{p+q=m+n-2} F_p\mathbb{Q}\Pi S(*_1,*_4) \otimes
F_q\mathbb{Q}\Pi S(*_3,*_2).
\end{align*}
We conclude that $\kappa$ extends naturally to completions:
$$\kappa\colon \widehat{\mathbb{Q}\Pi S}(*_1,*_2) \widehat{\otimes} \widehat{\mathbb{Q}\Pi S}(*_3,*_4)
\to \widehat{\mathbb{Q}\Pi S}(*_1,*_4) \widehat{\otimes} \widehat{\mathbb{Q}\Pi S}(*_3,*_2),$$
and by $\eta=(-1\otimes {\rm aug})\kappa$ so does $\eta$:
$$\eta\colon \widehat{\mathbb{Q}\pi_1(S,*)} \widehat{\otimes} \widehat{\mathbb{Q}\pi_1(S,*)}
\to \widehat{\mathbb{Q}\pi_1(S,*)}.$$

We end this section with a couple of remarks.

\begin{remark}
\begin{enumerate}
\item In later sections we consider the logarithms on the completed group ring of the fundamental
group of the surface, which is defined by
a formal power series with coefficients in $\mathbb{Q}$.
Thus we have to work with coefficients in a commutative ring including $\mathbb{Q}$.
For simplicity we confine ourselves to the case of $\mathbb{Q}$.
\item To define $\kappa$ for the degenerate case $\{ *_1,*_2 \}\cap \{ *_3,*_4\}\neq \emptyset$,
we move the points $*_1,*_2$ slightly along the negatively oriented boundary. However
this is not a unique way. Our aim is to achieve $\{ *_1,*_2 \}\cap \{ *_3,*_4\}=\emptyset$
by moving the endpoints of paths we consider. We may move the points $*_1,*_2$ slightly 
along the {\it positively} oriented boundary, etc., and we obtain a similar but different
operation. A similar matter occurs for the definition of $\mu$.
It is possible and might be desirable to develop this point in full generalities,
but here we avoid it for simplicity.
Our convention, in particular the choice of $*$ and $\bullet$ in Figure 2,
follows that of Massuyeau and Turaev \cite{MT11}.

\end{enumerate}
\end{remark}

\section{Dehn twists}
\label{sec:DT}

Let $S$ be an oriented surface and $C\subset S \setminus \partial S$ a simple closed curve.
The right handed {\it Dehn twist} along $C$, denoted by $t_C$, is a diffeomorphism
of the surface as in Figure 3. By definition, a Dehn twist is {\it local} in the sense that
the support of $t_C$ is contained in a regular neighborhood of $C$.
Dehn twists play a fundamental role in study of the mapping class group from combinatorial
group theory. For example, they give a generating set for the group,
cf. Dehn \cite{Deh38}, Lickorish \cite{Lick64} and Humphries \cite{HumLNM},
and a finite presentation of the group can be given in terms of Dehn twists.
The explicit presentation given first was Wajnryb \cite{Waj83} based on a result of Hatcher-Thurston \cite{HT80},
see also Matsumoto \cite{Mat00} and Gervais \cite{Ger01}.

In this section, we introduce an invariant of unoriented closed curves
on $S$, and using this invariant we give a formula for the image of $t_C$ by the completed Dehn-Nielsen homomorphism (\ref{eq:DN-c}).
The formula naturally leads us to introducing the {\it generalized Dehn twist} along an unoriented loop which
are not necessarily simple. This generalization takes value in $A(S,E)$,
a group introduced in \S \ref{subsec:DN}. We can ask whether a generalized Dehn twist comes from
a diffeomorphism of the surface. We partially give a negative answer to this question.

\begin{figure}
\label{fig:DT}
\begin{center}
\unitlength 0.1in
\begin{picture}( 39.2000,  8.3400)(  2.0000, -9.9800)
%
{\color[named]{Black}{%
\special{pn 13}%
\special{pa 200 358}%
\special{pa 1800 358}%
\special{fp}%
}}%
%
{\color[named]{Black}{%
\special{pn 13}%
\special{pa 200 998}%
\special{pa 1800 998}%
\special{fp}%
}}%
%
{\color[named]{Black}{%
\special{pn 13}%
\special{pa 2520 358}%
\special{pa 4120 358}%
\special{fp}%
}}%
%
{\color[named]{Black}{%
\special{pn 13}%
\special{pa 2520 998}%
\special{pa 4120 998}%
\special{fp}%
}}%
%
{\color[named]{Black}{%
\special{pn 8}%
\special{ar 1000 678 80 320  1.5707963 4.7123890}%
}}%
%
{\color[named]{Black}{%
\special{pn 8}%
\special{ar 1000 678 80 320  4.7123890 5.0123890}%
\special{ar 1000 678 80 320  5.1923890 5.4923890}%
\special{ar 1000 678 80 320  5.6723890 5.9723890}%
\special{ar 1000 678 80 320  6.1523890 6.4523890}%
\special{ar 1000 678 80 320  6.6323890 6.9323890}%
\special{ar 1000 678 80 320  7.1123890 7.4123890}%
\special{ar 1000 678 80 320  7.5923890 7.8539816}%
}}%
%
{\color[named]{Black}{%
\special{pn 8}%
\special{ar 3320 678 80 320  1.5707963 4.7123890}%
}}%
%
{\color[named]{Black}{%
\special{pn 8}%
\special{ar 3320 678 80 320  4.7123890 5.0123890}%
\special{ar 3320 678 80 320  5.1923890 5.4923890}%
\special{ar 3320 678 80 320  5.6723890 5.9723890}%
\special{ar 3320 678 80 320  6.1523890 6.4523890}%
\special{ar 3320 678 80 320  6.6323890 6.9323890}%
\special{ar 3320 678 80 320  7.1123890 7.4123890}%
\special{ar 3320 678 80 320  7.5923890 7.8539816}%
}}%
%
{\color[named]{Black}{%
\special{pn 8}%
\special{pa 200 678}%
\special{pa 1800 678}%
\special{fp}%
}}%
%
{\color[named]{Black}{%
\special{pn 8}%
\special{ar 3160 678 160 320  6.2831853 6.5331853}%
\special{ar 3160 678 160 320  6.6831853 6.9331853}%
\special{ar 3160 678 160 320  7.0831853 7.3331853}%
\special{ar 3160 678 160 320  7.4831853 7.7331853}%
}}%
%
{\color[named]{Black}{%
\special{pn 8}%
\special{ar 3480 678 160 320  3.1415927 3.3915927}%
\special{ar 3480 678 160 320  3.5415927 3.7915927}%
\special{ar 3480 678 160 320  3.9415927 4.1915927}%
\special{ar 3480 678 160 320  4.3415927 4.5915927}%
}}%
%
{\color[named]{Black}{%
\special{pn 8}%
\special{ar 3160 838 160 160  1.5707963 3.1415927}%
}}%
%
{\color[named]{Black}{%
\special{pn 8}%
\special{ar 2840 838 160 160  4.7123890 6.2831853}%
}}%
%
{\color[named]{Black}{%
\special{pn 8}%
\special{pa 2520 678}%
\special{pa 2840 678}%
\special{fp}%
}}%
%
{\color[named]{Black}{%
\special{pn 8}%
\special{ar 3480 518 160 160  4.7123890 6.2831853}%
}}%
%
{\color[named]{Black}{%
\special{pn 8}%
\special{ar 3800 518 160 160  1.5707963 3.1415927}%
}}%
%
{\color[named]{Black}{%
\special{pn 8}%
\special{pa 4120 678}%
\special{pa 3800 678}%
\special{fp}%
}}%
\put(3.7600,-6.3000){\makebox(0,0)[lb]{$\ell$}}%
\put(8.9600,-3.0200){\makebox(0,0)[lb]{$C$}}%
\put(32.1600,-2.9400){\makebox(0,0)[lb]{$C$}}%
\put(26.7200,-6.2200){\makebox(0,0)[lb]{$t_C(\ell)$}}%
\end{picture}%
\end{center}
\caption{the right handed Dehn twist}
\end{figure}
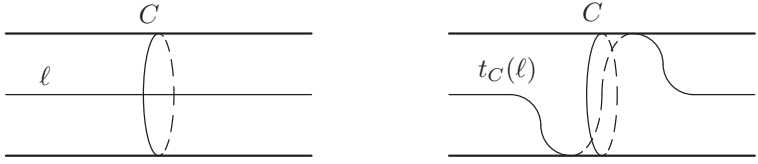

\subsection{The logarithms of Dehn twists}
\label{subsec:LDT}

Let $S$ and $E$ be as in \S \ref{subsec:MC}.
Recall that $S^*=S\setminus (E\setminus \partial S)$.
The homotopy set $\hat{\pi}(S^*)=[S^1,S^*]$ has an involution
which maps each $\gamma \in \hat{\pi}(S^*)$ to $\overline{\gamma}$,
the loop $\gamma$ with the reversed orientation. 
An {\it unoriented loop} on $S^*$ means an element of the quotient set
$\hat{\pi}(S^*)/(\gamma \sim \overline{\gamma})$.
We say that a (based) oriented loop $x$ on $S^*$ {\it represents} an unoriented loop $\gamma$ on $S^*$
if a suitable lift of $\gamma$ to $[S^1,S^*]$ equals (the homotopy class of) $x$.
As in \S \ref{sec:Ope}, we often identify an unoriented loop on $S^*$ with its image.
Likewise we use the word ``generic" for unoriented loops with the same meaning as before.

Let us consider the formal power series $L(t):=(1/2)(\log t)^2 \in \mathbb{Q}[[t-1]]$,
where $\log t=\sum_{n=1}^{\infty} ((-1)^{n-1}/n) (t-1)^n$.

\begin{definition}
Let $\gamma$ be an unoriented loop on $S^*\setminus \partial S$.
Take a base point $q\in S^*$ on the connected component of $S^*$ containing $\gamma$,
and a based oriented loop $x\in \pi_1(S^*,q)$ representing $\gamma$. Set
$$L(\gamma):=|L(x)| \in \widehat{\mathbb{Q}\hat{\pi}}(S^*)(2).$$
\end{definition}

Here $|\ |\colon \widehat{\mathbb{Q}\pi_1(S^*,q)} \to \widehat{\mathbb{Q}\hat{\pi}}(S^*)$
is the map introduced in \S \ref{subsec:Comp}. Since $L(t)=(1/2)(t-1)^2+({\rm higher\ term})$,
we have $L(x)\in F_2\widehat{\mathbb{Q}\pi_1(S,*)}$ and $|L(x)|\in \widehat{\mathbb{Q}\hat{\pi}}(S^*)(2)$.
The element $L(\gamma)$ does not depend on the choice of $q$ and $x$, because
the operation $|\ |$ is conjugate invariant and $L(t)=L(t^{-1})$.

We show that the invariant $L(\gamma)$ gives rise to an element of $A(S,E)$.
Recall from \S \ref{subsec:DA} the Lie algebra ${\rm Der}(\widehat{\mathbb{Q}\Pi S|_E})$
and its Lie subalgebra ${\rm Der}_{\Delta}(\widehat{\mathbb{Q}\Pi S|_E})$.
By Proposition \ref{prop:filt} (1) (see also (\ref{eq:co})), the Lie algebra homomorphism (\ref{eq:s-Lie})
induces a Lie algebra homomorphism
\begin{equation}
\label{eq:s-Lie-c}
\sigma\colon \widehat{\mathbb{Q}\hat{\pi}}(S^*) \to {\rm Der}(\widehat{\mathbb{Q}\Pi S|_E}).
\end{equation}
We claim that for any unoriented loop $\gamma$ on $S^* \setminus \partial S$,
the derivation $\sigma(L(\gamma))$ belongs to ${\rm Der}_{\Delta}(\widehat{\mathbb{Q}\Pi S|_E})$.
To see this, take an oriented loop $\alpha$ representing $\gamma$.
Let $*_0,*_1 \in E$ and take a path $\beta$ from $*_0$ to $*_1$, and assume that
$\alpha$ and $\beta$ are in general position.
It is sufficient to show that $\Delta \sigma(L(\gamma))\beta=
(\sigma(L(\gamma))\widehat{\otimes} 1+1\widehat{\otimes} \sigma(L(\gamma)))\Delta \beta$.
For $n\ge 0$, we have $\sigma(\alpha^n)\beta=\sum_{p\in \alpha \cap \beta}
n \varepsilon(p;\alpha,\beta)\beta_{*_0p}\alpha_p^n \beta_{p*_1}$,
thus for any formal power series $f(t)\in \mathbb{Q}[[t-1]]$ we have
$$\sigma(f(\alpha))\beta=\sum_{p\in \alpha \cap \beta}\beta_{*_0p}\alpha_p f^{\prime}(\alpha_p)\beta_{p*_1}.$$
Here $f^{\prime}(t)$ is the derivative of $f(t)$. In particular, since
$L^{\prime}(t)=(\log t)/t$,
$$\sigma(L(\gamma))\beta=\sum_{p\in \alpha \cap \beta}\varepsilon(p;\alpha,\beta)
\beta_{*_0p}(\log \alpha_p)\beta_{p*_1}.$$
On the other hand, $\Delta(\log \alpha_p)=\log \alpha_p\widehat{\otimes}1
+1\widehat{\otimes} \log \alpha_p \in \widehat{\mathbb{Q}\pi_1(S^*,p)}^{\widehat{\otimes} 2}$. Thus
\begin{align*}
& \Delta \sigma(L(\gamma))\beta=\sum_{p\in \alpha \cap \beta}\varepsilon(p;\alpha,\beta)
(\beta_{*_0p}\widehat{\otimes}\beta_{*_0p})(\log \alpha_p\widehat{\otimes}1+1\widehat{\otimes}\log \alpha_p)
(\beta_{p*_1}\widehat{\otimes}\beta_{p*_1}) \\
=& (\sigma(L(\gamma))\beta) \widehat{\otimes} \beta+ \beta \widehat{\otimes} (\sigma(L(\gamma))\beta)
=(\sigma(L(\gamma))\widehat{\otimes}1+1\widehat{\otimes}\sigma(L(\gamma)))\Delta \beta,
\end{align*}
as was to be shown.

\begin{lemma}[\cite{KK3}]
The derivation $\sigma(L(\gamma))\in {\rm Der}(\widehat{\mathbb{Q}\Pi S|_E})$
satisfies the three assumptions of {\rm Lemma \ref{lem:conv}}.
\end{lemma}

For the proof of this Lemma we refer to \cite{KK3} Lemma 5.1.1.
We remark that for any $*_0,*_1\in E$ we can take $\nu=2$ in the assumption (3),
and this corresponds to the following fact.
Using the intersection form $(\ \cdot \ ) \colon H_1(S^*)\otimes H_1(S^*,\partial S^*)\to \mathbb{Q}$
on the surface, we can assign any $X\in H_1(S^*)$ an endomorphism $D_X$ of $H_1(S^*,\partial S^*)$ given
by $D_X(Y)=(Y\cdot X)X$. Since $(\ \cdot \ )$ is skew-symmetric, the square of $D_X$ is zero.

Following the construction in \S \ref{subsec:DA} we obtain
$\exp(\sigma(L(\gamma)))\in {\rm Aut}(\widehat{\mathbb{Q}\Pi S|_E})$.
In fact, we have $\exp(\sigma(L(\gamma))) \in A(S,E)$, where $A(S,E)$ is the group
introduced in Definition \ref{def:ASE}. The condition (1) follows from the fact
that if $\alpha \in \mathbb{Q}\hat{\pi}(S^*)$ and a path $\beta\in \Pi S(*_0,*_1)$
are disjoint, then $\sigma(\alpha)\beta=0$. The condition (2) is automatic
(see the end of \S \ref{subsec:DA}), and (3) follows from $\sigma(L(\gamma))\in
{\rm Der}_{\Delta}(\widehat{\mathbb{Q}\Pi S|_E})$.

Now we have finished the main construction in this section. The reason
we are interested in $\exp(\sigma(L(\gamma)))$ comes from the following result.

\begin{theorem}[\cite{KK3}]
\label{thm:logDT}
Let $S$ and $E$ be as in {\rm \S \ref{subsec:MC}}, and
$C$ a simple closed curve on $S^*\setminus \partial S$,
where $S^*=S\setminus (E\setminus \partial S)$.
Then we have $$\widehat{{\sf DN}}(t_C)=\exp(\sigma(L(C))) \in A(S,E).$$
\end{theorem}

Using a toy model, we illustrate how the formula in Theorem \ref{thm:logDT} comes up.
Let $S$ be an annulus and $E=\{ p_0,p_1 \}$ as in Figure 4.
We consider the Dehn twist along a core $C$ of $S$.
Take curves $x$ and $y$ as in Figure 4. Then $x$ is
a representative of $C$. For $n\ge 0$ we have $\sigma(|x^n|)x=0$
and $\sigma(|x^n|)y=nx^ny$. Thus for any formal power series $f(t)\in \mathbb{Q}[[t-1]]$
we have $\sigma(|f(x)|)x=0$ and $\sigma(|f(x)|)y=xf^{\prime}(x)y$.
In particular, since $tL^{\prime}(t)=\log t$ we have
$\sigma(L(C))x=0$ and $\sigma(L(C))y=(\log x)y$. By the Leibniz rule (\ref{eq:Leib})
we obtain $\exp(\sigma(L(C)))x=x$ and $\exp(\sigma(L(C)))y=xy$.
On the other hand, clearly $t_C(x)=x$ and $t_C(y)=xy$.
Since $x$ and $y$ generate the fundamental groupoid $\Pi S|_E$,
we obtain $\exp(\sigma(L(C)))=\widehat{{\sf DN}}(t_C)$.
In fact, this is an essential part of the proof;
Theorem \ref{thm:logDT} for general $S$ is proved by
the theorem for an annulus and Proposition \ref{prop:vc}.


As an immediate consequence of Theorem \ref{thm:logDT},
for any $n>0$ the Dehn-Nielsen image $\widehat{{\sf DN}}(t_C)$
has a canonical $n$-th root: $\widehat{{\sf DN}}(t_C)^{1/n}=\exp((1/n)\sigma(L(C)))$.
Also, for any $a\in \mathbb{Q}$ we can consider the rational Dehn twist
$(t_C)^a:=\exp (a\sigma(L(C)))\in A(S,E)$. Suppose a generic path $\ell$ from $p_0\in E$ to
$p_1\in E$ intersects $C$ transversely and $\ell \cap C=\{ p\}$. We orient $C$ so
that $\varepsilon(p;C,\ell)=+1$ and define $\eta\in \pi_1(S,p)$ to be a oriented
loop $C$ based at $p$. Setting $\eta^a=\exp(a\log \eta)\in \widehat{\mathbb{Q}\pi_1(S,p)}$
as in \S \ref{subsec:c-p}, we have
\begin{equation}
\label{eq:r-DT}
(t_C)^a\ell=\ell_{p_0p}\eta^a\ell_{pp_1}.
\end{equation}

\begin{remark}
Theorem \ref{thm:logDT} was originally proved in \cite{KK1}
for a compact surface with one boundary component.
The formulation and proof in \cite{KK1} use a symplectic expansion
of the fundamental group of the surface. See also \S \ref{sec:Revi} and \ref{sec:Oth}.
Massuyeau and Turaev \cite{MT11} proved a similar result from another point of view.
See also Remark \ref{rem:GDT}.
\end{remark}

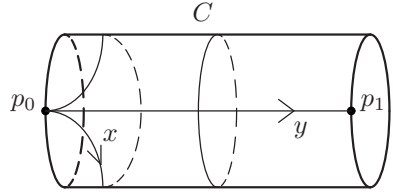
\begin{figure}
\label{fig:annulus}
\begin{center}
\unitlength 0.1in
\begin{picture}( 19.8000, 10.0000)(  9.2000,-12.0000)
%
{\color[named]{Black}{%
\special{pn 8}%
\special{ar 2000 800 100 400  1.5707963 4.7123890}%
}}%
%
{\color[named]{Black}{%
\special{pn 8}%
\special{ar 2000 800 100 400  4.7123890 4.9523890}%
\special{ar 2000 800 100 400  5.0963890 5.3363890}%
\special{ar 2000 800 100 400  5.4803890 5.7203890}%
\special{ar 2000 800 100 400  5.8643890 6.1043890}%
\special{ar 2000 800 100 400  6.2483890 6.4883890}%
\special{ar 2000 800 100 400  6.6323890 6.8723890}%
\special{ar 2000 800 100 400  7.0163890 7.2563890}%
\special{ar 2000 800 100 400  7.4003890 7.6403890}%
\special{ar 2000 800 100 400  7.7843890 7.8539816}%
}}%
\put(18.7000,-3.3000){\makebox(0,0)[lb]{$C$}}%
%
{\color[named]{Black}{%
\special{pn 13}%
\special{ar 1200 800 100 400  1.5707963 4.7123890}%
}}%
%
{\color[named]{Black}{%
\special{pn 13}%
\special{ar 1200 800 100 400  4.7123890 4.9523890}%
\special{ar 1200 800 100 400  5.0963890 5.3363890}%
\special{ar 1200 800 100 400  5.4803890 5.7203890}%
\special{ar 1200 800 100 400  5.8643890 6.1043890}%
\special{ar 1200 800 100 400  6.2483890 6.4883890}%
\special{ar 1200 800 100 400  6.6323890 6.8723890}%
\special{ar 1200 800 100 400  7.0163890 7.2563890}%
\special{ar 1200 800 100 400  7.4003890 7.6403890}%
\special{ar 1200 800 100 400  7.7843890 7.8539816}%
}}%
%
{\color[named]{Black}{%
\special{pn 13}%
\special{ar 2800 800 100 400  1.5707963 4.7123890}%
}}%
%
{\color[named]{Black}{%
\special{pn 13}%
\special{ar 2800 800 100 400  4.7123890 6.2831853}%
\special{ar 2800 800 100 400  0.0000000 1.5707963}%
}}%
%
{\color[named]{Black}{%
\special{pn 4}%
\special{sh 1}%
\special{ar 1100 800 20 20 0  6.28318530717959E+0000}%
}}%
%
{\color[named]{Black}{%
\special{pn 4}%
\special{sh 1}%
\special{ar 2700 800 20 20 0  6.28318530717959E+0000}%
}}%
%
{\color[named]{Black}{%
\special{pn 13}%
\special{pa 1200 1200}%
\special{pa 2800 1200}%
\special{fp}%
}}%
%
{\color[named]{Black}{%
\special{pn 13}%
\special{pa 1200 400}%
\special{pa 2800 400}%
\special{fp}%
}}%
%
{\color[named]{Black}{%
\special{pn 8}%
\special{ar 1100 1200 300 400  4.7123890 6.2831853}%
}}%
%
{\color[named]{Black}{%
\special{pn 8}%
\special{ar 1400 800 200 400  4.7123890 4.9123890}%
\special{ar 1400 800 200 400  5.0323890 5.2323890}%
\special{ar 1400 800 200 400  5.3523890 5.5523890}%
\special{ar 1400 800 200 400  5.6723890 5.8723890}%
\special{ar 1400 800 200 400  5.9923890 6.1923890}%
\special{ar 1400 800 200 400  6.3123890 6.5123890}%
\special{ar 1400 800 200 400  6.6323890 6.8323890}%
\special{ar 1400 800 200 400  6.9523890 7.1523890}%
\special{ar 1400 800 200 400  7.2723890 7.4723890}%
\special{ar 1400 800 200 400  7.5923890 7.7923890}%
}}%
%
{\color[named]{Black}{%
\special{pn 8}%
\special{ar 1100 400 300 400  6.2831853 6.2831853}%
\special{ar 1100 400 300 400  0.0000000 1.5707963}%
}}%
%
{\color[named]{Black}{%
\special{pn 8}%
\special{pa 1390 1080}%
\special{pa 1390 980}%
\special{fp}%
\special{pa 1390 1080}%
\special{pa 1320 1010}%
\special{fp}%
}}%
%
{\color[named]{Black}{%
\special{pn 8}%
\special{pa 1100 800}%
\special{pa 2700 800}%
\special{fp}%
}}%
%
{\color[named]{Black}{%
\special{pn 8}%
\special{pa 2400 800}%
\special{pa 2320 740}%
\special{fp}%
\special{pa 2400 800}%
\special{pa 2320 860}%
\special{fp}%
}}%
\put(9.2000,-8.0000){\makebox(0,0)[lb]{$p_0$}}%
\put(27.5000,-8.0000){\makebox(0,0)[lb]{$p_1$}}%
\put(24.0000,-9.4000){\makebox(0,0)[lb]{$y$}}%
\put(14.0000,-9.6000){\makebox(0,0)[lb]{$x$}}%
\end{picture}%
\end{center}
\caption{the annulus}
\end{figure}

\subsection{Generalized Dehn twists}
\label{subsec:GDT}

Let $S$ and $E$ be as in \S \ref{subsec:MC}.
Motivated by Theorem \ref{thm:logDT}, we introduce the following generalization
of Dehn twists.

\begin{definition}[\cite{KK3}]
Let $\gamma$ be an unoriented loop on $S^*\setminus \partial S$.
We define the {\it generalized Dehn twist} along $\gamma$ to be
$t_{\gamma}:=\exp (\sigma(L(\gamma)))\in A(S,E)$.
\end{definition}

\begin{remark}
\label{rem:GDT}
Generalized Dehn twists were introduced first for a compact surface with one boundary
component \cite{Ku11}, based on a main result of \cite{KK1}.
Then the definition for any oriented surfaces and their fundamental groupoids was given as above.
In \cite{MT11} Massuyeau and Turaev gave a similar construction.
They introduced the notion of a {\it Fox pairing} on the group ring of a group,
which generalizes the homotopy intersection form in \S \ref{subsec:htpy},
and worked in a more general framework.
Massuyeau and Turaev further discussed generalized Dehn twists for closed surfaces.
\end{remark}

We state two properties of generalized Dehn twists. First of all,
if $\gamma$ is an unoriented loop on $S^*\setminus \partial S$ and
$n\ge 0$ is an integer, we can consider the $n$-th power of $\gamma$, denoted by $\gamma^n$. Then
\begin{equation}
\label{eq:t^n}
t_{\gamma^n}=(t_{\gamma})^{n^2}.
\end{equation}
This follows from $L(t^n)=n^2L(t)$. Secondly, we have the following.

\begin{proposition}
\label{prop:ka-pres}
For any $*_i\in E\cap \partial S$, $1\le i \le 4$, and $u\in \widehat{\mathbb{Q}\Pi S}(*_1,*_2)$,
$v\in \widehat{\mathbb{Q}\Pi S}(*_3,*_4)$, we have
$$\kappa (t_{\gamma}(u)\widehat{\otimes} t_{\gamma}(v))=
(t_{\gamma}\widehat{\otimes}t_{\gamma})\kappa (u\widehat{\otimes} v).$$
\end{proposition}

Namely, $t_{\gamma}$ preserves the intersections of two paths on $S$.
This is first proved by Massuyeau and Turaev for the homotopy intersection form $\eta$,
see \cite{MT11} Lemma 8.2. The proof of Proposition \ref{prop:ka-pres} follows the
same line as the proof of \cite{MT11} Lemma 8.2 by Massuyeau and Turaev. Indeed, we first
prove that
\begin{equation}
\label{eq:k-sig}
\kappa(\sigma(x)u,v)+\kappa(u,\sigma(x)v)
=\sigma(x)(\kappa(u,v))
\end{equation}
for any $u\in \mathbb{Q}\Pi S(*_1,*_2)$, $v\in \mathbb{Q}\Pi S(*_3,*_4)$,
and $x\in \mathbb{Q}\hat{\pi}(S)$ (see \cite{MT11} Lemma 7.4). Here,
$\sigma(x)(a\otimes b)=(\sigma(x)a)\otimes b+a\otimes (\sigma(x)b)$ for
$a\in \mathbb{Q}\Pi S(*_1,*_4)$, $b\in \mathbb{Q}\Pi S(*_3,*_2)$.
The equation (\ref{eq:k-sig}) naturally induces an equality on completions.
Putting $x=L(\gamma)\in \widehat{\mathbb{Q}\hat{\pi}}(S^*)(2)$, we compute
\begin{align*}
& (t_{\gamma}\widehat{\otimes}t_{\gamma})\kappa (u\widehat{\otimes} v)=
(e^D \widehat{\otimes} e^D)\kappa (u\widehat{\otimes} v)=\sum_{r\ge 0}\frac{1}{r!} D^r(\kappa (u\widehat{\otimes} v)) \\
=& \sum_{r\ge 0} \frac{1}{r!} \sum_{i=0}^r \binom{r}{i} \kappa(D^iu\widehat{\otimes}D^{r-i}v)
= \sum_{i,j\ge 0} \frac{1}{i!j!} \kappa (D^iu\widehat{\otimes} D^jv)
=\kappa(e^Du \widehat{\otimes} e^Dv),
\end{align*}
where $D=\sigma(L(\gamma))$. This proves Proposition \ref{prop:ka-pres}.

We say that $t_{\gamma}$ is {\it realizable as a diffeomorphism}, or
{\it realizable} in short, if $t_{\gamma}$ is in the image of
the completed Dehn-Nielsen homomorphism (\ref{eq:DN-c}).
For example, if $C$ is a simple closed curve
on $S^*\setminus \partial S$ and $n\ge 0$ is an integer, $t_ {C^n}$ is realizable
by Theorem \ref{thm:logDT} and (\ref{eq:t^n}).

\begin{question}
For which unoriented loop $\gamma$ on $S^*\setminus \partial S$,
is $t_{\gamma}$ realizable as a diffeomorphism?
\end{question}

To study this question, we confine ourselves mainly to the case that $\widehat{{\sf DN}}$ is injective.
Then if $t_{\gamma}$ is realizable, there exists uniquely up to isotopy
an orientation diffeomorphism $\varphi$ of $S$ fixing $E\cup \partial S$ pointwise
such that $\widehat{{\sf DN}}(\varphi)=t_{\gamma}$. We call $\varphi$ a {\it representative}
for $t_{\gamma}$. Generalized Dehn twists are local in the following sense.

\begin{theorem}[\cite{Ku11}, \cite{KK3}]
\label{thm:local}
Suppose $S$ and $E$ satisfy the assumption of {\rm Theorem \ref{thm:DN}}.
Let $\gamma$ be an unoriented loop on $S \setminus \partial S$, and
suppose $t_{\gamma}$ is realizable as a diffeomorphism.
Then there is an representative of $t_{\gamma}$ whose support lies in a regular neighborhood of $\gamma$.
\end{theorem}

Here the {\it support} of a diffeomorphism $\varphi\colon S\to S$ is the closure
of the set $\{ x\in S; \varphi(x)\neq x \}$.

\subsection{Criterion using the self intersection}
\label{subsec:int}

We give an application of the operation $\mu$ in \S \ref{subsec:self} to generalized Dehn twists.

Assume that $S$ and $E$ satisfy the assumption of Theorem \ref{thm:DN}.
Let $\gamma$ be an unoriented loop on $S\setminus \partial S$.
Suppose $t_{\gamma}$ is realizable as a diffeomorphism and $\varphi$ is
a representative of $t_{\gamma}$. Let $*_0,*_1\in E\subset \partial S$ be distinct points
and $\ell$ a simple path from $*_0$ to $*_1$. Since $\mu$ maps simple paths to zero and
any diffeomorphism preserves the simplicity of paths, for any $n\ge 0$ we have
$\mu(\varphi^n(\ell))=\mu(\exp(n\sigma(L(\gamma))))=0$. Therefore, we must have
\begin{equation}
\label{eq:muL=0}
\mu(\sigma(L(\gamma))\ell)=0\in \widehat{\mathbb{Q}\Pi S}(*_0,*_1)\widehat{\otimes}
\widehat{\mathbb{Q}\hat{\pi}}(S).
\end{equation}
This observation is useful to detect the non-realizability of generalized Dehn twists.
Using our cut and paste arguments in \S \ref{subsec:c-p}, let us consider this in a more refined form.

Let $N\subset S\setminus \partial S$ be a connected compact subsurface which is a neighborhood of $\gamma$.
First of all we give the following variant of Proposition \ref{prop:kerphi}.
Let $N\subset S^*\setminus \partial S$ be a connected compact
subsurface with non-empty boundary, which is not diffeomorphic to the disk.
Assume that the inclusion homomorphism of fundamental groups $\pi_1(N)\to \pi_1(S)$
is injective. We number the components of $\partial N$ as $\partial N=\coprod_{i=1}^n C_i$.
For rational numbers $a_i$, $1\le i\le n$, we set
$F(a_1,\ldots, a_n):=\sum_{i=1}^n a_iL(C_i)\in \widehat{\mathbb{Q}\hat{\pi}}(N)(2)$.
Note that since $C_i$ are disjoint, the derivations $L(C_i)$ commute with each other.

\begin{proposition}[\cite{KK3}]
\label{prop:cut}
Keep the assumptions as above.
Let $U\in A(N,\partial N)$ and $\widetilde{U}\in A(S,E\cup \partial N)$ and
assume $\widetilde{U}(i(u))=iU(u)$ for any $p_0,p_1\in \partial N$ and
$u\in \widehat{\mathbb{Q}\Pi N}(p_0,p_1)$. Here $i\colon \Pi N|_{\partial N}
\to \Pi S|_{E\cup \partial N}$ is the inclusion homomorphism.
Further assume $\phi(\widetilde{U})=1$, where
$\phi\colon A(S,E\cup \partial N)\to A(S,E)$ is the forgetful homomorphism.
Then there exist rational numbers $a_i=a_i^U\in \mathbb{Q}$, $1\le i\le n$,
such that $U=\exp (\sigma(F(a_1,\ldots, a_n)))$.
\end{proposition}


Morally, this proposition says such $U$ is a product of rational Dehn twists:
$U=\prod_{i=1}^n (t_{C_i})^{a_i}$.
From the observation (\ref{eq:muL=0}), Theorem \ref{thm:local}, and Proposition \ref{prop:cut},
we can prove the following theorem.

\begin{theorem}[\cite{KK4}]
\label{thm:mus=0}
Keep the notations as above and suppose that the inclusion homomorphism
$\pi_1(N)\to \pi_1(S)$ is injective. If the generalized Dehn twist
$t_{\gamma}$ is realizable as a diffeomorphism, we have
$$\mu(\sigma(L(\gamma))\ell)=0\in \widehat{\mathbb{Q}\Pi N}(*_0,*_1)
\widehat{\otimes} \widehat{\mathbb{Q}\hat{\pi}}(N)$$
for any distinct points $*_0,*_1\in \partial N$ and any simple path
$\ell \in \Pi N(*_0,*_1)$.
\end{theorem}


The following theorem provides many examples of unoriented loops $\gamma$ such that
$t_{\gamma}$ is not realizable as a diffeomorphism. The proof is by
taking $N$ to be a regular neighborhood of $\gamma$ and $\ell$
a simple path intersecting $\gamma$ transversely in a single point,
and showing that $\mu(\sigma(L(\gamma))\ell)\neq 0$.

\begin{theorem}[\cite{KK4}]
\label{thm:non-r}
Let $\gamma\subset S\setminus \partial S$ be a generic non-simple loop and
assume that the inclusion homomorphism $\pi_1(N(\gamma))\to \pi_1(S)$ is injective,
where $N(\gamma)$ is a closed regular neighborhood of $\gamma$.
Then the generalized Dehn twist $t_{\gamma}$ is not realizable as a diffeomorphism.
\end{theorem}

For example, the generalized Dehn twist along a figure eight is not realizable as a diffeomorphism.
This is first proved in \cite{Ku11} \cite{KK3} by a rather ad hoc way.
Here we say that an oriented generic loop $\gamma$ is a {\it figure eight} if $\gamma$ has
a single self intersection and the inclusion homomorphism $\pi_1(N(\gamma))\to \pi_1(S)$ is injective.
For another example, suppose that $\gamma$ is generic and non-simple, and the inclusion map
$N(\gamma) \hookrightarrow S$ is a homotopy equivalence. Then $t_{\gamma}$ is
not realizable as a diffeomorphism. This shows that {\it locally}, a generalized Dehn
twist is never realized as a diffeomorphism. On the otherhand, let $\gamma$ be an
unoriented loop shown in Figure 5. Since $\pi_1(N(\gamma))\to \pi_1(S)$
is not injective, Theorem \ref{thm:non-r} cannot be applied. However, if $N$ is
a neighborhood of $\gamma$ diffeomorphic to a pair of pants as shown in Figure 5,
and $\pi_1(N)\to \pi_1(S)$ is injective, then we can apply Theorem \ref{thm:mus=0} to conclude
that $t_{\gamma}$ is not realizable as a diffeomorphism.

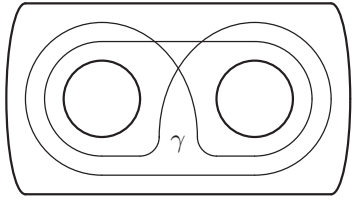
\begin{figure}
\label{fig:Poch}
\begin{center}
\unitlength 0.1in
\begin{picture}( 18.0000, 10.0000)(  7.0000,-13.0000)
%
{\color[named]{Black}{%
\special{pn 13}%
\special{pa 2400 300}%
\special{pa 800 300}%
\special{fp}%
}}%
%
{\color[named]{Black}{%
\special{pn 13}%
\special{pa 2400 1300}%
\special{pa 800 1300}%
\special{fp}%
}}%
%
{\color[named]{Black}{%
\special{pn 13}%
\special{ar 1200 800 200 200  0.0000000 6.2831853}%
}}%
%
{\color[named]{Black}{%
\special{pn 13}%
\special{ar 2000 800 200 200  0.0000000 6.2831853}%
}}%
%
{\color[named]{Black}{%
\special{pn 8}%
\special{ar 1200 800 300 300  1.5707963 4.7123890}%
}}%
%
{\color[named]{Black}{%
\special{pn 8}%
\special{ar 2000 800 300 300  4.7123890 6.2831853}%
\special{ar 2000 800 300 300  0.0000000 1.5707963}%
}}%
%
{\color[named]{Black}{%
\special{pn 8}%
\special{ar 1200 800 400 400  1.5707963 4.7123890}%
}}%
%
{\color[named]{Black}{%
\special{pn 8}%
\special{ar 2000 800 400 400  4.7123890 6.2831853}%
\special{ar 2000 800 400 400  0.0000000 1.5707963}%
}}%
%
{\color[named]{Black}{%
\special{pn 8}%
\special{pa 2000 1200}%
\special{pa 1200 1200}%
\special{fp}%
}}%
%
{\color[named]{Black}{%
\special{pn 8}%
\special{pa 1200 500}%
\special{pa 2000 500}%
\special{fp}%
}}%
%
{\color[named]{Black}{%
\special{pn 8}%
\special{pa 2000 1100}%
\special{pa 1800 1100}%
\special{fp}%
}}%
%
{\color[named]{Black}{%
\special{pn 8}%
\special{pa 1400 1100}%
\special{pa 1200 1100}%
\special{fp}%
}}%
%
{\color[named]{Black}{%
\special{pn 8}%
\special{ar 1800 1000 100 100  1.5707963 3.1415927}%
}}%
%
{\color[named]{Black}{%
\special{pn 8}%
\special{ar 1200 1000 500 600  4.7123890 6.2831853}%
}}%
%
{\color[named]{Black}{%
\special{pn 8}%
\special{ar 2000 1000 500 600  3.1415927 4.7123890}%
}}%
%
{\color[named]{Black}{%
\special{pn 8}%
\special{ar 1400 1000 100 100  6.2831853 6.2831853}%
\special{ar 1400 1000 100 100  0.0000000 1.5707963}%
}}%
%
{\color[named]{Black}{%
\special{pn 13}%
\special{ar 2400 800 100 500  4.7123890 6.2831853}%
\special{ar 2400 800 100 500  0.0000000 1.5707963}%
}}%
%
{\color[named]{Black}{%
\special{pn 13}%
\special{ar 800 800 100 500  1.5707963 4.7123890}%
}}%
\put(15.6000,-10.8000){\makebox(0,0)[lb]{$\gamma$}}%
\end{picture}%
\end{center}
\caption{the Pochhammer contour}
\end{figure}

We end this section with a conjecture about the characterization of simple closed curves.

\begin{conjecture}
Suppose $t_{\gamma}$ is realizable as a diffeomorphism.
Then $\gamma$ is homotopic to a power of a simple closed curve.
\end{conjecture}

\section{Classical theory revisited}
\label{sec:Revi}

In this section we reconsider the classical Torelli-Johnson-Morita theory 
for the surface $\Sigma = \Sigma_{g,1}$, $g > 0$, in \S \ref{sec:Cla}.
Let $\mathfrak{l}^+_g := \prod^\infty_{k=1}\mathfrak{h}^\mathbb{Z}_{g,1}(k)
\otimes\mathbb{Q}$ be the degree completion of the target of the Johnson 
homomorphisms tensored by the rationals $\mathbb{Q}$. 
Massuyeau's total Johnson map $\tau^\theta$ associated to a symplectic 
expansion $\theta$ of the group $\pi = \pi_1(\Sigma, *)$, 
introduced by Massuyeau \cite{Mas12}, 
is an embedding of the Torelli group $\mathcal{I}_{g,1}$ into the pro-nilpotent
Lie algebra $\mathfrak{l}^+_g$. 
We introduce a Lie subalgebra $L^+(\Sigma)$ of the completed Goldman 
Lie algebra $\widehat{\mathbb{Q}\hat\pi}(\Sigma)$, and decompose the 
map $\tau^\theta$ into a natural embedding of the group $\mathcal{I}_{g,1}$ 
into the pro-nilpotent Lie algebra $L^+(\Sigma)$ and 
an isomorphism of Lie algebras $-\lambda_\theta\colon L^+(\Sigma) 
\overset\cong\to \mathfrak{l}^+_g$
induced by a tensorial description of $\widehat{\mathbb{Q}\hat\pi}(\Sigma)$
through the symplectic expansion $\theta$. 
Then we can give an algebraic description of the geometric constraint on the Johnson image
in \S \ref{subsec:alg-Gold}.
In particular, we show that the Morita traces are recovered from this geometric constraint.
In the next section \S\ref{sec:Joh} we generalize some of these results 
to an arbitrary compact oriented surface with non-empty boundary.

\subsection{Symplectic expansions}
\label{subsec:Symp}

We begin by considering a tensorial description of the completed group ring
$\widehat{\mathbb{Q}\pi}$, where $\pi=\pi_1(\Sigma,*)$.
As in \S1.3, 
let $H=H_1(\Sigma;\mathbb{Q})=H_{\mathbb{Z}}\otimes_{\mathbb{Z}} \mathbb{Q}$ be the first
rational homology group of the surface $\Sigma$, and let $\widehat{T}:=\prod_{n=0}^{\infty}H^{\otimes n}$
be the completed tensor algebra generated by $H$. 
The algebra $\widehat{T}$ has the decreasing filtration
given by $\widehat{T}_p:=\prod_{n\ge p}H^{\otimes n}$, $p\ge 0$, 
and is a complete Hopf algebra
with the coproduct $\Delta$ given by $\Delta(X)=X\widehat{\otimes} 1+1\widehat{\otimes}X$ for $X\in H$,
and the antipode $\iota$ given by $\iota(X)=-X$ for $X\in H$.
Note that the subset $1+\widehat{T}_1$ is a subgroup 
of the multiplicative group of the algebra $\widehat{T}$.
For the rest of this chapter we omit the symbol $\otimes$ for the 
multiplication in the algebra $\widehat{T}$. For example, we denote 
$\omega = \sum^g_{i=1} A_iB_i - B_iA_i \in H_{\mathbb{Z}}^{\otimes 2}\subset
H^{\otimes 2} \subset \widehat{T}$ for the symplectic form 
(see \S \ref{subsec:J-i}).
Massuyeau \cite{Mas12} introduced the notion of a symplectic expansion, 
which is a group like expansion of the free group $\pi = \pi_1(\Sigma, *) 
\cong F_{2g}$ satisfying the symplectic condition (4) stated below.
Let $\zeta\in \pi$ be the based loop parallel to the {\it negatively}
oriented boundary of $\Sigma$.
\begin{definition}[Massuyeau \cite{Mas12}]
\label{def:Symp}
A map $\theta\colon \pi \to 1+\widehat{T}_1$ is called a {\it symplectic expansion} of $\pi$ if
\begin{enumerate}
\item The map $\theta\colon \pi \to 1+\widehat{T}_1$ is a group homomorphism.
\item For any $x\in \pi$ we have $\theta(x)\equiv 1+[x] {\rm \ mod\ }\widehat{T}_2$.
\item For any $x\in \pi$ the element $\theta(x)$ is group-like, i.e.,
$\Delta \theta(x)=\theta(x) \widehat{\otimes} \theta(x)$.
\item We have $\theta(\zeta)=\exp (\omega)=\sum_{n=0}^{\infty}(1/n!)\omega^{\otimes n}$.
\end{enumerate}
\end{definition}

Let $\widehat{\mathcal{L}}\subset \widehat{T}$ be the completed free Lie algebra
generated by $H$. In other words, $\widehat{\mathcal{L}}$ is the set of
primitive elements of the complete Hopf algebra $\widehat{T}$, 
$\widehat{\mathcal{L}} = \mathcal{L}(\widehat{T})$.
For a map $\theta$ satisfying the condition (1)(2), we denote $\ell^{\theta}:=\log \theta$.
If $\theta$ satisfies the condition (3), then $\ell^{\theta}$ takes values in $\widehat{\mathcal{L}}$.
In general, we have $\ell^{\theta}(\zeta)=\omega+({\rm higher\ term})$.
The condition (4) is equivalent to $\ell^{\theta}(\zeta)=\omega$, i.e.,
all the higher terms vanish, which we call the symplectic condition. 
This seems a quite severe condition, however, symplectic expansions do exist and there are infinitely many.
The first example (with real coefficients) was given by Kawazumi \cite{Ka06}
via iterated integrals and the Green operator, and Massuyeau \cite{Mas12} gave the second example using the LMO functor.
There is also a purely combinatorial construction by Kuno \cite{Ku12}.

Fix a symplectic expansion $\theta$. Then $\theta$ induces a
filter-preserving isomorphism $\theta\colon \widehat{\mathbb{Q}\pi}\overset{\cong}{\to} \widehat{T}$
of complete Hopf algebras. Moreover, from the condition (4) we have an
isomorphism $\theta\colon (\widehat{\mathbb{Q}\pi},\widehat{\mathbb{Q}\langle \zeta \rangle})
\to (\widehat{T},\mathbb{Q}[[\omega]])$ of pairs of complete Hopf algebras.
Here $\langle \zeta \rangle$ is the infinite cyclic group generated by
$\zeta$ and $\mathbb{Q}[[\omega]]$ is a formal power series ring generated
by a primitive element $\omega$.

\subsection{The Lie algebra of symplectic derivations}
\label{subsec:SD}

We recall the definition of the Lie algebra of symplectic derivations.
First we make a couple of remarks about the intersection form on the surface $\Sigma$.
The first homology group $H$ is equipped with a skew symmetric non-degenerate
bilinear form $(\ \cdot \ )\colon H\otimes H\to \mathbb{Q}$ called the {\it intersection form}.
We identify $H$ and its dual $H^*={\rm Hom}(H,\mathbb{Q})$ by
\begin{equation}
\label{eq:HH^*}
H\overset{\cong}{\to} H^*, \quad X\mapsto (Y\mapsto (Y\cdot X)).
\end{equation}
A {\it symplectic basis} of $H$ is a subset $\{ A_i,B_i\}_{i=1}^g\subset H$
satisfying $(A_i\cdot B_j)=\delta_{ij}$, $(A_i\cdot A_j)=(B_i\cdot B_j)=0$.
Then the symplectic form $\omega$ is given by $\omega=\sum_{i=1}^g A_iB_i-B_iA_i$.

By definition, the {\it Lie algebra of symplectic derivations}, denoted by $\mathfrak{a}_g^-$
or ${\rm Der}_{\omega}(\widehat{T})$, is
the Lie algebra of continuous derivations on the algebra $\widehat{T}$ annihilating $\omega$.
Since the algebra $\widehat{T}$ is generated by the degree 1 part as a complete algebra, the restriction
\begin{equation}
\label{eq:rest}
\mathfrak{a}_g^- \to {\rm Hom}(H,\widehat{T})=H^*\otimes \widehat{T}
= H\otimes \widehat{T}=\widehat{T}_1, \quad D\mapsto D|_H
\end{equation}
is injective. It is easy to show that the image coincides with
${\rm Ker}([\ ,\ ]\colon H\otimes \widehat{T}\to \widehat{T}_1)$.
In other words, $\mathfrak{a}_g^-$ is identified with the space of cyclically invariant
tensors. If we define a homogeneous $\mathbb{Q}$-linear map
$N\colon \widehat{T}\to \widehat{T}$ by
$N(X_1\cdots X_m)=\sum_{i=1}^m X_i\cdots X_mX_1\cdots X_{i-1}$,
for $m\ge 1$, $X_1,\ldots,X_m \in H$, and $N|_{H^{\otimes 0}}=0$, then we have an identification
\begin{equation}
\label{eq:agN}
\mathfrak{a}_g^-=\prod_{m=1}^{\infty} N(H^{\otimes m}).
\end{equation}
By a straightforward computation, we have the following.

\begin{lemma}
Under the identification {\rm (\ref{eq:agN})}, the Lie bracket of $\mathfrak{a}_g^-$
is given by
\begin{align*}
& [N(X_1\cdots X_m),N(Y_1\cdots Y_n)] \\
= & -\sum_{i,j} (X_i\cdot Y_j) N(X_{i+1}\cdots X_m X_1\cdots X_{i-1}
Y_{j+1}\cdots Y_n Y_1\cdots Y_{j-1}),
\end{align*}
where $m,n\ge 1$ and $X_i,Y_j\in H$.
\end{lemma}

The Lie subalgebra $\mathfrak{a}_g:=N(\widehat{T}_2)$ is nothing but (the completion of)
Kontsevich's ``associative" $a_g$ \cite{Kon93}.

Let $\mathfrak{l}_g$ be the Lie subalgebra of $\mathfrak{a}_g^-$ consisting of
derivations $D$ stabilizing the coproduct of $\widehat{T}$ in the sense that
$\Delta D=(D\widehat{\otimes}1+1\widehat{\otimes}D)\Delta$. This condition is equivalent
to $D(H)\subset \widehat{\mathcal{L}}$, thus
the restriction of (\ref{eq:rest}) to $\mathfrak{l}_g$ gives an identification
$\mathfrak{l}_g={\rm Ker}([\ ,\ ]\colon H\otimes \widehat{\mathcal{L}} \to \widehat{\mathcal{L}})$.
Moreover, we have ${\rm Ker}([\ ,\ ]\colon H\otimes \widehat{\mathcal{L}} \to \widehat{\mathcal{L}})
=N(\widehat{\mathcal{L}}\widehat{\otimes}\widehat{\mathcal{L}})$ (see \cite{KK1} Lemma 2.7.2). The Lie algebra
$\mathfrak{l}_g$ is the degree completion of Kontsevich's ``Lie" $\ell_g$ \cite{Kon93}.
It should be remarked that the Lie algebra $\ell_g$ was introduced earlier by Morita \cite{Mor89} \cite{MorICM}
as a target of the Johnson homomorphisms.
Let $\mathfrak{l}_g^+$ be the ideal of $\mathfrak{l}_g$ consisting of derivations $D$
such that $D(H)\subset \widehat{\mathcal{L}}\cap \widehat{T}_2$.
In fact, the Lie algebra $\mathfrak{l}_g^+$ is nothing but the completion of
the Lie algebra $\bigoplus_{k=1}^{\infty}\mathfrak{h}_{g,1}^{\mathbb{Z}}(k)\otimes \mathbb{Q}$
(see \S \ref{subsec:J-i}).

\subsection{Algebraic interpretation of the Goldman bracket}
\label{subsec:alg-Gold}

Let $\theta$ be a symplectic expansion of $\pi$. Consider the $\mathbb{Q}$-linear map
$\mathbb{Q}\pi \to N(\widehat{T}_1)=\mathfrak{a}_g^-$,
$\pi \ni x\mapsto N\theta(x)$. Since $N(uv)=N(vu)$ for any $u,v\in \widehat{T}$,
this map factors through to a map
$\lambda_{\theta}\colon \mathbb{Q}\hat{\pi}(\Sigma)\to \mathfrak{a}_g^-$.
Namely if $x\in \pi$, then $\lambda_{\theta}(|x|)=N\theta(x)$.

Using a symplectic expansion, we can relate the Goldman Lie algebra
with the Lie algebra of symplectic derivations.

\begin{theorem}[\cite{KK1}]
\label{thm:Ntheta}
\begin{enumerate}
\item
The map
$$-\lambda_{\theta}\colon \mathbb{Q}\hat{\pi}(\Sigma)
\to N(\widehat{T}_1)=\mathfrak{a}_g^-$$
is a filter preserving Lie algebra homomorphism. The kernel 
is spanned by the class of a constant loop $1$ and the image
is dense with respect to the $\widehat{T}_1$-adic topology.
\item
The following diagram commutes:
$$
\begin{CD}
\mathbb{Q}\hat{\pi}(\Sigma) \times \mathbb{Q}\pi @>{\sigma}>> \mathbb{Q}\pi \\
@V{-\lambda_{\theta}\times \theta}VV @V{\theta}VV\\
\mathfrak{a}_g^- \times \widehat{T} @>>> \widehat{T}
\end{CD}
$$
Here the bottom horizontal arrow is the action by derivations.
\end{enumerate}
\end{theorem}

Note that the minus sign comes from our convention about the isomorphism (\ref{eq:HH^*}).
From Theorem \ref{thm:Ntheta} (1) the map $\lambda_{\theta}$ induces a
filtered Lie algebra isomorphism $\widehat{\mathbb{Q}\hat{\pi}}(\Sigma)
\overset{\cong}{\to}\mathfrak{a}_g^-$.
As was shown in \cite{Ku12}, Theorem \ref{thm:Ntheta} holds also for 
any Magnus expansion satisfying the symplectic condition (4) 
in Definition \ref{def:Symp}. By definition, we have $\mathfrak{a}^-_g
= \{D \in {\rm Der}(\widehat{T}); \, D(\omega) = 0\} 
= \{D \in {\rm Der}(\widehat{T}); \, D(e^\omega) = 0\}$. 
Hence, from Theorem \ref{thm:Ntheta}, we have a filtration-preserving 
isomorphism
\begin{equation}
\sigma\colon \widehat{\mathbb{Q}\hat\pi}(\Sigma) \overset\cong\to 
{\rm Der}_\partial(\widehat{\mathbb{Q}\pi}),
\label{eq:sg-isom}
\end{equation}
This is Theorem \ref{thm:s-isom} for $(S, E) = (\Sigma, \{*\})$. \par
Now consider the Torelli group $\mathcal{I}_{g,1}$. For any $\varphi 
\in \mathcal{I}_{g,1}$, the logarithm of $\widehat{\sf DN}(\varphi)$ converges 
as an element of $F_1{\rm Der}(\widehat{\mathbb{Q}\pi})$, 
since ${\sf DN}(\varphi)
((I\pi)^m) \subset (I\pi)^{m+1}$ for any $m \geq 1$. From the fact $\varphi(\zeta) 
= \zeta \in \pi$ follows $\log\widehat{\sf DN}(\varphi) \in 
{\rm Der}_\partial(\widehat{\mathbb{Q}\pi})$. Hence we define the {\it geometric 
Johnson homomorphism} by 
$$
\tau:= \sigma^{-1}\circ\log\circ\widehat{\sf DN}\colon \mathcal{I}_{g,1} 
\to \widehat{\mathbb{Q}\hat\pi}(\Sigma)(3), \quad
\varphi \mapsto \sigma^{-1}(\log\widehat{\sf DN}(\varphi)).
$$
We remark $\widehat{\sf DN}(\varphi) = e^{\sigma(\tau(\varphi))}$ 
for any $\varphi \in \mathcal{I}_{g,1}$. Hence, if $\varphi$ is the right handed 
Dehn twist along a separating simple closed curve $C \subset \Sigma$, 
then we have $\tau(t_C) = L(C)$ by Theorem \ref{thm:logDT}. 
\par
Recall that the action of the group $\mathcal{M}_{g,1}$ on 
$\widehat{\mathbb{Q}\pi}$ preserves the coproduct $\Delta$. 
Hence $\tau(\mathcal{I}_{g,1})$ is included in the stabilizer of $\Delta$, 
which we denote
\begin{align}
& L(\Sigma) := \{u \in \widehat{\mathbb{Q}\hat\pi}(\Sigma); 
(\sigma(u)\widehat{\otimes}\sigma(u))\Delta = \Delta \sigma(u)\}, \quad
\text{and}\nonumber\\
& L^+(\Sigma) := L(\Sigma) 
\cap \widehat{\mathbb{Q}\hat\pi}(\Sigma)(3).\nonumber
\end{align}
The Lie algebra $L^+(\Sigma)$ is pro-nilpotent, so that the 
Hausdorff series define a natural group structure on it. 
Hence the geometric Johnson homomorphism
$$
\tau\colon \mathcal{I}_{g,1} \to L^+(\Sigma)
$$
is an injective group homomorphism.\par
On the other hand, recall that $\mathfrak{l}_g = \prod^\infty_{k=0}
\mathfrak{h}^\mathbb{Z}_{g,1}(k)\otimes\mathbb{Q}$ and 
$\mathfrak{l}^+_g = \prod^\infty_{k=1}
\mathfrak{h}^\mathbb{Z}_{g,1}(k)\otimes\mathbb{Q}$
are exactly the stabilizer of $\Delta$ in $\mathfrak{a}_g^-$ 
and  $\mathfrak{a}_g^+:= N(\widehat{T}_3)$, respectively. 
Since $\theta\colon \widehat{\mathbb{Q}\pi}\overset\cong\to 
\widehat{T}$ is a filtration-preserving isomorphism of complete 
Hopf algebras, the isomorphism $-\lambda_\theta$ induces 
isomorphisms $-\lambda_\theta\colon L^+(\Sigma) 
\overset\cong\to\mathfrak{l}^+_g$ and $-\lambda_\theta\colon L(\Sigma) 
\overset\cong\to\mathfrak{l}_g$. From the construction, 
Massuyeau's total Johnson map $\tau^\theta$ is decomposed as
$$
\tau^\theta = -\lambda_\theta\circ\tau\colon \mathcal{I}_{g,1} \to 
L^+(\Sigma) \overset\cong\to\mathfrak{l}^+_g.
$$
In particular, the graded quotient of the geometric Johnson homomorphism 
$\tau$ is the totality of the (original) Johnson homomorphims. \par
\par
Our geometric re-construction of the Johnson homomorphisms
leads us to finding a geometric constraint of the Johnson image. 
We look at the map $\mu$ as in \S \ref{subsec:int}.
It is clear that the action of any $\varphi \in \mathcal{I}_{g,1}$ preserves the map 
$\mu\colon \widehat{\mathbb{Q}\pi} \to \widehat{\mathbb{Q}\pi}\widehat{\otimes}
\widehat{\mathbb{Q}\hat\pi}(\Sigma)$. Hence $\mu(e^{n\sigma(\tau(\varphi))}v)
= (e^{n\sigma(\tau(\varphi))}\widehat{\otimes}e^{n\sigma(\tau(\varphi))})
\mu(v)$ for any $n \in \mathbb{Z}$ and $v \in \widehat{\mathbb{Q}\pi}$. 
Taking the linear terms in $n$, we have 
$\mu(\sigma(\tau(\varphi))v) = \sigma(\tau(\varphi))\mu(v)$. 
This means $(\overline{\sigma}\otimes1)(v\otimes \delta(\tau(\varphi))) = 0$
from Theorem \ref{thm:bim} (3), and $\delta(\tau(\varphi)) = 0$ from the 
isomorphism (\ref{eq:sg-isom}). Hence we obtain the following theorem.
\begin{theorem}[\cite{KK4}] \label{thm:cb}
$$
\delta\circ\tau  = 0\colon \mathcal{I}_{g,1} \to L^+(\Sigma) \to 
\widehat{\mathbb{Q}\hat\pi}(\Sigma)\widehat{\otimes}
\widehat{\mathbb{Q}\hat\pi}(\Sigma).
$$
\end{theorem}

\subsection{The Turaev cobracket and the Morita trace}
\label{subsec:Tc-Mo}

From Theorem \ref{thm:Ntheta}, the space $\mathfrak{a}_g^-$
has a structure of a complete Lie bialgebra with a ($\theta$-dependent) Lie cobracket
$\delta^{\theta}:=((-\lambda_{\theta})\widehat{\otimes}(-\lambda_{\theta}))\circ \delta \circ(-\lambda_{\theta})^{-1}$.
In this subsection we shall study the Laurent expansion of $\delta^{\theta}$.
The key ingredients is a tensorial description of the homotopy intersection form (see \S \ref{subsec:htpy})
due to Massuyeau and Turaev \cite{MT11}.
As in \S \ref{subsec:Comp}, we see that $\eta$ extends to
$\eta\colon \widehat{\mathbb{Q}\pi} \widehat{\otimes} \widehat{\mathbb{Q}\pi} \to \widehat{\mathbb{Q}\pi}$.
Let $\varepsilon\colon \widehat{T} \to \widehat{T}/\widehat{T}_1 = \mathbb{Q}$ be the augmentation map.
Define a $\mathbb{Q}$-bilinear map $\overset{\bullet}{\leadsto}\colon \widehat{T}_1
\widehat{\otimes} \widehat{T}_1 \to \widehat{T}$ by 
$$
X_1\cdots X_m\overset{\bullet}{\leadsto} Y_1\cdots Y_n := (X_m\cdot Y_1) X_1\cdots X_{m-1}Y_2\cdots Y_n \in H^{\otimes m+n-2}
$$
for any $m$, $n \geq 1$, and $X_i$, $Y_j \in H$. Here $(X_m\cdot Y_1) \in \mathbb{Q}$ is the intersection pairing of $X_m$ and $Y_1 \in H$. 
A $\mathbb{Q}$-linear map $\rho\colon \widehat{T}\widehat{\otimes}\widehat{T} \to \widehat{T}$ is defined by 
\begin{equation}
\rho(a\widehat{\otimes}b)=\rho(a, b) := (a -\varepsilon(a))\overset\bullet\leadsto(b -\varepsilon(b)) 
+ (a -\varepsilon(a))s(\omega)(b -\varepsilon(b))
\label{eq:rho}
\end{equation}
for any $a$ and $b \in \widehat{T}$, where $s(z)$ is the formal power series
$$
s(z) = \frac{1}{e^{-z}-1} + \frac{1}{z}
= -\frac{1}{2} - \sum_{k\geq 1}\frac{B_{2k}}{(2k)!}z^{2k-1} = 
 -\frac{1}{2}  -\frac{z}{12}  +\frac{z^3}{720}  -\frac{z^5}{30240} + \cdots.
$$
Here $B_{2k}$'s are the Bernoulli numbers.
 
\begin{theorem}[\cite{MT11}]
\label{thm:MT}
Let $\theta\colon \pi \to \widehat{T}$ be a symplectic expansion. Then the
following diagram commutes: 
$$
\begin{CD}
\widehat{\mathbb{Q}\pi}\widehat{\otimes} \widehat{\mathbb{Q}\pi}
@>{\eta}>> \widehat{\mathbb{Q}\pi}\\
@V{\theta\widehat{\otimes}\theta}VV @V{\theta}VV\\
\widehat{T}\widehat{\otimes}\widehat{T} @>{\rho}>> \widehat{T}.
\end{CD}$$
\end{theorem}

Let us fix a symplectic expansion $\theta$. We define $\mathbb{Q}$-linear maps
$\kappa^{\theta}\colon \widehat{T}\widehat{\otimes}\widehat{T} \to \widehat{T}\widehat{\otimes}\widehat{T}$
and $\mu^{\theta}\colon \widehat{T}\to \widehat{T}\widehat{\otimes} \mathfrak{a}_g^-$
so that the diagrams
$$
\begin{CD}
\widehat{\mathbb{Q}\pi}\widehat{\otimes} \widehat{\mathbb{Q}\pi}
@>{\kappa}>> \widehat{\mathbb{Q}\pi} \widehat{\otimes} \widehat{\mathbb{Q}\pi}\\
@V{\theta\widehat{\otimes}\theta}VV @V{\theta\widehat{\otimes}\theta}VV\\
\widehat{T}\widehat{\otimes}\widehat{T} @>{\kappa^{\theta}}>> \widehat{T} \widehat{\otimes}\widehat{T}.
\end{CD}
\quad {\rm and} \quad
\begin{CD}
\widehat{\mathbb{Q}\pi}
@>{\mu}>> \widehat{\mathbb{Q}\pi} \widehat{\otimes} \widehat{\mathbb{Q}\hat{\pi}}\\
@V{\theta}VV @V{-\theta\widehat{\otimes}\lambda_{\theta}}VV\\
\widehat{T} @>{\mu^{\theta}}>> \widehat{T} \widehat{\otimes}\mathfrak{a}_g^-.
\end{CD}
$$
commute. From Proposition \ref{prop:k-eta} and Theorem \ref{thm:MT}, the map $\kappa=\kappa^{\theta}$
does not depend on the choice of $\theta$. Explicitly, for $X,Y\in H$ we have
\begin{align}
\kappa(X\otimes Y)
&=-(1\widehat{\otimes}1)((1\widehat{\otimes}\iota)\Delta \rho(X,Y))(1\widehat{\otimes}1) \nonumber \\
&=-(X\cdot Y)(1\widehat{\otimes}1)-(1\widehat{\otimes}\iota)\Delta(Xs(\omega)Y).
\label{eq:alg-k}
\end{align}
On the other hand, the map $\mu^{\theta}$ depends on the choice of $\theta$.
By Proposition \ref{eq:mu-pro}, for any $m \geq 0$ and $X_i \in H$, $1\le i\le m$, we have
\begin{align}
& \mu^\theta(X_1\cdots X_m)\nonumber \\
=& (1\widehat{\otimes}(-N))\sum_{i<j}(X_1\cdots X_{i-1}\widehat{\otimes}1)
\kappa(X_i, X_j)(X_{j+1}\cdots X_m\widehat{\otimes}X_{i+1}\cdots X_{j-1})\nonumber \\
& + \sum^m_{i=1}(X_1\cdots X_{i-1}\widehat{\otimes}1)\mu^\theta(X_i)(X_{i+1}\cdots X_m\widehat{\otimes}1).
\label{eq:mu-exp}
\end{align}


We consider the Laurent expansion of $\mu^{\theta}$.
We denote by $\mu^{\theta}_{(k)}$ the degree $k$ part of $\mu^{\theta}$.
In other words, we have 
$$
\mu^\theta(X_1\cdots X_m) = \sum^\infty_{k=-\infty}
\mu^{\theta}_{(k)}(X_1\cdots X_m), \quad \mu^{\theta}_{(k)}(X_1\cdots X_m)
\in H^{\otimes(m+k)}
$$
for $X_i \in H$, $1\le i\le m$. 
We define the homogeneous $\mathbb{Q}$-linear map
$\mu^{\rm alg}\colon \widehat{T}\to \widehat{T}\widehat{\otimes} \mathfrak{a}_g^-$ by
$$\mu^{\rm alg}(X_1\cdots X_m)=
\sum_{i<j}(X_i\cdot X_j)X_1\cdots X_{i-1}X_{j+1}\cdots X_m\widehat{\otimes}N(X_{i+1}\cdots X_{j-1})$$
for $m\ge 0$ and $X_i\in H$, $1\le i\le m$.
Looking at (\ref{eq:mu-exp}) and (\ref{eq:alg-k}) in detail, we have the following.
As is announced in \cite{MT12} Remark 7.4.3, this result is obtained 
independently by Massuyeau and Turaev \cite{MT13}. 

\begin{theorem}[\cite{KK4}\cite{MT13}]
\label{thm:mu-exp}
For any $u\in \widehat{T}$ we have
\begin{enumerate}
\item $\mu^{\theta}_{(k)}(u)=0$ for $k\le-3$ and $k=-1$,
\item $\mu^{\theta}_{(-2)}(u)=\mu^{\rm alg}(u)$, and
\item $\mu^{\theta}_{(0)}(u)=(-1/2)(1\widehat{\otimes} N(u))$.
\end{enumerate}
\end{theorem}

Therefore, for any $u\in \widehat{T}$ we can write 
$$\mu^{\theta}(u)=\mu^{\rm alg}(u)-\frac{1}{2}(1\widehat{\otimes} N(u))+\mu^{\theta}_{(1)}(u)+
\mu^{\theta}_{(2)}(u)+\cdots.$$
In general, the higher terms $\mu^{\theta}_{(k)}(u)$, $k\ge 1$, do depend on the choice of $\theta$. \par

Now we consider the Lie cobracket $\delta^{\theta}$. We
denote by $\delta^{\theta}_{(k)}$ the degree $k$ part of $\delta^{\theta}$, i.e., 
$$
\delta^\theta(N(X_1\cdots X_m)) = \sum^\infty_{k=-\infty}
\delta^{\theta}_{(k)}(N(X_1\cdots X_m)),
$$
and $\delta^{\theta}_{(k)}(N(X_1\cdots X_m))
\in N(H^{\otimes(m+k)})$ for $X_i \in H$, $1\le i\le m$. 
We define the homogeneous $\mathbb{Q}$-linear map $\delta^{\rm alg}\colon \mathfrak{a}_g^-
\to \mathfrak{a}_g^- \widehat{\otimes} \mathfrak{a}_g^-$ by
\begin{align*}
& \delta^{\rm alg}(N(X_1\cdots X_m)) \\
= & -\sum_{i<j} (X_i\cdot X_j)
\left\{ \begin{array}{c}
 N(X_{i+1}\cdots X_{j-1})\widehat{\otimes}
N(X_{j+1}\cdots X_mX_1\cdots X_{i-1}) \\
-N(X_{j+1}\cdots X_mX_1\cdots X_{i-1})\widehat{\otimes} N(X_{i+1}\cdots X_{j-1})
\end{array} \right\}
\end{align*}
for $m\ge 1$ and $X_i\in H$, $1\le i\le m$.
We call $\delta^{\rm alg}$ {\it Schedler's cobracket} since it was introduced by Schedler \cite{Sch05}.
By Proposition \ref{prop:mu-d},
and Theorem \ref{thm:mu-exp}, for any $u\in \mathfrak{a}^-_g$ we have
$$\delta^{\theta}(u)=\delta^{\rm alg}(u)+\delta^{\theta}_{(1)}(u)+\delta^{\theta}_{(2)}(u)+\cdots,$$
and in general the higher terms $\delta^{\theta}_{(k)}(u)$, $k\ge 1$, depend on the choice of $\theta$.
As a corollary of Theorem \ref{thm:cb}, we have 
$$
\delta^{\rm alg} \circ \tau = 0\colon \bigoplus_{k=1}^{\infty} {\rm gr}^k(\mathcal{I}_{g,1})
\to \bigoplus_{k=1}^{\infty}\mathfrak{h}_{g,1}^{\mathbb{Z}}(k)
\to \mathfrak{a}^-_g\widehat{\otimes}\mathfrak{a}^-_g. 
$$
\par

Finally we show that Schedler's cobracket $\delta^{\rm alg}$ restricted to 
$\mathfrak{l}_g^+$
recovers the Morita traces of all degrees ${\rm Tr}_k\colon (\mathfrak{l}_g^+)_{(k+2)} := \mathfrak{h}^\mathbb{Z}_{g,1}(k)\otimes\mathbb{Q} \to S^{k}H$,
$k \geq 3$ (see \S \ref{subsec:J-i}). Here $S^{k}H$ is the $k$-th symmetric power of $H$.
Let $p_1\colon \mathfrak{a}_g^- = \prod^\infty_{m=1} N(H^{\otimes m}) \to N(H^{\otimes 1}) = H$
be the first projection, $i\colon \mathfrak{a}_g^- = \prod^\infty_{m=1} N(H^{\otimes m})
\hookrightarrow \prod^\infty_{m=1} H^{\otimes m} = \widehat{T}_1$ the inclusion map, and
$\varpi\colon \widehat{T} \to \widehat{\rm Sym}(H) := \prod^\infty_{m=0}S^m(H)$
the natural projection. We define 
$$\mathfrak{s} := \varpi\circ (p_1\widehat{\otimes}i)\colon
\mathfrak{a}_g^-\widehat{\otimes}\mathfrak{a}_g^- \to H\otimes \widehat{T}_1 =
\widehat{T}_2 \to \widehat{\rm Sym}(H).$$
By some straightforward computation we obtain the following.
\begin{theorem}[\cite{KK4}]
\label{54trace}
For any $k \ge 3$, we have
$$\mathfrak{s}\circ\delta^{\rm alg}\vert_{(\mathfrak{l}_g^+)_{(k+2)}} = (-k)\times{\rm Tr}_{k}
\colon (\mathfrak{l}_g^+)_{(k+2)}=\mathfrak{h}^\mathbb{Z}_{g,1}(k)\otimes\mathbb{Q} \to S^kH.$$ 
\end{theorem}
Thus all the Morita traces are derived from the geometric fact that 
any diffeomorphism preserves the self-intersection of any curve 
on the surface. 
Very recently Enomoto \cite{Eno} proved that the Enomoto-Satoh traces 
\cite{ES} are inside of Schedler's cobracket $\delta^{\rm alg}$. 
But we do not know whether they are inside of the Turaev cobracket 
$\delta^\theta$ itself or not.

\section{Compact surfaces with non-empty boundary}
\label{sec:Joh}

Let $\Sigma_{g,r}$ be a compact connected oriented surface of genus $g$
with $r$ boundary components. Now we generalize some of the results in \S\ref{sec:Revi} 
to such a surface with $r>0$. 
Throughout this section we work over the rationals $\mathbb{Q}$. 
In particular, we denote by $H_*(X,A)$ the rational homology group 
$H_*(X, A; \mathbb{Q})$ for any pair of spaces $(X, A)$.
Let $S=\Sigma_{g,n+1}$ for some $g$ and $n\ge 0$. Note that 
the interior of the surface $S$ has a complete hyperbolic structure. 
We number the components of the boundary 
$\partial S = \coprod^n_{j=0}\partial_jS$. Choose one point $*_j$ from each 
$\partial_jS$ to form a finite set $E := \{*_j\}^n_{j=0} \subset\partial S$. 
Gluing $(n+1)$ copies of the $2$-disks on the surface $S$ along the boundary, 
we obtain a closed surface $\overline{S} \cong \Sigma_g$. \par
In \S\ref{subsec:assoc} we construct an analogous Lie algebra associated 
to the surface $S$ for the degree completion $\mathfrak{a}^-_g$ 
of an enhancement of
Kontsevich's associative $a_g$, where we need an 
additional data on the inclusion homomorphism 
$H_1(S) \to H_1(\overline{S})$. In \S\ref{subsec:tensor} we introduce 
a Magnus expansion of the groupoid $\Pi S\vert_E$ satisfying some 
boundary condition, which induces an isomorphism of Lie algebras
from the completed Goldman Lie algebra onto the Lie algebra 
constructed in \S \ref{subsec:assoc}. As a consequence of the 
isomorphism, we obtain Theorem \ref{thm:s-isom}. The geometric 
Johnson homomorphism on the largest Torelli group of 
$S$ in the sense of Putman \cite{Pu} is defined to be an embedding 
of the Torelli group into some pro-nilpotent Lie subalgebra of 
$\widehat{\mathbb{Q}\hat\pi}(S)$ in \S\ref{subsec:geocon}. 
Its image is included in the kernel of the Turaev cobracket
in a similar way to $\Sigma_{g,1}$. \par

\subsection{The ``associative" Lie algebra for a compact surface}
\label{subsec:assoc}

We use similar notation to that in \S \ref{subsec:Ext} and \S \ref{sec:Revi}. 
Let $H$ be a finite-dimensional $\mathbb{Q}$ vector space. 
Then let $\widehat{T}(H) := \prod^\infty_{m=0} H^{\otimes m}$ be
the completed tensor algebra over $H$, 
and $N = N^H\colon \widehat{T}(H) \to \widehat{T}(H)$ 
the {\it cyclic symmetrizer} or the {\it cyclicizer} defined by 
$N\vert_{H^{\otimes 0}} := 0$ and $N(X_1\cdots X_m) 
:= \sum^m_{i=1} X_i\cdots X_m X_1\cdots X_{i-1}$ for 
$m \geq 1$ and $X_i \in H$. As in \S\ref{sec:Revi} 
we omit the symbol $\otimes$ if it means the product in the 
algebra $\widehat{T}(H)$. 
The filtered $\mathbb{Q}$-vector space $N(\widehat{T}(H)_1)$ is an analogue 
of the Lie algebra $\mathfrak{a}^-_g$, the degree completion of 
an enhancement of Kontsevich's ``associative" Lie algebra $a_g$. 
The algebra $\widehat{T}(H)$ is filtered by the two-sided ideals 
$\widehat{T}(H)_p := \prod^\infty_{m\ge p}H^{\otimes m}$, $p \geq 1$,
and constitutes a complete Hopf algebra 
whose coproduct $\Delta\colon \widehat{T}(H) \to \widehat{T}(H)
\widehat{\otimes}\widehat{T}(H)$ is given by 
$\Delta(X) = X\widehat{\otimes}1
+ 1\widehat{\otimes}X$ for any $X \in H$. 
\par
Let $S$ be a surface as above. Then the fundamental group $\pi_1(S)$ 
is a free group of finite rank, where we may choose any point in $S$. 
The reason why we introduce a Lie algebra structure on the filtered 
$\mathbb{Q}$-vector space $N(\widehat{T}(H_1(S))_1)$ 
comes from the following proposition. 
\begin{proposition}[\cite{KK3} Corollary 4.3.5.]\label{prop:Nisom}
For any group-like expansion $\theta\colon \pi_1(S) \to \widehat{T}(H_1(S))$, 
the map
$N\theta\colon {\mathbb{Q}\hat\pi}(S) \to N(\widehat{T}(H_1(S))_1)$, 
given by $(N\theta)(\vert x\vert) := N(\theta(x))$ for any $x \in \pi_1(S)$, 
is injective, and induces an isomorphism of filtered $\mathbb{Q}$-vector space
$$
N\theta\colon \widehat{\mathbb{Q}\hat\pi}(S) \overset\cong\to N(\widehat{T}(H_1(S))_1).
$$
\end{proposition}
\begin{proof} We simply write $\pi = \pi_1(S)$ and $\hat\pi = \hat\pi(S)$. 
Let $\mathbb{Q}\pi^c$ and $\widehat{\mathbb{Q}\pi}^c$ 
be the group ring of the group $\pi$ and its completion, respectively, 
which we regard as left $\mathbb{Q}\pi$-modules 
by conjugation of the group $\pi$. 
Since the interior of $S$ has a complete hyperbolic structure, 
we have a natural injective map
$\lambda\colon \mathbb{Q}\hat\pi/\mathbb{Q}1 \to H_1(\pi; \mathbb{Q}\pi^c)$
introduced in Proposition 3.4.3 \cite{KK1}. 
The completion map 
$\mathbb{Q}\pi^c \to \widehat{\mathbb{Q}\pi}^c$ is injective, so that 
the induced map $H_1(\pi; \mathbb{Q}\pi^c)\to 
H_1(\pi; \widehat{\mathbb{Q}\pi}^c)$ is also injective, since $\pi$ is free and so 
$H_2(\pi; \widehat{\mathbb{Q}\pi}^c/\mathbb{Q}\pi^c) = 0$. The group-like 
expansion $\theta$ induces an isomorphism $H_1(\pi; 
\widehat{\mathbb{Q}\pi}^c) \overset\cong\to N(\widehat{T}(H_1(S))_1)$
in the context of twisted homology of (complete) Hopf algebras. 
By straightforward computation in Lemma 5.3.2 \cite{KK1} using the group-like condition of $\theta$, we can prove that the composite of these maps equals 
the map $N\theta$. In particular, $N\theta\colon \mathbb{Q}\hat{\pi}^{\prime}(S) \to 
N(\widehat{T}(H_1(S))_1)$ is injective. Clearly the image of $N\theta$ 
is dense in $N(\widehat{T}(H_1(S))_1)$, and 
$N(\widehat{T}(H_1(S))_1)$ is complete with respect to the filtration 
$\{N(\widehat{T}(H_1(S))_m\}^\infty_{m=1}$. 
As was proved in Lemma 4.3.3 \cite{KK3}, 
$(N\theta)^{-1}(N(\widehat{T}(H_1(S))_m)) = \vert\mathbb{Q}1 + I\pi^m\vert$
for any $m \geq 1$. 
This proves the rest of the assertions of the proposition.
\end{proof}
Through the isomorphism $N\theta$ we can consider a Lie algebra 
structure on $N(\widehat{T}(H_1(S)_1)$. 
In this subsection we will give a candidate for such a structure. 
As will be stated in the next subsection, the map $-N\theta$ is an isomorphism 
of Lie algebras for some expansion $\theta$. \par
Recall $S \cong \Sigma_{g, n+1}$, $\partial S = \coprod^n_{j=0}\partial_jS$
and $\overline{S}$ is a closed surface of genus $g$ obtained from capping 
the boundary components of $S$ by $(n+1)$ disks.
The first homology group $H_1(\overline{S})$ is a symplectic vector space 
of dimension $2g$. We denote by $\overline{\omega} \in 
H_1(\overline{S})^{\otimes 2}$ the symplectic form on it. 
If $\{\overline{A_i}, \overline{B_i}\}^g_{i=1} \subset H_1(\overline{S})$ is 
a symplectic basis, then we have $\overline{\omega} = \sum^g_{i=1}
\overline{A_i}\overline{B_i} -\overline{B_i}\overline{A_i}$. 
Let $C_j \in H_1(S)$ be the 
homology class of a boundary loop on $\partial_jS$ in the positive direction. 
Consider the inclusion map $i\colon S \hookrightarrow \overline{S}$. 
In the homology exact sequence 
\begin{equation}
H_2(\overline{S}, S) \overset{\partial_*}\longrightarrow H_1(S) \overset{i_*}
\longrightarrow
H_1(\overline{S}) \longrightarrow 0,
\label{eq:1exact}
\end{equation}
the set $\{C_j\}^n_{j=1}$ is a basis of the image ${\rm Im}(\partial_*)$. 
To define our Lie bracket on $N(\widehat{T}_1(H_1(S)))$, 
we need to choose a section of the surjection $i_*\colon H_1(S) \to 
H_1(\overline{S})$. We denote by ${\rm Sect}(i_*)$ the set of all
sections of the surjection $i_*$. 
Any complex structure of the surface $\overline{S}$ defines 
a canonical $\mathbb{R}$-valued section 
of the surjection $i_*\colon H_1(S; \mathbb{R}) \to H_1(\overline{S}; \mathbb{R})$ 
induced by the normalized Abelian integrals of the third kind
if we regard the interior of $S$ as a punctured Riemann surface. 
\par
Fix a section $s \in {\rm Sect}(i_*)$. 
We denote $A^s_i := s(\overline{A_i})$, $B^s_i := s(\overline{B_i}) \in H_1(S)$
and $\omega_s := s^{\otimes2}(\overline{\omega}) = \sum^g_{i=1}A^s_iB^s_i
- B^s_iA^s_i \in H_1(S)^{\otimes 2}$. The set $\{A^s_i, B^s_i\}^g_{i=1}\cup 
\{C_j\}^n_{j=1}$ is a basis of $H_1(S)$. Let 
$u = \sum^g_{i=1}A^s_iu'_i + \sum^g_{i=1}B^s_iu''_i + \sum^n_{j=1}C_ju^0_j$ 
and 
$v = \sum^g_{i=1}A^s_iv'_i + \sum^g_{i=1}B^s_iv''_i + \sum^n_{j=1}C_jv^0_j$ 
be elements of $N(\widehat{T}(H_1(S))_1)$, where 
$u'_i, u''_i, u^0_j, v'_i, v''_i, v^0_j \in \widehat{T}(H_1(S))$. 
Then a bracket $[u, v] = [u, v]_s$ of $u$ and $v$ is defined by 
\begin{equation}
[u, v]_s := -N\left(\sum^g_{i=1} u'_iv''_i - u''_iv'_i 
+ \sum^n_{j=1}C_j(u^0_jv^0_j - v^0_ju^0_j)\right) \in N(\widehat{T}(H_1(S))_1). 
\label{eq:1br}
\end{equation}
It is easy to prove that the bracket does not depend on the choice of 
the symplectic basis $\{\overline{A_i}, \overline{B_i}\}^g_{i=1}$ and 
satisfies the Jacobi identity. 
We denote by $N(\widehat{T}_1)_s = N(\widehat{T}(H_1(S))_1)_s$ 
the Lie algebra $N(\widehat{T}(H_1(S))_1)$ equipped with 
the Lie bracket $[\,,\,]_s$. 
As Massuyeau and Turaev \cite{MT12} \cite{MT13} point out, 
this Lie algebra structure is related to quiver theory. 
If $g=0$ or $n=0$, then the set ${\rm Sect}(i_*)$ is a singleton, 
and $N(\widehat{T}_1)_s$ is just the Lie algebra of special derivations 
of the algebra $\widehat{T}$ or that of symplectic derivations, 
$\mathfrak{a}_g^-$, respectively. \par
For any compact connected oriented surface $S$ 
the map $-N\theta$ is a Lie algebra isomorphism
if an expansion $\theta$ satisfies some condition, 
which will be formulated in the next subsection. \par

\subsection{A tensorial description of the Goldman Lie algebra}
\label{subsec:tensor}

Let $(S,E)$ be as above. We begin this subsection 
by introducing the notion of a Magnus expansion of 
the groupoid $\Pi S\vert_E$. 
In this subsection we simply write $H = H_1(S)$, $\widehat{T} 
= \widehat{T}(H_1(S))$ and so on. 
If $M$ is a monoid, then we denote by $M_E$ the small category such that the set of objects 
is $E$, and the set of morphisms from $*_a$ to $*_b$, $0 \leq a,b \leq n$, 
is defined by $(M_E)(*_a, *_b) := M$. The composite and the unit 
in the monoid $M$ induce the composite and the units in the category 
$M_E$. For example, we consider the small category $\widehat{T}_E$ 
and the groupoid $(1+\widehat{T}_1)_E$ over the set $E$. 
\begin{definition}\label{def:2Magnus}
A homomorphism $\theta\colon \Pi S\vert_E \to (1+\widehat{T}_1)_E$ of groupoids 
over the set $E$ is a {\it Magnus expansion} of the pair $(S, E)$, 
if the restriction to $\pi_1(S, *_a)$, $\theta\colon \pi_1(S, *_a) \to 1+\widehat{T}_1$, 
is a Magnus expansion in Definition \ref{def:Magnus} 
for any $a$, $0 \leq a \leq n$. 
\end{definition}
It is easy to show that a homomorphism $\theta\colon \Pi S\vert_E 
\to (1+\widehat{T}_1)_E$ is a Magnus expansion if 
$\theta\colon \pi_1(S, *_b) \to 1+\widehat{T}_1$ 
is a Magnus expansion of the free group $\pi_1(S, *_b)$
for some $b$, $0 \leq b\leq n$. 
Any Magnus expansion $\theta$ induces an isomorphism
$\theta\colon \widehat{\mathbb{Q}\Pi S\vert_E} \overset\cong\longrightarrow 
\widehat{T}_E$. Hence, for any two Magnus 
expansions $\theta$ and $\theta'$, there exists a unique 
derivation $\hat u_0 \in F_1{\rm Der}(\widehat{T}_E)$ such that 
$\theta' = (\exp \hat u_0)\circ\theta\colon 
\Pi S\vert_E \to (1+\widehat{T}_1)_E$. 
Here $F_1{\rm Der}(\widehat{T}_E)$ is the Lie subalgebra of all derivations 
$D$ increasing the filtration on $\widehat{T}_E$ strictly (see \S \ref{subsec:DA}). 
A group-like expansion is defined to be a Magnus expansion whose target 
is reduced to ${\rm Gr}(\widehat{T})_E$, where ${\rm Gr}(\widehat{T})$ 
is the set of group-like elements in the complete Hopf algebra $\widehat{T}$. 
\par
Let $s \in {\rm Sect}(i_*)$ be a section of the surjection $i_*$ as in \S \ref{subsec:assoc}. 
Then we introduce some boundary condition on a Magnus expansion $\theta$ 
with respect to the section $s$. 
Let $\xi_j \in \pi_1(S, *_j)$, $0 \leq j \leq n$, be a based 
boundary loop along $\partial_jS$ in the positive direction.
We define the boundary condition with respect to the section $s$,
which we denote by $(\sharp_s)$, by
\begin{equation}
\theta(\xi_j) = \begin{cases}
\exp(- \omega_s + C_0) = \exp(- \omega_s - \sum^n_{j=1}C_j), 
& \text{if $j=0$,}\\
\exp(C_j), & \text{if $1 \leq j \leq n$.}
\end{cases}
\label{eq:2bdry}
\end{equation}
A group-like expansion satisfying the condition $(\sharp_s)$ 
is a generalization of a symplectic expansion.
If we fix a complex structure on the surface $\overline{S}$ and 
regard the interior of $S$ as a punctured Riemann surface, 
then we can construct 
a canonical  $\mathbb{R}$-valued group-like expansion 
satisfying the condition $(\sharp_s)$ 
with respect to the canonical section stated above by a similar way 
to \cite{Ka06}. \par
Now we can state the following theorem.
\begin{theorem}[\cite{KK5}\cite{MT13}]\label{thm:2br}
If a Magnus expansion $\theta$ satisfies the condition $(\sharp_s)$, 
then the map
$$
-N\theta\colon \widehat{\mathbb{Q}\hat\pi}(S) \to N(\widehat{T}_1)_s
$$
is a Lie algebra isomorphism. 
\end{theorem}
\par
Recall the natural action of 
$\widehat{\mathbb{Q}\hat\pi}(S)$ on the completion 
$\widehat{\mathbb{Q}\Pi S\vert_E}$ and 
the Lie algebra homomorphism
$$
\sigma\colon \widehat{\mathbb{Q}\hat\pi}(S) \to {\rm Der}_\partial
(\widehat{\mathbb{Q}\Pi S\vert_E})
$$
stated in \S\ref{subsec:Comp}. Theorem \ref{thm:s-isom} asserts 
that the map $\sigma$ is an isomorphism.
Here we remark that $\sigma$ does {\it not} preserve the filtrations 
if $S \neq \Sigma_{g,1}$. This seems to be related to the diversity of 
the Torelli groups \cite{Pu}. 
A homomorphism similar to $\sigma$ is constructed for the Lie algebra 
$N(\widehat{T}_1)_s$
and the small category $\widehat{T}_E$. 
We begin by defining a derivation on $\widehat{T} = (\widehat{T}_E)(*_0, *_0)$. 
Let $u = \sum^g_{i=1}A^s_iu'_i + \sum^g_{i=1}B^s_iu''_i + \sum^n_{j=1}C_ju^0_j$ 
be an element of $N(\widehat{T}_1)_s$. Then a continuous derivation 
$\sigma^0_s(u)$ of the algebra $\widehat{T}$ is defined by 
$$
\sigma^0_s(u)(A^s_i) := u''_i, \quad 
\sigma^0_s(u)(B^s_i) := -u'_i \,\,\,\mbox{and}\,\,\,
\sigma^0_s(u)(C_j) := u^0_jC_j - C_ju^0_j
$$
for $1 \leq i \leq g$ and $1 \leq j \leq n$. 
We define $u^0_0 := 0$ for our convenience. Then a continuous derivation 
$\sigma_s(u)$ of the small category $\widehat{T}_E$ is defined by 
\begin{equation}
\sigma_s(u)(v) := -u^0_av + \sigma^0_s(u)(v) + vu^0_b
\label{eq:2ss}
\end{equation}
for any $v \in (\widehat{T}_E)(*_a, *_b) = \widehat{T}$, $0 \leq a, b \leq n$. 
Similarly we denote by ${\rm Der}_\partial(\widehat{T}_E)$ the Lie algebra 
of continuous derivations annihilating $-\omega_s + C_0 \in (\widehat{T}_E)(*_0, *_0)$ and
$C_j \in (\widehat{T}_E)(*_j, *_j)$, $1 \leq j \leq n$. 
Then we have a Lie algebra homomorphism
$$
\sigma_s\colon N(\widehat{T}_1)_s \to {\rm Der}_\partial(\widehat{T}_E).
$$
By some straightforward computation, we can prove that $\sigma_s$ is an isomorphism.
\begin{theorem}[\cite{KK5}\cite{MT13}]
\label{thm:2der} If $\theta$ is a Magnus expansion of the pair $(S, E)$ 
satisfying the condition $(\sharp_s)$, then the diagram
$$
\begin{CD}
\widehat{\mathbb{Q}\hat\pi}(S) @>{\sigma}>> {\rm Der}_\partial
(\widehat{\mathbb{Q}\Pi S\vert_E})\\
@V{-N\theta}VV @V{\theta}VV\\
N(\widehat{T}_1)_s @>{\sigma_s}>>{\rm Der}_\partial(\widehat{T}_E)
\end{CD}
$$
commutes. 
\end{theorem}

To prove Theorems \ref{thm:2br} and \ref{thm:2der}, we consider a group-like 
expansion $\bar\theta$ obtained by gluing a symplectic expansion 
of $\Sigma_{g,1}$ 
and a special expansion of $\Sigma_{0, n+2}$. 
Here a group-like expansion $\theta\colon 
\Pi\Sigma_{0, n+2}\vert_E \to (1+\widehat{T}(H_1(\Sigma_{0, n+2}))_1)_E$ 
is called a special expansion if it satisfies 
$\theta(\xi_j) = e^{C_j}$ for any $j \geq 0$. 
By a similar argument to that in \cite{KK1} on relative twisted homology 
we deduce Theorems \ref{thm:2br} and \ref{thm:2der} 
for the expansion $\bar\theta$. 
As a corollary of them, we obtain Theorem \ref{thm:s-isom}, from which 
Theorems \ref{thm:2br} and \ref{thm:2der} 
for any expansion with the condition $(\sharp_s)$ follow. \par
Independently Massuyeau and Turaev \cite{MT13} give a proof of 
Theorems \ref{thm:2br} and \ref{thm:2der} in the context of quiver theory. 
See also \cite{MT12}. 
\par

\subsection{The geometric Johnson homomorphism}
\label{subsec:geocon}

Let $(S, E)$ be as above, and 
$\mathcal{M}(S)=\mathcal{M}(S,\partial S)$ the mapping class group of the surface $S$ 
fixing the boundary $\partial S$ {\it pointwise} (see \S \ref{subsec:DN}).
The largest Torelli group $\mathcal{I}^L(S)$  
in the sense of Putman \cite{Pu} is defined to be the kernel of the 
natural action on the quotient $H_1(S)/(\sum^n_{j=0}\mathbb{Q}C_j)$.  
In this subsection we introduce a pro-nilpotent Lie subalgebra 
$L^+(S)$ of the 
completed Goldman Lie algebra $\widehat{\mathbb{Q}\hat\pi}(S)$, 
and construct a natural embedding
$\tau\colon \mathcal{I}^L(S) \hookrightarrow L^+(S)$, 
the geometric Johnson homomorphism for the surface $S$,  
using the isomorphism
$$
\sigma\colon \widehat{\mathbb{Q}\hat\pi}(S) \overset\cong\to
{\rm Der}_{\partial}(\widehat{\mathbb{Q}\Pi S|_E})
$$
in Theorem \ref{thm:s-isom}. 
\par
By Theorem \ref{thm:DN}, the Dehn-Nielsen map
$\widehat{\sf DN}\colon \mathcal{M}(S) \to 
{\rm Aut}(\widehat{\mathbb{Q}\Pi S\vert_E})$ is injective. 
If $\varphi \in \mathcal{I}^L(S)$, then, 
by some straightforward argument, we have the logarithm
$\log\widehat{\sf DN}(\varphi) = \sum^\infty_{n=1}((-1)^{n-1}/n)
(\widehat{\sf DN}(\varphi) -1)^n$ as an element of 
$ {\rm Der}_\partial(\widehat{\mathbb{Q}\Pi S\vert_E})$. 
From Theorem \ref{thm:s-isom} we can define 
$$
\tau(\varphi) := \sigma^{-1}\left(\log\widehat{\sf DN}(\varphi)\right)
\in \widehat{\mathbb{Q}\hat\pi}(S).
$$
Hence we obtain a natural embedding $\tau\colon \mathcal{I}^L(S) \hookrightarrow 
\widehat{\mathbb{Q}\hat\pi}(S)$. \par
We use Putman's capping argument \cite{Pu} to define Lie subalgebras 
$\widehat{\mathbb{Q}\hat\pi}(S)(2\frac13)$ and $L^+(S)$. 
Let $g_j \geq 1$ be a positive integer for $1 \leq j \leq n$. 
We glue $\Sigma_{g_j,1}$ on $S$ along $\partial_jS$ for each $j \geq 1$ 
to obtain a compact connected oriented surface $\widetilde{S} := 
S\cup_{\partial S\setminus \partial_0S}\bigcup^n_{j=1}\Sigma_{g_j,1}$ of genus $g +\sum^n_{j=1}g_j$ 
with one boundary component.
We denote by $\iota\colon S \hookrightarrow \widetilde{S}$ the inclusion map.
Then the kernel of the induced homomorphism $\iota_*\colon 
\widehat{\mathbb{Q}\hat\pi}(S) \to 
\widehat{\mathbb{Q}\hat\pi}(\widetilde{S})$ 
is spanned by the set $\{\vert\log\xi_j\vert\}^n_{j=1}$ 
\cite{KK3} Lemma 6.2.3 (2). 
Choose a section $s \in {\rm Sect}(i_*)$. 
We define a weight ${\rm wt}_s$ on the algebra 
$\widehat{T}$ by ${\rm wt}_s(A^s_i) = {\rm wt}_s(B^s_i) = 1$ and 
${\rm wt}_s(C_j) = 2$ for $1 \leq i \leq g$ and $1 \leq j \leq n$. 
The following is the key lemma to define 
$\widehat{\mathbb{Q}\hat\pi}(S)(2\frac13)$ and $L^+(S)$.
\begin{lemma}\label{lem:3cap}
 If $\theta\colon \Pi S\vert_E \to (1+\widehat{T}_1)_E$ 
is a Magnus expansion of the pair $(S, E)$ satisfying the condition 
$(\sharp_s)$, then we have 
$$
(\iota_*)^{-1}(\widehat{\mathbb{Q}\hat\pi}(\widetilde{S})(3))
= {\rm Ker}(\iota_*) \oplus (N\theta)^{-1}\left(
\{u \in N(\widehat{T}_1) \subset \widehat{T}_1; \, {\rm wt}_s u 
\geq 3\}\right).
$$
\end{lemma}
We define 
\begin{align*}
\widehat{\mathbb{Q}\hat\pi}(S)(2\frac13) &:= \widehat{\mathbb{Q}\hat\pi}(S)(2) 
\cap (\iota_*)^{-1}(\widehat{\mathbb{Q}\hat\pi}(\widetilde{S})(3))\\
&= (N\theta)^{-1}\left(
\{u \in N(\widehat{T}_1) \subset \widehat{T}_1; \, {\rm wt}_s u 
\geq 3\}\right),
\end{align*}
which does not depend on the choice of 
integers $g_j$'s, a Magnus expansion $\theta$ nor a section $s$
by Lemma \ref{lem:3cap}. We have $\widehat{\mathbb{Q}\hat\pi}(S)
(3) \subset \widehat{\mathbb{Q}\hat\pi}(S)(2\frac13) \subset
\widehat{\mathbb{Q}\hat\pi}(S)(2)$. 
The restriction of $\iota_*$ to $\widehat{\mathbb{Q}\hat\pi}(S)(2)$ is 
injective, and $\widehat{\mathbb{Q}\hat\pi}(\widetilde{S})(3)$ 
is pro-nilpotent. Hence $\widehat{\mathbb{Q}\hat\pi}(S)(2\frac13)$ 
is also pro-nilpotent, so that it has a natural group structure induced 
by the Hausdorff series. The inclusion $\tau(\mathcal{I}^L(S)) 
\subset \widehat{\mathbb{Q}\hat\pi}(S)(2\frac13)$ follows 
from some straightforward argument. By construction 
the injective map 
$
\tau\colon \mathcal{I}^L(S) \to
\widehat{\mathbb{Q}\hat\pi}(S)(2\frac13)
$
is a group homomorphism. 
We define a Lie subalgebra $L^+(S)$ by 
$$
L^+(S) := \{u \in \widehat{\mathbb{Q}\hat\pi}(S)(2\frac13); \,
(\sigma(u)\widehat{\otimes}\sigma(u))\Delta = \Delta\sigma(u)\}.
$$
Here we have $\widehat{\mathbb{Q}\hat\pi}(\Sigma_{g,1})(2\frac13) = 
\widehat{\mathbb{Q}\hat\pi}(\Sigma_{g,1})(3)$, and so 
$L^+(\Sigma_{g,1})$ equals what we have defined in \S\ref{sec:Revi}. 
For an arbitrary $S$, 
the subalgebra $L^+(S)$ can be regarded as a subgroup of 
$\widehat{\mathbb{Q}\hat\pi}(S)(2\frac13)$. 
The image $\tau(\mathcal{I}^L(S))$ is included
in $L^+(S)$, since any element of $\mathcal{M}(S)$ 
preserves the coproduct 
$\Delta$ in $\widehat{\mathbb{Q}\Pi S|_E}$. 
Consequently we obtain an injective group homomorphism
$$
\tau\colon \mathcal{I}^L(S) \hookrightarrow L^+(S),
$$
which we call the {\it geometric Johnson homomorphism} for the surface $S$.
Note that notation here is different from  \cite{KK3} \S 6.3 and
\cite{KK4} \S 5.1, where we discussed Putman's ``smallest Torelli group". In fact,
when $S\neq \Sigma_{g,1}$ the Lie algebra $L^+(S,E)$ in \cite{KK3} \cite{KK4} 
is a proper subspace of $L^+(S)$.
Recently Church \cite{Chu11} constructed the first Johnson homomorphism
for all kinds of Putman's Torelli groups \cite{Pu}. We do not know any relation between
Church's construction and ours.

From the same argument as in Theorem \ref{thm:cb} follows
\begin{theorem}
$$
\delta\circ\tau  = 0\colon \mathcal{I}^L(S) \to L^+(S) \to 
\widehat{\mathbb{Q}\hat\pi}(S)\widehat{\otimes}
\widehat{\mathbb{Q}\hat\pi}(S).
$$
\end{theorem}

\section{Other topics and applications}
\label{sec:Oth}

In this section, we discuss other related topics.
In \S \ref{subsec:nil}, continuing the discussion in \S \ref{sec:DT}, we study the action
of a Dehn twist on the fundamental groupoid of the surface.
In \S \ref{subsec:liecho}, we describe the Lie algebra of chord diagrams,
which comes from the $Sp$-invariants of the Lie algebra of symplectic derivations.
An idea of using chord diagrams for description of the $Sp$-invariant tensors in $H^{\otimes m}$
goes back to a classical result due to Weyl \cite{Weyl}.
In \S \ref{subsec:center}, we make remarks on the center of the Goldman Lie algebra
and show a result on a surface of infinite genus.
In \S \ref{subsec:HGLA}, we make a review on the homological Goldman Lie algebra,
which is a quotient of the Goldman Lie algebra. It is easier to handle than the Goldman Lie algebra,
but still we have not reached a full understanding of it. We especially focus
on its homological properties.

\subsection{The action of Dehn twists on the nilpotent quotients}
\label{subsec:nil}

Let $S$ and $E$ be as in \S \ref{subsec:MC}.
Fix an integer $k\ge 1$. For $p_0,p_1\in E$, we define
\begin{align*}
N_k\mathbb{Q}\Pi S(p_0,p_1):&=
F_1\mathbb{Q}\Pi S(p_0,p_1)/F_{k+1}\mathbb{Q}\Pi S(p_0,p_1) \\
&\cong F_1\widehat{\mathbb{Q}\Pi S}(p_0,p_1)/F_{k+1}\widehat{\mathbb{Q}\Pi S}(p_0,p_1).
\end{align*}
The mapping class group $\mathcal{M}(S,E)$ acts on the $\mathbb{Q}$-vector space
$N_k\mathbb{Q}\Pi S(p_0,p_1)$ through the Dehn-Nielsen homomorphism (\ref{eq:DNhom}).
In this subsection we mention a few results about this action.

Notice that generalized Dehn twists induce
$\mathbb{Q}$-linear automorphisms of $N_k\mathbb{Q}\Pi S(p_0,p_1)$.
Let $\gamma\subset S^* \setminus \partial S$ be an unoriented loop.
Since $L(\gamma)\in \widehat{\mathbb{Q}\hat{\pi}}(S^*)(2)$, we have
$\sigma(L(\gamma))u\in F_{k+1}\widehat{\mathbb{Q}\Pi S}(p_0,p_1)$
for any $u\in F_{k+1}\widehat{\mathbb{Q}\Pi S}(p_0,p_1)$, cf.\ Proposition \ref{prop:filt}.
This implies that $\sigma(L(\gamma))$ induces a $\mathbb{Q}$-linear endomorphism of
$N_k\mathbb{Q}\Pi S(p_0,p_1)$, and
$t_{\gamma}=\exp(\sigma(L(\gamma)))$ induces a $\mathbb{Q}$-linear automorphism
of $N_k\mathbb{Q}\Pi S(p_0,p_1)$.
Fix $q\in S^*$ and let $\gamma_1$ and $\gamma_2$ be unoriented loops on $S^*$ represented by
loops $x_1$ and $x_2$ based at $q$, respectively. As in \S \ref{subsec:LT},
let $N_k\pi_1(S^*,q)=\pi_1(S^*,q)/\Gamma_{k+1}(\pi_1(S^*,q))$ be the $k$-th
nilpotent quotient of $\pi_1(S^*,q)$.

\begin{proposition}
\label{prop:actonN}
If $x_1=x_2\in N_k\pi_1(S^*,q)$, then $t_{\gamma_1}=t_{\gamma_2}$
on $N_k\mathbb{Q}\Pi S(p_0,p_1)$. Moreover, if the homology classes
$[x_1], [x_2]\in H_1(S^*;\mathbb{Z})$ are zero,
then $t_{\gamma_1}=t_{\gamma_2}$ on $N_{k+1}\mathbb{Q}\Pi S(p_0,p_1)$.
\end{proposition}

\begin{proof}
For simplicity we denote $\pi^*=\pi_1(S^*,q)$. Since $x_1=x_2\in N_k(\pi^*)$
there exists some $a\in \Gamma_{k+1}(\pi^*)$ such that
$x_1=x_2a$. Then $a-1\in (I\pi^*)^{k+1}$. Since
$(x_1-1)-(x_2-1)=x_2(a-1) \in (I\pi^*)^{k+1}$,
$(1/2)(\log x_1)^2-(1/2)(\log x_2)^2 \in F_{k+2}\widehat{\mathbb{Q}\pi^*}$.
Therefore, $L(\gamma_1)-L(\gamma_2)\in \widehat{\mathbb{Q}\hat{\pi}}(S^*)(k+2)$.
By Proposition \ref{prop:filt}, $\sigma(L(\gamma_1))=\sigma(L(\gamma_2))$
on $N_k\mathbb{Q}\Pi S(p_0,p_1)$.

The condition $[x_1]=[x_2]=0\in H_1(S^*;\mathbb{Z})$
means that $x_1-1$ and $x_2-1$ are elements of $(I\pi^*)^2$. Hence
$(1/2)(\log x_1)^2-(1/2)(\log x_2)^2 \in F_{k+3}\widehat{\mathbb{Q}\pi^*}$
and $L(\gamma_1)-L(\gamma_2)\in \widehat{\mathbb{Q}\hat{\pi}}(S^*)(k+3)$.
This proves the latter part.
\end{proof}

Suppose $S$ and $E$ satisfy the assumption of Theorem \ref{thm:DN} and $p_0=p_1$.
The map $\pi_1(S,p_0)\to N_k\mathbb{Q}\Pi S(p_0,p_0)=I\pi_1(S,p_0)/(I\pi_1(S,p_0))^{k+1}$,
$x\mapsto x-1$ induces a $\mathcal{M}(S,E)$-equivariant injection
$N_k(\pi_1(S,p_0))\to N_k\mathbb{Q}\Pi S(p_0,p_0)$.

\begin{corollary}
\label{cor:NQ}
Keep the assumptions as above and choose $q\in S$.
Let $C_1$ and $C_2$ be simple closed curves on $S$,
and let $x_1$ and $x_1$ be loops based at $q$ representing $C_1$ and $C_2$.
If $x_1=x_2 \in N_k(\pi_1(S,p_0))$, then the action of the Dehn twists $t_{C_1}$ and $t_{C_2}$
on $N_k(\pi_1(S,p_0))$ coincide. Moreover, if $C_1$ and $C_2$ are separating,
the the action of $t_{C_1}$ and $t_{C_2}$ on $N_{k+1}(\pi_1(S,p_0))$ coincide.
\end{corollary}

This is first proved for the classical case $S=\Sigma_{g,1}$ in \cite{KK1} Theorem 1.1.2.
We do not know whether this corollary can be proved without using the results in \S \ref{sec:DT}.

Now we consider the case the surface is $\Sigma=\Sigma_{g,1}$
as in \S \ref{sec:Cla} and \S \ref{sec:Revi}. We shall give a partial result
about formulas for the Johnson maps of a Dehn twist, which is closely related
to the action of a Dehn twist on the nilpotent quotients of the fundamental group $\pi=\pi_1(\Sigma,*)$.
Fix a symplectic expansion $\theta$ of $\pi$, cf. Definition \ref{def:Symp}.
Recall from \S \ref{subsec:Ext} the total Johnson map $T^{\theta}$ and
the $k$-th Johnson map $\tau_k^{\theta}$ associated to $\theta$.
Note that the action of $\varphi\in \mathcal{M}_{g,1}$ on the $k$-th nilpotent quotient $N_k=N_k(\pi)$
is determined by $|\varphi|$ and $\tau_i^{\theta}(\varphi)$, $1\le i\le k-1$.

Recall from \S \ref{subsec:alg-Gold} that any symplectic expansion $\theta$
induces a filtered Lie algebra isomorphism $-\lambda_{\theta}\colon
\widehat{\mathbb{Q}\hat{\pi}}\overset{\cong}{\to}\mathfrak{a}_g^-$.
Let $\gamma$ be an unoriented free loop on $\Sigma$. We set
$L^{\theta}(\gamma):=\lambda_{\theta}(L(\gamma))=(1/2)N(\ell^{\theta}(x)\ell^{\theta}(x))
\in \mathfrak{l}_g$, where $x\in \pi$ is a representative of $\gamma$.
By Theorems \ref{thm:logDT} and \ref{thm:Ntheta}, if $C$ is a simple closed curve on $\Sigma$, then
\begin{equation}
\label{eq:logDT-theta}
T^{\theta}(t_C)=\exp(-L^{\theta}(C))\in {\rm Aut}(\widehat{T}).
\end{equation}
Let $L_k^{\theta}$ be the degree $k$ part of $L^{\theta}$.
For example, $L_2^{\theta}(C)=[C][C]\in H^{\otimes 2}$, where $[C]\in H$ is the homology
class of $C$ with a fixed orientation.
For $X\in H$ we have $L_2^{\theta}(C)X=(X\cdot [C])[C]$, thus $(L_2^{\theta}(C))^2|_H=0$.
From (\ref{eq:logDT-theta}), computing modulo $\widehat{T}_2$ we obtain
$$|t_C|X=X-L_2^{\theta}(C)X=X-(X\cdot [C])[C], \quad {\rm for\ }X\in H.$$
This is the classical transvection formula. Computing modulo higher tensors, we obtain
explicit formulas for $\tau_k^{\theta}(t_C)$.

\begin{theorem}[\cite{KK1}]
Let $\theta$ be a symplectic expansion and $C$ a non-separating simple
closed curve on $\Sigma$. For simplicity we denote $L_k=L_k^{\theta}(C)$. Then we have
\begin{enumerate}
\item $\tau_1^{\theta}(t_C)=-L_3$,
\item $\tau_2^{\theta}(t_C)=-L_4+\displaystyle\frac{1}{2}[L_2,L_4]+
\displaystyle\frac{1}{2}(L_3)^2$.
\end{enumerate}
\end{theorem}

Note that if $x\in \pi$ is a representative of $C$,
then $L_3^{\theta}(C)=[x]\wedge \ell_2^{\theta}(x)\in \Lambda^3 H \subset H^{\otimes 3}$
(see \cite{KK1} Lemma 6.4.1). At the present stage we do not know explicit formulas for $\tau_k^{\theta}(t_C)$, $k\ge 3$
and $C$ non-separating. If $C$ is separating, the formula for $\tau_k^{\theta}$ becomes simple
since $L_2^{\theta}(C)=0$. See \cite{KK1} Theorem 6.3.1.

\subsection{Lie algebras based on chord diagrams}
\label{subsec:liecho}

As in \S \ref{sec:Revi}, let $H=H_1(\Sigma;\mathbb{Q})$ be the first rational homology
group of the surface $\Sigma=\Sigma_{g,1}$ and $Sp=Sp(H)\cong Sp(2g;\mathbb{Q})$.
A classical result of Weyl \cite{Weyl} is that the space of
$Sp$-invariant tensors in $H^{\otimes m}$ is generated by {\it chord diagrams}.
The idea is to make the symplectic form $\omega$ correspond to a labeled chord.
This description of $Sp$-invariant tensors has been used in several works
such as \cite{Kon93} \cite{KM} \cite{Mor03} \cite{Mor99}.

In this subsection we review Lie algebra structures on the spaces of chord diagrams introduced in \cite{KK2},
which come from the Lie bracket on the $Sp$-invariants of the
Lie algebras ${\rm Der}(T)$ and ${\rm Der}_{\omega}(T)$.
Here $T=\bigoplus_{m=0}^{\infty} H^{\otimes m}$
is the tensor algebra generated by $H$, ${\rm Der}(T)$ is the Lie algebra of
derivations of $T$, and ${\rm Der}_{\omega}(T)$ is the Lie subalgebra of ${\rm Der}(T)$
consisting of derivations annihilating $\omega$. Note that
the degree completion of ${\rm Der}_{\omega}(T)$ is the Lie algebra $\mathfrak{a}_g^-$ in \S \ref{subsec:SD}.
The symplectic group $Sp$ acts naturally on ${\rm Der}(T)$, and this action preserves ${\rm Der}_{\omega}(T)$. As in
\S \ref{subsec:SD} we can identify
${\rm Der}(T)$ with $\bigoplus_{m=1}^{\infty} H^{\otimes m}$ by the restriction
$${\rm Der}(T) \overset{\cong}{\to} {\rm Hom}(H,T)=H^* \otimes T\cong H\otimes T
=\bigoplus_{m=1}^{\infty} H^{\otimes m},\quad D\mapsto D|_H.$$
Also we have ${\rm Der}_{\omega}(T)=\bigoplus_{m=1}^{\infty} N(H^{\otimes m})$.
The action of $Sp$ coincides with the diagonal action on the tensor spaces $H^{\otimes m}$. 
The $Sp$-invariant parts ${\rm Der}(T)^{Sp}$ and ${\rm Der}_{\omega}(T)^{Sp}$
are Lie subalgebras of ${\rm Der}(T)$ and ${\rm Der}_{\omega}(T)$, respectively.

Let $m$ be a positive integer. A {\it labeled linear chord diagram} of $m$ chords
is a set of $m$ ordered pairs $C=\{ (i_1,j_1),(i_2,j_2),\ldots, (i_m,j_m)\}$
satisfying $\{ i_1,\ldots,i_m, \\ j_1,\ldots,j_m \}=\{1,2,\ldots,2m \}$.
We draw a picture of a labeled linear chord diagram as in Figure 6.
If $i_k< j_k$ for any $1\le k\le m$, we say the label of $C$ is {\it standard}.
If $C^{\prime}$ is another labeled linear chord diagram such that
$$C^{\prime}=\{ (i_1,j_1),\ldots,(i_{k-1},j_{k-1}),(j_k,i_k),(i_{k+1},j_{k+1}),
\ldots,(i_m,j_m) \}$$
for some $1\le k\le m$, we say $C^{\prime}$ is obtained from $C$ by a single
label change. Let $\mathcal{LC}_m$ be the $\mathbb{Q}$-linear space
spanned by the labeled linear chord diagram of $m$ chords modulo the $\mathbb{Q}$-linear
subspace generated by the set
$$\{ C+C^{\prime}; C^{\prime} {\rm \ is \ obtained \ from \ } C
{\rm \ by \ a \ single \ label \ change} \}.$$
Note that $\mathcal{LC}_m$ is $(2m-1)!!$ dimensional, and the set of
linear chord diagrams with standard label is a basis for $\mathcal{LC}_m$.

\begin{figure}
\label{fig:lcd}
\begin{center}
\unitlength 0.1in
\begin{picture}( 20.8000,  4.2400)(  4.0000, -5.7400)
%
{\color[named]{Black}{%
\special{pn 4}%
\special{sh 1}%
\special{ar 560 510 20 20 0  6.28318530717959E+0000}%
\special{sh 1}%
\special{ar 880 510 20 20 0  6.28318530717959E+0000}%
\special{sh 1}%
\special{ar 1200 510 20 20 0  6.28318530717959E+0000}%
\special{sh 1}%
\special{ar 1520 510 20 20 0  6.28318530717959E+0000}%
\special{sh 1}%
\special{ar 1840 510 20 20 0  6.28318530717959E+0000}%
\special{sh 1}%
\special{ar 2160 510 20 20 0  6.28318530717959E+0000}%
}}%
%
{\color[named]{Black}{%
\special{pn 8}%
\special{ar 720 510 160 160  3.1415927 6.2831853}%
}}%
%
{\color[named]{Black}{%
\special{pn 8}%
\special{ar 1520 510 320 320  3.1415927 6.2831853}%
}}%
%
{\color[named]{Black}{%
\special{pn 8}%
\special{ar 1840 510 320 320  3.1415927 6.2831853}%
}}%
%
{\color[named]{Black}{%
\special{pn 8}%
\special{pa 1840 190}%
\special{pa 1760 150}%
\special{fp}%
}}%
%
{\color[named]{Black}{%
\special{pn 8}%
\special{pa 1520 190}%
\special{pa 1440 150}%
\special{fp}%
}}%
%
{\color[named]{Black}{%
\special{pn 8}%
\special{pa 720 350}%
\special{pa 640 310}%
\special{fp}%
}}%
\put(5.6000,-6.7000){\makebox(0,0)[lb]{$1$}}%
\put(8.8000,-6.7000){\makebox(0,0)[lb]{$2$}}%
\put(12.0000,-6.7000){\makebox(0,0)[lb]{$3$}}%
\put(15.2000,-6.7000){\makebox(0,0)[lb]{$4$}}%
\put(18.4000,-6.7000){\makebox(0,0)[lb]{$5$}}%
\put(21.6000,-6.7000){\makebox(0,0)[lb]{$6$}}%
%
{\color[named]{Black}{%
\special{pn 8}%
\special{pa 1840 190}%
\special{pa 1776 254}%
\special{fp}%
}}%
%
{\color[named]{Black}{%
\special{pn 8}%
\special{pa 1520 190}%
\special{pa 1456 254}%
\special{fp}%
}}%
%
{\color[named]{Black}{%
\special{pn 8}%
\special{pa 720 350}%
\special{pa 656 414}%
\special{fp}%
}}%
%
{\color[named]{Black}{%
\special{pn 8}%
\special{pa 400 510}%
\special{pa 2480 510}%
\special{fp}%
}}%
%
{\color[named]{Black}{%
\special{pn 8}%
\special{pa 2320 510}%
\special{pa 2256 446}%
\special{fp}%
}}%
%
{\color[named]{Black}{%
\special{pn 8}%
\special{pa 2320 510}%
\special{pa 2256 574}%
\special{fp}%
}}%
\end{picture}%
\end{center}
\caption{$C=\{ (1,2),(3,5),(4,6) \}$}
\end{figure}
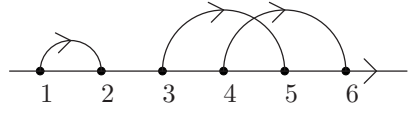

The symmetric group $\mathfrak{S}_{2m}$ acts naturally on the tensor space $H^{\otimes 2m}$.
For a labeled linear chord diagram $C$, we define
$$a(C):=\left( \begin{array}{ccccc}
1 & 2 & \cdots 2m-1 & 2m \\
i_1 & j_1 & \cdots i_m & j_m
\end{array} \right)(\omega^{\otimes m}) \in H^{\otimes 2m}.
$$
This is an $Sp$-invariant tensor. Since $a(C^{\prime})=-a(C)$ if
$C^{\prime}$ is obtained from $C$ by a single label change, the correspondence $a$ induces
a $\mathbb{Q}$-linear map
$$a\colon \mathcal{LC}_m \to (H^{\otimes 2m})^{Sp}, \quad C \mapsto a(C).$$
The following proposition is due to Weyl \cite{Weyl} except for
``only if" part of (3) which is due to Morita \cite{Mor99}.

\begin{proposition}
\label{prop:stbl-isom}
\begin{enumerate}
\item If $n$ is odd, the space of $Sp$-invariant tensors $(H^{\otimes n})^{Sp}$ is zero.
\item The map $a \colon \mathcal{LC}_m \to (H^{\otimes 2m})^{Sp}$ is surjective for any $m\ge 1$.
\item The map $a \colon \mathcal{LC}_m \to (H^{\otimes 2m})^{Sp}$ is an isomorphism if and only
if $m\ge g$.
\end{enumerate}
\end{proposition}

Set $\mathcal{LC}:=\bigoplus_{m=1}^{\infty}\mathcal{LC}_m$.
From Proposition \ref{prop:stbl-isom} the map
\begin{equation}
\label{eq:stbl-a}
a\colon \mathcal{LC}\to
\bigoplus_{m=1}^{\infty}(H^{\otimes 2m})^{Sp}={\rm Der}(T)^{Sp}
\end{equation}
is a stable isomorphism, and we can introduce a Lie bracket on $\mathcal{LC}$
so that (\ref{eq:stbl-a}) is a Lie algebra homomorphism.

To describe the Lie bracket on $\mathcal{LC}$, we define the {\it amalgamation} of two linear chord diagrams.
Let $C$ and $C^{\prime}$ be linear chord diagrams with standard label of $m$ and $l$ chords, respectively.
For $2\le t\le 2l$, we define the $t$-th amalgamation $C\ast_t C^{\prime}$ as a linear chord
diagram with standard label as follows. We first cut $C^{\prime}$ at the $t$-th vertex and
$C$ at the first vertex, insert $C$ into the $t$-th hole of the cut $C^{\prime}$,
then connect the first vertex of $C$ to the $t$-th vertex of $C^{\prime}$. See Figure 7.
The amalgamation $C\ast_t C^{\prime}$ is the linear chord diagram of the result with standard label.
Then the bracket $[C,C^{\prime}]\in \mathcal{LC}_{m+l-1}$ is given by
\begin{equation}
\label{eq:braCC'}
[C,C^{\prime}]=-\sum_{t=2}^{2l}C\ast_t C^{\prime}+\sum_{s=2}^{2m} C^{\prime}\ast_s C.
\end{equation}

\begin{figure}
\label{fig:amalgamation}
\begin{center}
\unitlength 0.1in
\begin{picture}( 31.2000, 14.5000)(  3.0000,-16.5000)
%
{\color[named]{Black}{%
\special{pn 8}%
\special{pa 700 520}%
\special{pa 1660 520}%
\special{fp}%
}}%
%
{\color[named]{Black}{%
\special{pn 4}%
\special{sh 1}%
\special{ar 860 520 20 20 0  6.28318530717959E+0000}%
}}%
%
{\color[named]{Black}{%
\special{pn 8}%
\special{ar 1100 520 240 240  3.1415927 6.2831853}%
}}%
%
{\color[named]{Black}{%
\special{pn 8}%
\special{pa 1020 440}%
\special{pa 1580 440}%
\special{pa 1580 600}%
\special{pa 1020 600}%
\special{pa 1020 440}%
\special{pa 1580 440}%
\special{fp}%
}}%
%
{\color[named]{Black}{%
\special{pn 8}%
\special{pa 2140 520}%
\special{pa 3420 520}%
\special{fp}%
}}%
%
{\color[named]{Black}{%
\special{pn 8}%
\special{pa 2220 440}%
\special{pa 2700 440}%
\special{pa 2700 600}%
\special{pa 2220 600}%
\special{pa 2220 440}%
\special{pa 2700 440}%
\special{fp}%
}}%
%
{\color[named]{Black}{%
\special{pn 8}%
\special{pa 2860 600}%
\special{pa 3340 600}%
\special{pa 3340 440}%
\special{pa 2860 440}%
\special{pa 2860 600}%
\special{pa 3340 600}%
\special{fp}%
}}%
%
{\color[named]{Black}{%
\special{pn 4}%
\special{sh 1}%
\special{ar 2772 520 20 20 0  6.28318530717959E+0000}%
}}%
%
{\color[named]{Black}{%
\special{pn 8}%
\special{ar 3020 520 240 240  3.1415927 6.2831853}%
}}%
%
{\color[named]{Black}{%
\special{pn 8}%
\special{ar 2700 520 320 320  3.1415927 3.3290927}%
\special{ar 2700 520 320 320  3.4415927 3.6290927}%
\special{ar 2700 520 320 320  3.7415927 3.9290927}%
\special{ar 2700 520 320 320  4.0415927 4.2290927}%
\special{ar 2700 520 320 320  4.3415927 4.5290927}%
\special{ar 2700 520 320 320  4.6415927 4.8290927}%
\special{ar 2700 520 320 320  4.9415927 5.1290927}%
\special{ar 2700 520 320 320  5.2415927 5.4290927}%
\special{ar 2700 520 320 320  5.5415927 5.7290927}%
\special{ar 2700 520 320 320  5.8415927 6.0290927}%
\special{ar 2700 520 320 320  6.1415927 6.2831853}%
}}%
\put(4.6000,-5.2000){\makebox(0,0)[lb]{$C$}}%
\put(19.0000,-5.2000){\makebox(0,0)[lb]{$C^{\prime}$}}%
\put(27.8000,-6.8000){\makebox(0,0)[lb]{$t$}}%
%
{\color[named]{Black}{%
\special{pn 8}%
\special{pa 860 1620}%
\special{pa 1340 1620}%
\special{pa 1340 1460}%
\special{pa 860 1460}%
\special{pa 860 1620}%
\special{pa 1340 1620}%
\special{fp}%
}}%
%
{\color[named]{Black}{%
\special{pn 8}%
\special{pa 1500 1620}%
\special{pa 2060 1620}%
\special{pa 2060 1460}%
\special{pa 1500 1460}%
\special{pa 1500 1620}%
\special{pa 2060 1620}%
\special{fp}%
}}%
%
{\color[named]{Black}{%
\special{pn 8}%
\special{pa 2220 1460}%
\special{pa 2700 1460}%
\special{pa 2700 1620}%
\special{pa 2220 1620}%
\special{pa 2220 1460}%
\special{pa 2700 1460}%
\special{fp}%
}}%
%
{\color[named]{Black}{%
\special{pn 8}%
\special{ar 2220 1540 400 400  3.1415927 6.2831853}%
}}%
%
{\color[named]{Black}{%
\special{pn 8}%
\special{ar 1740 1540 640 640  3.1415927 3.2353427}%
\special{ar 1740 1540 640 640  3.2915927 3.3853427}%
\special{ar 1740 1540 640 640  3.4415927 3.5353427}%
\special{ar 1740 1540 640 640  3.5915927 3.6853427}%
\special{ar 1740 1540 640 640  3.7415927 3.8353427}%
\special{ar 1740 1540 640 640  3.8915927 3.9853427}%
\special{ar 1740 1540 640 640  4.0415927 4.1353427}%
\special{ar 1740 1540 640 640  4.1915927 4.2853427}%
\special{ar 1740 1540 640 640  4.3415927 4.4353427}%
\special{ar 1740 1540 640 640  4.4915927 4.5853427}%
\special{ar 1740 1540 640 640  4.6415927 4.7353427}%
\special{ar 1740 1540 640 640  4.7915927 4.8853427}%
\special{ar 1740 1540 640 640  4.9415927 5.0353427}%
\special{ar 1740 1540 640 640  5.0915927 5.1853427}%
\special{ar 1740 1540 640 640  5.2415927 5.3353427}%
\special{ar 1740 1540 640 640  5.3915927 5.4853427}%
\special{ar 1740 1540 640 640  5.5415927 5.6353427}%
\special{ar 1740 1540 640 640  5.6915927 5.7853427}%
\special{ar 1740 1540 640 640  5.8415927 5.9353427}%
\special{ar 1740 1540 640 640  5.9915927 6.0853427}%
\special{ar 1740 1540 640 640  6.1415927 6.2353427}%
}}%
\put(11.8000,-7.6000){\makebox(0,0)[lb]{$C_{>1}$}}%
\put(24.6000,-7.6000){\makebox(0,0)[lb]{$C^{\prime}_{<t}$}}%
\put(31.0000,-7.6000){\makebox(0,0)[lb]{$C^{\prime}_{>t}$}}%
%
{\color[named]{Black}{%
\special{pn 8}%
\special{pa 700 1540}%
\special{pa 2780 1540}%
\special{fp}%
}}%
\put(3.0000,-13.8000){\makebox(0,0)[lb]{$C\ast_t C^{\prime}$}}%
\put(16.6000,-17.8000){\makebox(0,0)[lb]{$C_{>1}$}}%
\put(9.4000,-17.8000){\makebox(0,0)[lb]{$C^{\prime}_{<t}$}}%
\put(23.0000,-17.8000){\makebox(0,0)[lb]{$C^{\prime}_{>t}$}}%
%
{\color[named]{Black}{%
\special{pn 4}%
\special{sh 1}%
\special{ar 1340 520 20 20 0  6.28318530717959E+0000}%
}}%
%
{\color[named]{Black}{%
\special{pn 4}%
\special{sh 1}%
\special{ar 3260 520 20 20 0  6.28318530717959E+0000}%
}}%
%
{\color[named]{Black}{%
\special{pn 4}%
\special{sh 1}%
\special{ar 1820 1540 20 20 0  6.28318530717959E+0000}%
}}%
%
{\color[named]{Black}{%
\special{pn 4}%
\special{sh 1}%
\special{ar 2620 1540 20 20 0  6.28318530717959E+0000}%
}}%
\end{picture}%
\end{center}
\caption{the $t$-th amalgamation $C\ast_t C^{\prime}$}
\end{figure}
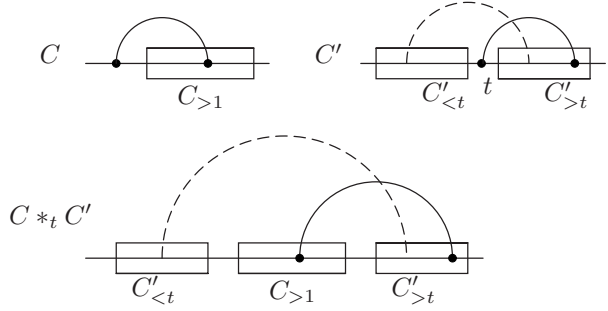

We next consider the Lie algebra ${\rm Der}_{\omega}(T)^{Sp}$. Let $\nu_m=\nu\in \mathfrak{S}_{2m}$
be the cyclic permutation
$$\nu=\left( \begin{array}{ccccc}
1 & 2 & 3 & \cdots & 2m \\
2m & 1 & 2 & \cdots & 2m-1
\end{array} \right).$$
For a labeled linear chord diagram $C=\{ (i_1,j_1),(i_2,j_2),\ldots,(i_m,j_m) \}$ and $s\in \mathbb{Z}$, we define
$$\nu^s(C):=\{ (\nu^s(i_1),\nu^s(j_1)),(\nu^s(i_2),\nu^s(j_2)),\ldots,(\nu^s(i_m),\nu^s(j_m))\}.$$
The cyclic group of order $2m$, generated by $\nu$, acts on the $\mathbb{Q}$-vector space $\mathcal{LC}_m$.
Let $\mathcal{C}_m\subset \mathcal{LC}_m$ be the $\mathbb{Z}_{2m}$-invarints under this action.
If $m=1$ and $C=\{ (1,2)\}$, then $\nu(C)=-C\in \mathcal{LC}_1$. This means that $\mathcal{C}_1=0$.
The $\mathbb{Q}$-linear space $\mathcal{C}_m$ is generated by labeled {\it circular} chord diagrams.
More precisely, the space $\mathcal{C}_m$ is generated by element of the form $N(C)=\sum_{s=0}^{2m-1} \nu^s(C)$,
where $C$ is a labeled linear chord diagram of $m$ chords. We draw a picture of $N(C)$ as in Figure 8.
Here the picture is a labeled circular chord diagram obtained as the ``closing" of the picture of 
$C=\{ (1,2),(3,5),(4,6) \}$ in Figure 6.
We call $\mathcal{C}_m$ the {\it space of oriented circular chord diagram} of $m$ chords.
The direct sum $\mathcal{C}=\bigoplus_{m=2}^{\infty} \mathcal{C}_m$ is a Lie subalgebra of $\mathcal{LC}$.
Since the tensor $a(N(C))$ is cyclically invariant, (\ref{eq:stbl-a}) induces a Lie algebra homomorphism
$$a\colon \mathcal{C}\to \bigoplus_{m=2}^{\infty}(N(H^{\otimes 2m}))^{Sp}={\rm Der}_{\omega}(T)^{Sp}.$$

\begin{figure}
\label{fig:circular}
\begin{center}
\unitlength 0.1in
\begin{picture}(  9.9700,  9.6000)(  4.0000,-13.6300)
%
{\color[named]{Black}{%
\special{pn 8}%
\special{ar 880 884 480 480  0.0000000 6.2831853}%
}}%
%
{\color[named]{Black}{%
\special{pn 4}%
\special{sh 1}%
\special{ar 1294 644 20 20 0  6.28318530717959E+0000}%
}}%
%
{\color[named]{Black}{%
\special{pn 4}%
\special{sh 1}%
\special{ar 1294 1124 20 20 0  6.28318530717959E+0000}%
}}%
%
{\color[named]{Black}{%
\special{pn 4}%
\special{sh 1}%
\special{ar 472 1124 20 20 0  6.28318530717959E+0000}%
}}%
%
{\color[named]{Black}{%
\special{pn 4}%
\special{sh 1}%
\special{ar 472 644 20 20 0  6.28318530717959E+0000}%
}}%
%
{\color[named]{Black}{%
\special{pn 4}%
\special{sh 1}%
\special{ar 880 404 20 20 0  6.28318530717959E+0000}%
}}%
%
{\color[named]{Black}{%
\special{pn 4}%
\special{sh 1}%
\special{ar 880 1364 20 20 0  6.28318530717959E+0000}%
}}%
%
{\color[named]{Black}{%
\special{pn 8}%
\special{ar 1088 1242 240 240  2.6147147 5.7542102}%
}}%
%
{\color[named]{Black}{%
\special{pn 8}%
\special{pa 880 404}%
\special{pa 472 1124}%
\special{fp}%
}}%
%
{\color[named]{Black}{%
\special{pn 8}%
\special{pa 472 644}%
\special{pa 1288 644}%
\special{fp}%
}}%
%
{\color[named]{Black}{%
\special{pn 8}%
\special{pa 1000 644}%
\special{pa 1048 596}%
\special{fp}%
}}%
%
{\color[named]{Black}{%
\special{pn 8}%
\special{pa 1000 644}%
\special{pa 1048 692}%
\special{fp}%
}}%
%
{\color[named]{Black}{%
\special{pn 8}%
\special{pa 660 786}%
\special{pa 640 722}%
\special{fp}%
}}%
%
{\color[named]{Black}{%
\special{pn 8}%
\special{pa 660 786}%
\special{pa 724 768}%
\special{fp}%
}}%
%
{\color[named]{Black}{%
\special{pn 8}%
\special{pa 976 1038}%
\special{pa 910 1032}%
\special{fp}%
\special{pa 976 1038}%
\special{pa 950 1100}%
\special{fp}%
}}%
%
{\color[named]{Black}{%
\special{pn 8}%
\special{pa 1360 846}%
\special{pa 1398 902}%
\special{fp}%
\special{pa 1360 846}%
\special{pa 1320 900}%
\special{fp}%
}}%
\end{picture}%
\end{center}
\caption{A labeled circular chord diagram}
\end{figure}
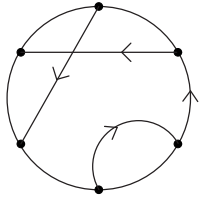 

The Lie bracket on $\mathcal{C}$ is given as follows.
Let $D$ and $D^{\prime}$ be labeled circular chord diagrams of $m$ and $m^{\prime}$ chords, respectively.
For vertices $p$ of $D$ and $q$ of $D^{\prime}$, we define the labeled circular chord
diagram $\mathcal{D}(D,p,D^{\prime},q)$ as the result of a certain surgery at $p$ and $q$ illustrated in Figure 9.
Here the label of the chord connecting $\overline{p}$ and $\overline{q}$ are determined by the rule in Figure 10.
Then the bracket $[D,D^{\prime}]\in \mathcal{C}_{m+m^{\prime}-1}$ is given by
\begin{equation}
\label{eq:braDD'}
[D,D^{\prime}]=\sum_{(p,q)} \mathcal{D}(D,p,D^{\prime},q),
\end{equation}
where the sum is taken over all pairs of the vertices of $D$ and $D^{\prime}$.

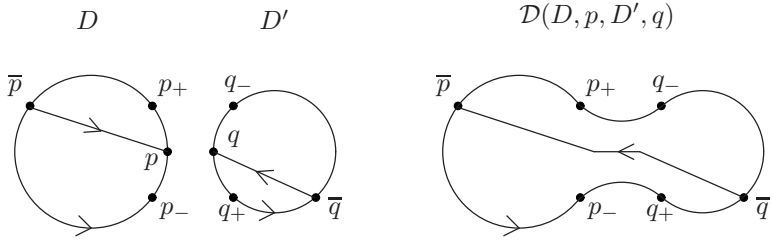
\begin{figure}
\label{fig:D(p,q)}
\begin{center}
\unitlength 0.1in
\begin{picture}( 39.5200, 12.1000)(  2.6000,-14.5400)
%
{\color[named]{Black}{%
\special{pn 8}%
\special{ar 1652 1014 320 320  0.0000000 6.2831853}%
}}%
\put(6.1200,-3.7400){\makebox(0,0)[lb]{$D$}}%
\put(15.7200,-3.7400){\makebox(0,0)[lb]{$D^{\prime}$}}%
%
{\color[named]{Black}{%
\special{pn 8}%
\special{ar 692 1014 400 400  0.0000000 6.2831853}%
}}%
%
{\color[named]{Black}{%
\special{pn 4}%
\special{sh 1}%
\special{ar 1092 1014 20 20 0  6.28318530717959E+0000}%
}}%
%
{\color[named]{Black}{%
\special{pn 4}%
\special{sh 1}%
\special{ar 1332 1014 20 20 0  6.28318530717959E+0000}%
}}%
%
{\color[named]{Black}{%
\special{pn 4}%
\special{sh 1}%
\special{ar 1012 774 20 20 0  6.28318530717959E+0000}%
}}%
%
{\color[named]{Black}{%
\special{pn 4}%
\special{sh 1}%
\special{ar 1436 1254 20 20 0  6.28318530717959E+0000}%
}}%
%
{\color[named]{Black}{%
\special{pn 4}%
\special{sh 1}%
\special{ar 1012 1254 20 20 0  6.28318530717959E+0000}%
}}%
%
{\color[named]{Black}{%
\special{pn 4}%
\special{sh 1}%
\special{ar 1436 774 20 20 0  6.28318530717959E+0000}%
}}%
%
{\color[named]{Black}{%
\special{pn 4}%
\special{sh 1}%
\special{ar 372 774 20 20 0  6.28318530717959E+0000}%
}}%
%
{\color[named]{Black}{%
\special{pn 4}%
\special{sh 1}%
\special{ar 1868 1254 20 20 0  6.28318530717959E+0000}%
}}%
%
{\color[named]{Black}{%
\special{pn 8}%
\special{pa 1092 1014}%
\special{pa 372 782}%
\special{fp}%
}}%
%
{\color[named]{Black}{%
\special{pn 8}%
\special{pa 1332 1014}%
\special{pa 1860 1254}%
\special{fp}%
}}%
\put(9.7200,-11.1000){\makebox(0,0)[lb]{$p$}}%
\put(14.0400,-9.9000){\makebox(0,0)[lb]{$q$}}%
\put(2.6000,-7.2600){\makebox(0,0)[lb]{$\overline{p}$}}%
\put(19.3200,-13.6600){\makebox(0,0)[lb]{$\overline{q}$}}%
\put(10.5200,-13.7400){\makebox(0,0)[lb]{$p_-$}}%
\put(10.4400,-7.2600){\makebox(0,0)[lb]{$p_+$}}%
\put(13.8800,-7.1000){\makebox(0,0)[lb]{$q_-$}}%
\put(13.5600,-13.8200){\makebox(0,0)[lb]{$q_+$}}%
%
{\color[named]{Black}{%
\special{pn 4}%
\special{sh 1}%
\special{ar 3252 774 20 20 0  6.28318530717959E+0000}%
}}%
%
{\color[named]{Black}{%
\special{pn 4}%
\special{sh 1}%
\special{ar 3676 1254 20 20 0  6.28318530717959E+0000}%
}}%
%
{\color[named]{Black}{%
\special{pn 4}%
\special{sh 1}%
\special{ar 3252 1254 20 20 0  6.28318530717959E+0000}%
}}%
%
{\color[named]{Black}{%
\special{pn 4}%
\special{sh 1}%
\special{ar 3676 774 20 20 0  6.28318530717959E+0000}%
}}%
%
{\color[named]{Black}{%
\special{pn 4}%
\special{sh 1}%
\special{ar 2612 774 20 20 0  6.28318530717959E+0000}%
}}%
%
{\color[named]{Black}{%
\special{pn 4}%
\special{sh 1}%
\special{ar 4108 1254 20 20 0  6.28318530717959E+0000}%
}}%
\put(25.0000,-7.2600){\makebox(0,0)[lb]{$\overline{p}$}}%
\put(41.7200,-13.6600){\makebox(0,0)[lb]{$\overline{q}$}}%
\put(32.9200,-13.7400){\makebox(0,0)[lb]{$p_-$}}%
\put(32.8400,-7.2600){\makebox(0,0)[lb]{$p_+$}}%
\put(36.2800,-7.1000){\makebox(0,0)[lb]{$q_-$}}%
\put(35.9600,-13.8200){\makebox(0,0)[lb]{$q_+$}}%
%
{\color[named]{Black}{%
\special{pn 8}%
\special{ar 2932 1014 400 400  0.6435011 5.6396842}%
}}%
%
{\color[named]{Black}{%
\special{pn 8}%
\special{ar 3892 1014 320 320  3.9269908 6.2831853}%
\special{ar 3892 1014 320 320  0.0000000 2.3561945}%
}}%
%
{\color[named]{Black}{%
\special{pn 8}%
\special{ar 3468 1494 320 320  3.9958473 5.4071273}%
}}%
%
{\color[named]{Black}{%
\special{pn 8}%
\special{ar 3468 534 320 320  0.8760581 2.2873380}%
}}%
%
{\color[named]{Black}{%
\special{pn 8}%
\special{pa 2612 782}%
\special{pa 3324 1014}%
\special{fp}%
\special{pa 3324 1014}%
\special{pa 3564 1014}%
\special{fp}%
\special{pa 3564 1014}%
\special{pa 4108 1254}%
\special{fp}%
}}%
\put(29.3200,-3.7400){\makebox(0,0)[lb]{$\mathcal{D}(D,p,D^{\prime},q)$}}%
%
{\color[named]{Black}{%
\special{pn 8}%
\special{pa 742 906}%
\special{pa 684 838}%
\special{fp}%
\special{pa 742 906}%
\special{pa 654 912}%
\special{fp}%
}}%
%
{\color[named]{Black}{%
\special{pn 8}%
\special{pa 1556 1118}%
\special{pa 1646 1112}%
\special{fp}%
\special{pa 1556 1118}%
\special{pa 1614 1186}%
\special{fp}%
}}%
%
{\color[named]{Black}{%
\special{pn 8}%
\special{pa 3452 1014}%
\special{pa 3532 974}%
\special{fp}%
\special{pa 3452 1014}%
\special{pa 3532 1054}%
\special{fp}%
}}%
%
{\color[named]{Black}{%
\special{pn 8}%
\special{pa 692 1414}%
\special{pa 628 1350}%
\special{fp}%
\special{pa 692 1414}%
\special{pa 612 1454}%
\special{fp}%
}}%
%
{\color[named]{Black}{%
\special{pn 8}%
\special{pa 2932 1414}%
\special{pa 2868 1350}%
\special{fp}%
\special{pa 2932 1414}%
\special{pa 2852 1454}%
\special{fp}%
}}%
%
{\color[named]{Black}{%
\special{pn 8}%
\special{pa 1652 1334}%
\special{pa 1588 1270}%
\special{fp}%
\special{pa 1652 1334}%
\special{pa 1572 1374}%
\special{fp}%
}}%
\end{picture}%
\end{center}
\caption{$\mathcal{D}(D,p,D^{\prime},q)$}
\end{figure} 

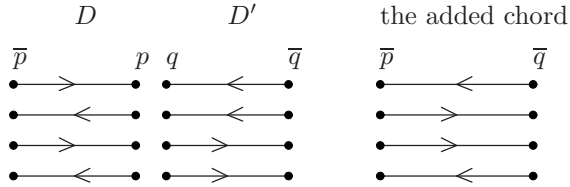
\begin{figure}
\label{fig:add}
\begin{center}
\unitlength 0.1in
\begin{picture}( 27.2000,  9.6200)(  4.0000,-14.0600)
%
{\color[named]{Black}{%
\special{pn 8}%
\special{pa 400 894}%
\special{pa 1040 894}%
\special{fp}%
}}%
%
{\color[named]{Black}{%
\special{pn 8}%
\special{pa 1200 894}%
\special{pa 1840 894}%
\special{fp}%
}}%
%
{\color[named]{Black}{%
\special{pn 8}%
\special{pa 2320 894}%
\special{pa 3120 894}%
\special{fp}%
}}%
\put(7.2000,-5.7400){\makebox(0,0)[lb]{$D$}}%
\put(15.2000,-5.7400){\makebox(0,0)[lb]{$D^{\prime}$}}%
\put(4.0000,-8.1400){\makebox(0,0)[lb]{$\overline{p}$}}%
\put(10.4000,-8.1400){\makebox(0,0)[lb]{$p$}}%
\put(12.0000,-8.1400){\makebox(0,0)[lb]{$q$}}%
\put(18.4000,-8.1400){\makebox(0,0)[lb]{$\overline{q}$}}%
\put(23.2000,-8.1400){\makebox(0,0)[lb]{$\overline{p}$}}%
\put(31.2000,-8.1400){\makebox(0,0)[lb]{$\overline{q}$}}%
\put(23.2000,-5.7400){\makebox(0,0)[lb]{the added chord}}%
%
{\color[named]{Black}{%
\special{pn 8}%
\special{pa 1200 1054}%
\special{pa 1840 1054}%
\special{fp}%
}}%
%
{\color[named]{Black}{%
\special{pn 8}%
\special{pa 2320 1054}%
\special{pa 3120 1054}%
\special{fp}%
}}%
%
{\color[named]{Black}{%
\special{pn 8}%
\special{pa 720 894}%
\special{pa 640 854}%
\special{fp}%
\special{pa 720 894}%
\special{pa 640 926}%
\special{fp}%
}}%
%
{\color[named]{Black}{%
\special{pn 8}%
\special{pa 1520 894}%
\special{pa 1600 854}%
\special{fp}%
\special{pa 1520 894}%
\special{pa 1600 926}%
\special{fp}%
}}%
%
{\color[named]{Black}{%
\special{pn 8}%
\special{pa 2720 894}%
\special{pa 2800 854}%
\special{fp}%
\special{pa 2720 894}%
\special{pa 2800 926}%
\special{fp}%
}}%
%
{\color[named]{Black}{%
\special{pn 8}%
\special{pa 400 1214}%
\special{pa 1040 1214}%
\special{fp}%
}}%
%
{\color[named]{Black}{%
\special{pn 8}%
\special{pa 1200 1214}%
\special{pa 1840 1214}%
\special{fp}%
}}%
%
{\color[named]{Black}{%
\special{pn 8}%
\special{pa 2320 1214}%
\special{pa 3120 1214}%
\special{fp}%
}}%
%
{\color[named]{Black}{%
\special{pn 8}%
\special{pa 400 1374}%
\special{pa 1040 1374}%
\special{fp}%
}}%
%
{\color[named]{Black}{%
\special{pn 8}%
\special{pa 1200 1374}%
\special{pa 1840 1374}%
\special{fp}%
}}%
%
{\color[named]{Black}{%
\special{pn 8}%
\special{pa 2320 1374}%
\special{pa 3120 1374}%
\special{fp}%
}}%
%
{\color[named]{Black}{%
\special{pn 8}%
\special{pa 720 1054}%
\special{pa 800 1014}%
\special{fp}%
\special{pa 720 1054}%
\special{pa 800 1086}%
\special{fp}%
}}%
%
{\color[named]{Black}{%
\special{pn 8}%
\special{pa 720 1214}%
\special{pa 640 1174}%
\special{fp}%
\special{pa 720 1214}%
\special{pa 640 1246}%
\special{fp}%
}}%
%
{\color[named]{Black}{%
\special{pn 8}%
\special{pa 720 1374}%
\special{pa 800 1334}%
\special{fp}%
\special{pa 720 1374}%
\special{pa 800 1406}%
\special{fp}%
}}%
%
{\color[named]{Black}{%
\special{pn 8}%
\special{pa 1520 1374}%
\special{pa 1440 1334}%
\special{fp}%
\special{pa 1520 1374}%
\special{pa 1440 1406}%
\special{fp}%
}}%
%
{\color[named]{Black}{%
\special{pn 8}%
\special{pa 1520 1214}%
\special{pa 1440 1174}%
\special{fp}%
\special{pa 1520 1214}%
\special{pa 1440 1246}%
\special{fp}%
}}%
%
{\color[named]{Black}{%
\special{pn 8}%
\special{pa 1520 1054}%
\special{pa 1600 1014}%
\special{fp}%
\special{pa 1520 1054}%
\special{pa 1600 1086}%
\special{fp}%
}}%
%
{\color[named]{Black}{%
\special{pn 8}%
\special{pa 2720 1054}%
\special{pa 2640 1014}%
\special{fp}%
\special{pa 2720 1054}%
\special{pa 2640 1086}%
\special{fp}%
}}%
%
{\color[named]{Black}{%
\special{pn 8}%
\special{pa 2720 1214}%
\special{pa 2640 1174}%
\special{fp}%
\special{pa 2720 1214}%
\special{pa 2640 1246}%
\special{fp}%
}}%
%
{\color[named]{Black}{%
\special{pn 8}%
\special{pa 2720 1374}%
\special{pa 2800 1334}%
\special{fp}%
\special{pa 2720 1374}%
\special{pa 2800 1406}%
\special{fp}%
}}%
%
{\color[named]{Black}{%
\special{pn 4}%
\special{sh 1}%
\special{ar 400 894 20 20 0  6.28318530717959E+0000}%
}}%
%
{\color[named]{Black}{%
\special{pn 8}%
\special{pa 400 1054}%
\special{pa 1040 1054}%
\special{fp}%
}}%
%
{\color[named]{Black}{%
\special{pn 4}%
\special{sh 1}%
\special{ar 400 1054 20 20 0  6.28318530717959E+0000}%
\special{sh 1}%
\special{ar 400 1214 20 20 0  6.28318530717959E+0000}%
\special{sh 1}%
\special{ar 400 1374 20 20 0  6.28318530717959E+0000}%
\special{sh 1}%
\special{ar 1040 1374 20 20 0  6.28318530717959E+0000}%
\special{sh 1}%
\special{ar 1040 1214 20 20 0  6.28318530717959E+0000}%
\special{sh 1}%
\special{ar 1040 1054 20 20 0  6.28318530717959E+0000}%
\special{sh 1}%
\special{ar 1040 894 20 20 0  6.28318530717959E+0000}%
\special{sh 1}%
\special{ar 1200 894 20 20 0  6.28318530717959E+0000}%
\special{sh 1}%
\special{ar 1200 1054 20 20 0  6.28318530717959E+0000}%
\special{sh 1}%
\special{ar 1200 1214 20 20 0  6.28318530717959E+0000}%
\special{sh 1}%
\special{ar 1200 1374 20 20 0  6.28318530717959E+0000}%
\special{sh 1}%
\special{ar 1840 1374 20 20 0  6.28318530717959E+0000}%
\special{sh 1}%
\special{ar 1840 1214 20 20 0  6.28318530717959E+0000}%
\special{sh 1}%
\special{ar 1840 1054 20 20 0  6.28318530717959E+0000}%
\special{sh 1}%
\special{ar 1840 894 20 20 0  6.28318530717959E+0000}%
\special{sh 1}%
\special{ar 2320 894 20 20 0  6.28318530717959E+0000}%
\special{sh 1}%
\special{ar 2320 1054 20 20 0  6.28318530717959E+0000}%
\special{sh 1}%
\special{ar 2320 1214 20 20 0  6.28318530717959E+0000}%
\special{sh 1}%
\special{ar 2320 1374 20 20 0  6.28318530717959E+0000}%
\special{sh 1}%
\special{ar 3120 1374 20 20 0  6.28318530717959E+0000}%
\special{sh 1}%
\special{ar 3120 1214 20 20 0  6.28318530717959E+0000}%
\special{sh 1}%
\special{ar 3120 1054 20 20 0  6.28318530717959E+0000}%
\special{sh 1}%
\special{ar 3120 894 20 20 0  6.28318530717959E+0000}%
}}%
\end{picture}%
\end{center}
\caption{The label of the chord $\overline{p}\overline{q}$}
\end{figure}

The structure of graded Lie algebras $\mathcal{LC}$ and $\mathcal{C}$ are not fully understood.
The Lie algebra $\mathcal{LC}$ has the trivial center and its homology $H_*(\mathcal{LC})$
is the same as the homology of the circle $S^1$. However, the homology of the Lie subalgebra
$\mathcal{LC}^1:=\bigoplus_{m=2}^{\infty} \mathcal{LC}_m$ is highly non-trivial, and
so is the homology of $\mathcal{C}$.
In \cite{KK2}, the center of $\mathcal{C}$ was computed.
For an integer $m\ge 2$ let $C_m=\{ (1,2),(3,4),\ldots, (2m-1,2m)\}$ and
set $\Omega_m=N(C_m)\in \mathcal{C}_m$. Note that $a(\Omega_m)=N(\omega^{\otimes m})$.

\begin{theorem}[\cite{KK2}]
\label{thm:centcho}
The center of the Lie algebra $\mathcal{C}$ is spanned by $\Omega_m$, $m\ge 2$:
$$Z(\mathcal{C})=\bigoplus_{m=2}^{\infty} \mathbb{Q}\Omega_m.$$
\end{theorem}

\subsection{The center of the Goldman Lie algebra}
\label{subsec:center}

Let $\mathfrak{g}$ be a Lie algebra. The center of $\mathfrak{g}$, denoted by $Z(\mathfrak{g})$,
is the set of $X\in \mathfrak{g}$ such that $[X,Y]=0$ for any $Y\in \mathfrak{g}$.
It is a fundamental problem to compute the center of $\mathfrak{g}$.

Let $S$ be an oriented surface. It is clear from the definition of the Goldman bracket
that if $\xi$ is a loop parallel to a boundary component of $S$, $\xi$ and
its powers $\xi^n$, $n\in \mathbb{Z}$, are in the center $Z(\mathbb{Q}\hat{\pi}(S))$.
The question is whether these elements span $Z(\mathbb{Q}\hat{\pi}(S))$.
Goldman gave a partial result in this direction.

\begin{theorem}[Goldman \cite{Go86}, Theorem 5.17]
\label{thm:disjoint}
Let $\alpha,\beta\in \hat{\pi}(S)$ and assume that $\alpha$ is represented by a simple closed curve.
Then $[\alpha,\beta]=0$ in $\mathbb{Q}\hat{\pi}(S)$ if and only if $\alpha$ and $\beta$ are
freely homotopic to disjoint curves.
\end{theorem}

For example, we see that  if $S$ is compact then $\hat{\pi}(S)\cap Z(\mathbb{Q}\hat{\pi}(S))$
is the set of loops parallel to some boundary component of $S$ and their powers.
To see this, we take a system of simple closed curves that fills $S$.
This means that each component of the complement of these curves is a disk or
an annulus whose boundary contains some boundary component of $S$.
If $\beta\in \hat{\pi}(S)\cap Z(\mathbb{Q}\hat{\pi}(S))$, by Theorem \ref{thm:disjoint}
one can assume that $\beta$ is disjoint from each member of the filling curves.
Therefore, $\beta$ is homotopic to a point or a power of a boundary loop.

Whether the set $\hat{\pi}(S)\cap Z(\mathbb{Q}\hat{\pi}(S))$ spans the
center $Z(\mathbb{Q}\hat{\pi}(S))$ or not is an open question.
If $S$ is closed, this is affirmative. The following result was conjectured by Chas and Sullivan.

\begin{theorem}[Etingof \cite{E06}]
\label{thm:Eti}
If $S$ is closed, the center $Z(\mathbb{Q}\hat{\pi}(S))$ is spanned by the constant loop $1\in \hat{\pi}(S)$.
\end{theorem}

The proof of Etingof uses symplectic geometry of the moduli space of flat
$GL_{N}(\mathbb{C})$-bundles over the surface $S$.

As a bi-product of the proof of Therem \ref{thm:centcho},
we obtain a partial result on the center $Z(\mathbb{Q}\hat{\pi}(\Sigma_{g,1}))$
The idea is to use the relation between $\mathbb{Q}\hat{\pi}(\Sigma_{g,1})$
and $\mathfrak{a}_g^-$ in Theorem \ref{thm:Ntheta} and the
fact that any element of $Z(\mathfrak{a}_g^-)$ must be an $Sp$-invariant tensor
since it commutes with the degree two part $N(H^{\otimes 2})\cong \mathfrak{sp}(H)$.
Let $\zeta\in \pi_1(\Sigma_{g,1})$ be the boundary loop as in \S \ref{sec:Revi}.

\begin{theorem}[\cite{KK2}]
\label{thm:parcent}
For a positive integer $g$, set $m(g):=[(g-1)/4]+1$.
Here $[x]$ is the greatest integer less than or equal to $x$.
For any $u\in Z(\mathbb{Q}\hat{\pi}(\Sigma_{g,1}))$, there exists a
polynomial $f(\zeta)\in \mathbb{Q}[\zeta]$ such that
$$u\equiv |f(\zeta)| \quad ({\rm mod\ } \mathbb{Q}\hat{\pi}(\Sigma_{g,1})(2m(g))).$$
\end{theorem}

Let $\Sigma_{\infty,1}$ be the inductive limit of the embeddings
$\Sigma_{g,1}\hookrightarrow \Sigma_{g+1,1}$, $g>0$,
obtained by gluing $\Sigma_{1,2}$ on $\Sigma_{g,1}$ along the boundary.
Based on Theorem \ref{thm:parcent}, we can determine the center of
$Z(\mathbb{Q}\hat{\pi}(\Sigma_{\infty,1}))$.

\begin{theorem}[\cite{KK2}]
The center $Z(\mathbb{Q}\hat{\pi}(\Sigma_{\infty,1}))$ is spanned by the constant loop
$1\in \hat{\pi}(\Sigma_{\infty,1})$.
\end{theorem}

\subsection{Homological Goldman Lie algebra}
\label{subsec:HGLA}

As was mentioned in \S \ref{subsec:GTL}, the Goldman bracket comes 
from the Poisson bracket 
of two trace functions on the moduli space of flat $G$-bundles over the surface,
$\Hom(\pi_1(S), G)/G$. Goldman \cite{Go86} \S3 already 
showed the explicit formula for the Poisson bracket depends heavily on 
the choice of a Lie group $G$. This fact led him to introducing some 
variants of the (original) Goldman Lie algebra.
Later Andersen, Mattes and Reshetikhin 
\cite{AMR} unified diversity of the Poisson structures 
into the Poisson algebra of chord diagrams on a surface. 
It would be very interesting if some phenomena analogous to what was stated in 
\S \ref{sec:Ope} and \S \ref{subsec:alg-Gold} could be found for this Poisson algebra. \par
In this subsection we discuss a variant which appears for an abelian $G$,
and some relation to the first and the second homology groups 
of the original one. Results on the second homology group are due to 
Toda \cite{Toda2}. 
If $G$ is abelian, the Poisson action of $\mathbb{Z}{\hat\pi}(S)$ on 
$\Hom(\pi_1(S),G)/G$ factors through the group ring of the 
integral homology group $\HZ= H_1(S; \mathbb{Z})$, $\mathbb{Z}\HZ$, 
which we call {\it the homological Goldman Lie algebra} of the surface 
$S$. We denote by $[X] \in \mathbb{Z}\HZ$ the basis element corresponding 
to $X \in \HZ$. Then the Lie bracket on $\mathbb{Z}\HZ$ is given by 
\begin{equation}
[[X], [Y]] = (X\cdot Y)[X+Y] \in \mathbb{Z}\HZ
\label{31bracket}
\end{equation}
for any $X, Y \in \HZ$, cf. \cite{Go86} \S5.10. 
Here $(X\cdot Y)$ is the intersection number of $X$ and $Y$. 
More generally, if $H$ is an additive group with a
bi-additive alternating pairing $(\ \cdot \ )\colon 
H \times H \to \mathbb{Z}$, then the formula (\ref{31bracket}) defines 
a structure of a Lie algebra on the group ring $\mathbb{Z}H$,
which we also call the {\it homological Goldman Lie algebra} associated 
to the alternating pairing $(\ \cdot \ )$. 
We remark that the pairing $(\ \cdot \ )$ is not necessarily non-degenerate.
In the last part of this subsection, 
we will present an outline of Toda's works \cite{Toda1}
\cite{Toda2} on the algebraic structure of 
the homological Goldman Lie algebra in this general setting. 
\par

In \S \ref{subsec:alg-Gold} we constructed a Lie algebra homomorphism of
$\mathbb{Q}{\hat\pi}(\Sigma_{g,1})$ into the Lie algebra of 
symplectic derivations. A similar homomorphism for $\mathbb{R}\HZ$ 
was already given in \cite{Go86} \S5.10. 
In the first half of this subsection, until Corollary \ref{32dim},
we suppose $S = \Sigma_{g,1}$, $g \geq 1$. Note that $H_1(\Sigma_{g,1};
\mathbb{Z}) = H_1(\Sigma_{g};\mathbb{Z}) \cong \mathbb{Z}^{2g}$, and that 
the intersection pairing is non-degenerate. 
Here we give a slightly modified version of Goldman's homomorphism. 
The $2g$-dimensional torus $T^{2g} := H_1(\Sigma_{g,1}; \mathbb{R}/\mathbb{Z})$
has a natural symplectic form $\omega \in \Omega^2(T^{2g})$ derived from the
intersection form on the surface $\Sigma_{g,1}$. Hence the Poisson bracket 
makes $C^\infty(T^{2g}) = C^{\infty}(T^{2g}; \mathbb{C})$ a complex Lie algebra. 
We define a linear map $\rho\colon \mathbb{C}\HZ \to C^{\infty}(T^{2g})$ by 
$$
\rho([X])(Z) :=-\frac{1}{4\pi^2}e^{2\pi\sqrt{-1}(X\cdot Z)}
$$
for any $X \in \HZ$ and $Z \in T^{2g}$. It is easy to check that $\rho$ is a Lie 
algebra homomorphism. \par
On the other hand, for any closed $2g$-dimensional symplectic manifold $(M,
\omega)$ the linear map $\varphi^M\colon C^{\infty}(M) \to \mathbb{C}$ given by 
$$
\varphi^M(f) :=\int_Mf\omega^g \in \mathbb{C}
$$
induces a linear map on the abelianization $C^{\infty}(M)^\abel$ of the Poisson
Lie algebra $C^\infty(M)$. In fact, we have $\varphi^M(\{f, h\}) =
\int_MH_f(h)\omega^g = \int_M\mathcal{L}_{H_f}(h\omega^g) = 0$ for any $f$ and $h
\in C^\infty(M)$. Here $H_f \in \mathcal{X}(M)$ is the Hamiltonian vector field
associated to $f$. In our situation, we have 
\begin{equation}
(\varphi^{T^{2g}}\circ\rho)([X]) = 
\begin{cases}
0, &\mbox{if $X\neq 0$,}\\
-\displaystyle\frac{g!}{4\pi^2}, &\mbox{if $X = 0$.}
\end{cases}\label{31phi}
\end{equation}
This induces a non-trivial element of the first cohomology group  of
the Lie algebra $\mathbb{C}\HZ$, $H^1(\mathbb{C}\HZ)$.

From (\ref{31phi}) we have 
\begin{equation*}
\frac{16\pi^2}{g!}\int_{T^{2g}}\rho([X])\rho([Y])\omega^g = 
\begin{cases}
0, &\mbox{if $X\neq Y$,}\\
1, &\mbox{if $X = Y$.}
\end{cases}
\end{equation*}
Hence the homomorphism $\rho\colon \mathbb{C}\HZ \to C^\infty(T^{2g})$ is injective. 
We can use $\rho$ to compute the center of $\mathbb{C}\HZ$
in a similar way to that in \S \ref{subsec:center}.

\begin{proposition} 
\label{hcenter}
The center of $\mathbb{C}\HZ$, $Z(\mathbb{C}\HZ)$, is spanned
by $[0]$. 
\end{proposition}
\begin{proof}
Let $\{A_i, B_i\}^g_{i=1} \subset \HZ$ be a symplectic basis, and 
$(x_i, y_i)$ the global coordinates of $T^{2g}$ corresponding to the basis. 
We have $\omega = \sum^g_{i=1} dx_i\wedge dy_i$, $\rho([A_i]) =
-\frac{1}{4\pi^2}e^{2\pi\sqrt{-1}y_i}$ and $\rho([B_i]) =
-\frac{1}{4\pi^2}e^{-2\pi\sqrt{-1}x_i}$. Suppose $u \in Z(\mathbb{C}\HZ)$. 
Then $0 = \rho([[A_i], u]) = -\frac{1}{4\pi^2}\{e^{2\pi\sqrt{-1}y_i}, \rho(u)\}
= \frac{1}{4\pi^2}\left(\frac{\partial}{\partial
y_i}e^{2\pi\sqrt{-1}y_i}\right)\left(\frac{\partial}{\partial x_i}\rho(u)\right)$, 
and so $\frac{\partial}{\partial x_i}\rho(u) = 0$. Similarly
$\frac{\partial}{\partial y_i}\rho(u) = 0$. Hence $\rho(u) \in C^\infty(T^{2g})$
is a constant function $\in \mathbb{C} = \mathbb{C}\rho([0])$. Since $\rho$ is
injective, we obtain $u \in \mathbb{C}[0]$. 
Clearly we have $[0] \in Z(\mathbb{C}\HZ)$. 
This proves the proposition. 
\end{proof}
As will be stated in Theorem \ref{icl} , Toda \cite{Toda1} classifies the ideals of 
the homological Goldman Lie algebra over the rationals  $\mathbb{Q}$
in the most general setting. 
This proposition follows also from his result. \par

Next we discuss the abelianization, i.e., the first homology group 
of the Goldman Lie 
algebra $\mathbb{Z}{\hat\pi}(\Sigma_{g,1})$. 
We begin by computing the abelianization of the homological Goldman Lie 
algebra $\mathbb{Z}\HZ$. Let $\{A_i, B_i\}_{i=1}^g \subset \HZ$ be a 
symplectic basis. We define $\nu(X) \in \mathbb{Z}_{>0}$ for $X \in \HZ
\setminus\{0\}$ by $\nu(X) = {\rm g.c.d.}\{a_i, b_i; 1 \leq i \leq g\}$ 
where $X = \sum a_iA_i+b_iB_i$. In other words, $\nu(X)^{-1}X$ is in $\HZ$, 
and primitive. We define $\nu(0) := 0$ for $0 \in \HZ$. 

\begin{lemma}\label{32commutator}
$$
[\mathbb{Z}\HZ, \mathbb{Z}\HZ] 
= \bigoplus_{X \in \HZ\setminus\{0\}}\mathbb{Z}\nu(X)[X].
$$
\end{lemma}
\begin{proof} We have $(X\cdot Y) = ((X+Y)\cdot Y)$ for any $X$ and $Y \in \HZ$. 
Hence $[[X], [Y]] \in \mathbb{Z}\nu(X+Y)[X+Y]$. Conversely, for any $X \in
\HZ\setminus\{0\}$, there exists $Y \in \HZ$ such that $(X\cdot Y) = \nu(X)$. 
Then we have $[[X-Y], [Y]] = \nu(X)[X]$. This porves the lemma. 
\end{proof}

\begin{corollary}\label{32abel}
$$
{\mathbb{Z}\HZ}^\abel = \bigoplus_{X\in \HZ}(\mathbb{Z}/\nu(X)).
$$
In particular, the Lie algebra $\mathbb{Z}\HZ$ is not finitely generated, 
while ${\mathbb{Q}\HZ}^\abel = \mathbb{Q}$. 
\end{corollary}

The result ${\mathbb{Q}\HZ}^\abel = \mathbb{Q}$ follows also from 
Toda's Theorem \ref{icl} \cite{Toda1}. 
Furthermore we have 

\begin{theorem}[\cite{KKT}]
There is a subset $S$ of $\HZ$ such that $\{[X];  X \in S\}$ 
generates $\mathbb{Q}\HZ$ as a Lie algebra and $\sharp S = 2g + 2$. 
In particular, the Lie algebra $\mathbb{Q}\HZ$ is finitely generated. Moreover, if S is a subset of H and $\{[X];  X \in S\}$ generates $\mathbb{Q}\HZ$ as a Lie algebra, we have $\sharp S \geq 2g + 2$.
\end{theorem}

Since we have a natural surjection of Lie algebras $\mathbb{Z}{\hat\pi}(\Sigma_{g,1}) \to
\mathbb{Z}\HZ$, we obtain the following from Corollary \ref{32abel}
\begin{corollary}
The Goldman Lie algebra $\mathbb{Z}{\hat\pi}(\Sigma_{g,1})$ is not finitely generated.
\end{corollary}

The following question arises from this Corollary.
\begin{question}
Is the abelianization $\mathbb{Q}{\hat\pi}(\Sigma_{g,1})^\abel$ finite
dimensional, or not?
Furthermore, is the {\rm (}rational{\rm )} Goldman Lie algebra $\mathbb{Q}{\hat\pi}(\Sigma_{g,1})$ finitely
generated, or not?  
\end{question}
This question is open. 
As was explained in \S\ref{subsec:alg-Gold}, we have a Lie algebra homomorphism of 
$\mathbb{Q}{\hat\pi}(\Sigma_{g,1})$ into the Lie algebra of symplectic derivations
$\mathfrak{a}_g^-$.  Recently Morita, Sakasai and Suzuki \cite{MSS} proved that a {\it
stable part} of  the abelianization of the Lie algebra $\mathfrak{a}_g^-$
is finite  dimensional. But the Lie algebra homomorphism does not fit to the
homomorphisms 
$\varphi_{\{1\}}$ and $\varphi_{[\pi,\pi]}$ stated below. They do {\it not} induce
maps on $\mathfrak{a}_g^-$.
On the other hand, from Toda's classification of the ideals, Theorem \ref{icl}
\cite{Toda1}, 
 ${\mathbb{Q}H_1(\Sigma_{g,r}; \mathbb{Z})}^\abel$ 
is infinite-dimensional if $r \geq 2$. 
\par

The abelianization ${\mathbb{Q}\HZ}^\abel$ is spanned by 
$[0]$. The map $\varphi_{[\pi,\pi]}\colon \mathbb{Q}\HZ \to \mathbb{Q}$ defined by 
\begin{equation}
\varphi_{[\pi,\pi]}([X]) = 
\begin{cases}
0, &\mbox{if $X\neq 0$,}\\
1, &\mbox{if $X = 0$,}
\end{cases}\label{32phi}
\end{equation}
is proportional to the map $\varphi^{T^{2g}}\circ\rho$ in (\ref{31phi}). 
Hence it induces an isomorphism $\varphi_{[\pi,\pi]}\colon \mathbb{Q}\HZ^\abel
\overset\cong\to\mathbb{Q}$. 
Since we have a natural surjection of Lie algebras $\mathbb{Q}{\hat\pi}(\Sigma_{g,1})\to 
\mathbb{Q}\HZ$, the map $\varphi_{[\pi,\pi]}\colon \mathbb{Q}{\hat\pi}(\Sigma_{g,1})^\abel
\to\mathbb{Q}\HZ^\abel \to \mathbb{Q}$ is nontrivial. 
Moreover, as will be shown later in Corollary \ref{bGo}, 
the map $\varphi_{\{1\}}\colon\mathbb{Q}{\hat\pi}(\Sigma_{g,1})
\to \mathbb{Q}$ defined in (\ref{pGamma}) descends to 
$\mathbb{Q}{\hat\pi}(\Sigma_{g,1})^\abel$. Since 
$\varphi_{\{1\}}$ and $\varphi_{[\pi,\pi]}$ are linearly independent, 
we obtain 
\begin{corollary} \label{32dim}
$$
\dim_\mathbb{Q}\mathbb{Q}{\hat\pi}(\Sigma_{g,1})^\abel \geq 2.
$$
\end{corollary}

\par
Now let $S$ be any connected oriented surface. 
Choose a base point $* \in S$, 
and denote $\pi := \pi_1(S, *)$. Consider a normal subgroup $\Gamma \subset \pi$.
We denote by $\widehat{N} = \widehat{N_\Gamma}$ the set of conjugacy classes in
the quotient $N = N_\Gamma := \pi/\Gamma$. We have a natural surjection 
$\varpi = \varpi_\Gamma\colon \mathbb{Q}{\hat\pi}(S) \to \mathbb{Q}\widehat{N}$. \par
Now the following question seems to be natural.
\begin{question}
Does the Goldman bracket descend to $\mathbb{Q}\widehat{N_\Gamma}$? Or,
equivalently,  is the subspace ${\rm Ker}(\varpi_\Gamma)$ an ideal of
$\mathbb{Q}{\hat\pi}(S)$?
\end{question}
Clearly the answer is yes if $\Gamma = \{1\}$ or $\Gamma =
[\pi,\pi]$. 
Remark that $N_{[\pi,\pi]} = \widehat{N_{[\pi,\pi]}} = \HZ$. 
The answer is no if $S = \Sigma_{g,1}$, $g \geq 2$ and $\Gamma = [\pi,[\pi,\pi]]$. 
To see this choose a symplectic generator system $\{\alpha_i, \beta_i\}^g_{i=1} \subset \pi$.
By a straightforward computation, we obtain 
$$
[\vert\alpha_1\vert, 
\vert[[\beta_1,\alpha_2], {\alpha_2}^{-1}]{\alpha_1}^{-1}\vert
- \vert{\alpha_1}^{-1}\vert] 
= \vert(1-[\alpha_2, \alpha_1])(1-[\alpha_1, \alpha_2])\vert 
\in\mathbb{Q}\widehat{N_{[\pi,[\pi,\pi]]}},
$$
which is not zero. This implies that the Goldman bracket does {\it not} descend 
to $\mathbb{Q}\widehat{N_{[\pi,[\pi,\pi]]}}$. \par

The question is open for the other normal subgroups. It is closely related to
the former question. We introduce a map $\varphi_\Gamma\colon \mathbb{Q}{\hat\pi}(S) \to 
\mathbb{Q}$ by 
\begin{equation}
\label{pGamma}
\varphi_{\Gamma}(\vert x\vert) = 
\begin{cases}
0, &\mbox{if $x \in \pi\setminus\Gamma$,}\\
1, &\mbox{if $x \in \Gamma$.}
\end{cases}
\end{equation}
It is well-defined since $\Gamma$ is a normal subgroup of $\pi$. 
\begin{lemma}
\label{lem:des}
If the Goldman bracket descends to $\mathbb{Q}\widehat{N_\Gamma}$, 
then 
$$
\varphi_\Gamma([\mathbb{Q}{\hat\pi}(S), \mathbb{Q}{\hat\pi}(S)]) = 0.
$$
In other words, $\varphi_\Gamma$ induces a nonzero element of 
$H^1(\mathbb{Q}{\hat\pi}(S))$. 
\end{lemma}
\begin{proof} Assume $\varphi_\Gamma([\alpha, \beta]) \neq 0$ for 
some $\alpha$ and $\beta \in \hat{\pi}(S)$, from which we will deduce 
a contradiction. We may assume $\alpha$ is in $\pi_1(S, *)$, and 
$\alpha\coprod\beta$ a generic immersion. Then we have 
$\varphi_\Gamma(\vert\alpha_p\beta_p\vert) = 1$ for some $p_0 \in 
\alpha\cap\beta$. If we write $\gamma := \alpha_{*p_0}\alpha_{p_0}\beta_{p_0}
{\alpha_{*{p_0}}}^{-1} = \alpha\alpha_{*{p_0}}\beta_{p_0}{\alpha_{*{p_0}}}^{-1}
\in \pi_1(S, *)$, then $\gamma \in \Gamma$ and
$\beta=\vert\gamma\alpha^{-1}\vert$. Since the Goldman bracket 
descends to $\widehat{N_\Gamma}$, we have
$$
0 \neq \varphi_\Gamma([\alpha, \beta]) = \varphi_\Gamma([\varpi_\Gamma\alpha,
\varpi_\Gamma\beta]) = \varphi_\Gamma([\varpi_\Gamma\alpha,
\varpi_\Gamma\alpha^{-1}]) = \varphi_\Gamma([\alpha, \alpha^{-1}]).
$$  
On the other hand, 
let $\alpha^{-1}$ be represented by a generic immersion such that $\alpha\cup
\alpha^{-1}$ bounds a narrow annulus, as in \cite{Go86}, p.295. Let $p$ be a
double  point of the loop $\alpha$. It divides the loop $\alpha$ into 
two based loops $\alpha_1$ and $\alpha_2$ with base point $p$
as in Figure 11. The two intersection points derived from $p$ 
contributes $\alpha_1\alpha_2{\alpha_1}^{-1}{\alpha_2}^{-1}$ 
and $\alpha_2\alpha_1{\alpha_2}^{-1}{\alpha_1}^{-1}$, 
respectively, with the opposite sign. Then 
$\alpha_1\alpha_2{\alpha_1}^{-1}{\alpha_2}^{-1} \in \Gamma$ 
is equivalent to 
$\alpha_2\alpha_1{\alpha_2}^{-1}{\alpha_1}^{-1} \in \Gamma$.
This implies that the contributions of the two points cancel, namely,
$\varphi_\Gamma([\alpha, \alpha^{-1}]) = 0$. 
This contradicts what we proved above, and proves the lemma. 
\end{proof}

\begin{figure}
\begin{center}
\label{fig:alpha12}
\input{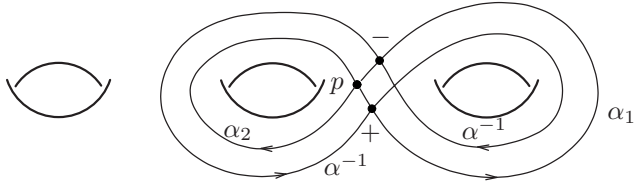}
\caption{loops $\alpha_1$ and $\alpha_2$}
\end{center}
\end{figure}

In the case $\Gamma = \{1\}$, we have
\begin{corollary}[\cite{Go86} Proposition 5.9]
\label{bGo}
If we write $\hat\pi'(S) := \hat\pi(S) \setminus\{1\}$, then 
$$
[\mathbb{Q}\hat\pi(S), \mathbb{Q}\hat\pi(S)] \subset \mathbb{Q}\hat\pi'(S).
$$
\end{corollary}
Goldman's original proof \cite{Go86} pp.294--294 is not correct. For details, 
see \cite{KK1} Remark 3.1.2.\par

\begin{corollary}
\label{hGamma}
If $S = \Sigma_g$ or $\Sigma_{g,1}$, $g \geq 1$, then the
Goldman bracket does not descend to $\mathbb{Q}\widehat{N_\Gamma}$ for any 
normal subgroup $\Gamma$ with $[\pi, \pi] \subsetneqq \Gamma \subsetneqq \pi$.
\end{corollary}
\begin{proof}
This follows from $\mathbb{Q}\HZ^\abel\cong \mathbb{Q}$, Lemma \ref{32abel}. 
\end{proof}

\par
We conclude this chapter by giving an outline of Toda's works \cite{Toda1}
\cite{Toda2} on the algebraic structure of 
the rational homological Goldman Lie algebra for any additive group 
$H$ equipped with a bi-additive  alternating pairing 
$(\ \cdot \ )\colon H \times H \to \mathbb{Z}$. The pairing induces the map 
$\mu\colon H \to \Hom(H, \mathbb{Z})$ defined by $\mu(x)(y) = (x\cdot y)$ 
for any $x$ and $y \in H$. In Toda's results, the set ${\rm Ker}(\mu)$ and 
its complement subset $H \setminus{\rm Ker}(\mu)$ play some important 
roles. \par
Toda classified all the ideals of the rational homological Goldman Lie algebra 
$\mathbb{Q}H$ as follows. For any $x \in H$, we define 
$T(X)\colon \mathbb{Q}H \to \mathbb{Q}H$ by $T(X)([Y]) := [X+Y]$ for any $Y \in H$. 
\begin{theorem}[\cite{Toda1}]
\label{icl}
For any ideal $\mathfrak{h}$ in $\mathbb{Q}H$, 
there exists a unique pair $(V_0, V)$ such that\par
{\rm (1)} $V_0$ and $V$ are subspaces of the linear span of ${\rm Ker}(\mu)$,\par
{\rm (2)} For any $Z \in {\rm Ker}(\mu)$ we have $T(Z)(V) \subset V$, and \par
{\rm (3)} $\mathfrak{h} = V_0 \oplus\sum_{X \in H \setminus{\rm Ker}(\mu)} T(X)(V)$.\par
If $\mu = 0$, we define $V = 0$.
Conversely, a subset $\mathfrak{h} \subset \mathbb{Q}H$ 
satisfying the conditions {\rm (1)},{\rm (2)} and {\rm (3)} is an ideal of $\mathbb{Q}H$.
\end{theorem}
As a corollary, the center of $\mathbb{Q}H$ equals the $\mathbb{Q}$-linear span 
of ${\rm Ker}(\mu)$, and it is isomorphic to the abelianization of 
the Lie algebra $\mathbb{Q}H$. 
\begin{corollary}[\cite{DZ}]
If $(\ \cdot\ )$ is non-degenerate, then 
any ideal of $\mathbb{Q}H$ equals one of the followings
$$
0, \,\,\ \mathbb{Q}[0], \,\,
\mathbb{Q}H \,\,\text{and}\,\, \mathbb{Q}(H\setminus\{0\}.
$$
\end{corollary}
This corollary was already obtained by Dokovi\'c and Zhao \cite{DZ}. 
Moreover they asserted that their results covered the degenerate cases.
But it was based on the non-correct claim on p.154, l.-10
that the quotient $\mathbb{Q}H/\mathbb{Q}{\rm Ker}(\mu)$ would 
be isomorphic to the rational homological Goldman Lie algebra 
associated to the quotient $H/{\rm Ker}(\mu)$. \par
In \cite{Toda2} he computed the second homology group of $\mathbb{Q}H$ 
in the general setting. Let $\mathbb{Q}H^{(1)}$ be the derived ideal of
$\mathbb{Q}H$, which equals $\mathbb{Q}(H\setminus{\rm Ker}(\mu))$
by Theorem \ref{icl}. 
\begin{theorem}[\cite{Toda2} Theorem 1] 
If the pairing $(\ \cdot \ )$ is non-zero, then 
we have natural isomorphisms
\begin{align}
& H_2(\mathbb{Q}H) \cong ({\wedge}^2\mathbb{Q}{\rm Ker}(\mu))\oplus
H_2(\mathbb{Q}H^{(1)}), \label{dcp}\\
& H_2(\mathbb{Q}H^{(1)}) \cong \bigoplus_{z\in {\rm Ker}(\mu)}
\mathbb{Q}\otimes (H/\mathbb{Z}z).\nonumber
\end{align}
\end{theorem}

\begin{corollary}[\cite{DZ}] 
If the pairing $(\ \cdot \ )$ is non-degenerate, then we have
$H_2(\mathbb{Q}H) = H\otimes\mathbb{Q}$. 
\end{corollary}

In \cite{Toda2} Theorem 13, Toda proved that the third homology group 
of $\mathbb{Q}H$ is non-trivial if the pairing is non-zero. \par
Finally we go back to the (original) Goldman Lie algebra. 
Let $S$ be a compact connected oriented surface. Then Toda \cite{Toda2} 
proved
\begin{theorem}[\cite{Toda2} Theorem 11] The composite of the map induced by the natural surjection $\mathbb{Q}\hat\pi(S) \to \mathbb{Q}H_1(S; \mathbb{Z})$
and the projection in the decomposion {\rm (\ref{dcp})}
$$
H_2(\mathbb{Q}\hat\pi(S)) \to H_2(\mathbb{Q}H_1(S; \mathbb{Z})) 
\to H_2(\mathbb{Q}H_1(S; \mathbb{Z})^{(1)})
$$
is surjective.
\end{theorem}
To prove the theorem, he constructed some explicit abelian $2$-cycles 
in the Goldman Lie algebra $\mathbb{Q}\hat\pi(S)$.

\end{document}